\documentclass[reqno,11pt]{amsart}

\usepackage[english]{babel}
 
\usepackage{amsmath, latexsym, amsfonts, amssymb, amsthm, amscd}
\usepackage{mathrsfs}

\usepackage[lofdepth,lotdepth]{subfig}
\usepackage{graphicx}
\usepackage{caption}
\newcommand\fnote[1]{\captionsetup{font=footnotesize}\caption*{#1}}

\usepackage{url}
\usepackage{hyperref}
\usepackage{enumerate}
\usepackage[yyyymmdd,hhmmss]{datetime}
\usepackage{color}
\definecolor{midgray}{gray}{0.5}
\usepackage{stmaryrd} % le double  crochet gauche est \llbracket, le droit \rrbracket.
\usepackage{enumitem}

\numberwithin{equation}{section}

\theoremstyle{remark} 
\newtheorem*{rem}{Remark} 
\theoremstyle{plain} 
\newtheorem{thm}{Theorem}[section] 
\newtheorem{cor}[thm]{Corollary}
\newtheorem{prop}[thm]{Proposition}
\newtheorem{lem}[thm]{Lemma} 
\theoremstyle{definition} 
\newtheorem{deff}[thm]{Definition}

\newtheorem{hyp}[thm]{Hypothesis}
\newtheorem{hyps}[thm]{Hypothesis}
\newtheorem{remark}[thm]{Remark}

\newtheorem{example}[thm]{Example}

\begin{document}

\title[Multivariate Hawkes processes on inhomogeneous random graphs]
{Multivariate Hawkes processes on inhomogeneous random graphs}

\author{ Zo\'e \textsc{Agathe-Nerine}}
\address{ Universit\'e de  Paris, Sorbonne Paris Cité, Laboratoire MAP5, UMR CNRS 8145 \& FP2M, CNRS FR 2036. E-mail: \href{mailto: zoe.agathe-nerine@u-paris.fr}{zoe.agathe-nerine@u-paris.fr}}

\date{}

\maketitle

\begin{abstract} 
We consider a population of $N$ interacting neurons, represented by a multivariate Hawkes process: the firing rate of each neuron depends on the history of the connected neurons. Contrary to the mean-field framework where the interaction occurs on the complete graph, the connectivity between particles is given by a random possibly diluted and inhomogeneous graph where the probability of presence of each edge depends on the spatial position of its vertices. We address the well-posedness of this system and Law of Large Numbers results as $N\to\infty$. A crucial issue will be to understand how spatial inhomogeneity influences the large time behavior of the system.
\end{abstract}


%----------------------------------------------------------------------------------------
%----------------------------------------------------------------------------------------
%----------------------------------------------------------------------------------------

\section{Introduction}

\subsection{Biological and mathematical context}

Neurons are cells specialised in the reception, integration and transfer of information in the brain. A propagating electrical signal is transmitted from a neuron to the others in terms of all-or-none emission of action potential also called \emph{spike} which is a stereotyped phenomenon. More precisely, neurons possess a permeable membrane which allows ion exchanges. Without stimulus, the difference of respective ion concentrations induces a voltage gradient called resting potential. This potential evolves depending on the information received from other neurons: a presynaptic neuron emitting a spike leads to the release of neurotransmitters, and induces a change in the ions distribution around the membrane of post-synaptic neurons. If the stimulus reaches a sufficient threshold, the neuron generates an action potential, the \emph{synaptic integration}.\\

The progress of monitoring methods as MRI (Magnetic Resonance Imaging) and ECG (Electrocardiography) since the 50's led to a better understanding of the physiology of a neuron.  As a result, the implementation of mathematical models started with the Hodgkin-Huxley model \cite{Hodgkin1952} (in 1952) describing the evolution of the membrane potential in terms of a system of four ODEs, further simplified in two equations by FitzHugh \cite{FitzHugh1961} and Nagumo \cite{Nagumo1962} (in 1962). 

Stochasticity is intrinsic to the neuronal activity: noise in neuronal systems may come from different sources. To name a few, randomness accounts for the emergence of spontaneous spikes \cite{fatt1952spontaneous}, failed propagation \cite{smith1980mechanisms}, and the stochastic opening and closing of the ion channels (the probability of the channel being open or closed depends on the membrane potential). Stochasticity is also present at the scale of a whole population in the large variability  of synaptic connections between neurons. From a mathematical perspective, this naturally led to diffusion models: mean-field Hodgkin-Huxley and FitzHugh-Nagumo's models in \cite{Baladron2012}, mean-field Piecewise Deterministic Markov Processes (PDMP) in \cite{DeMasi2014,Cormier2020}. Another popular model is the integrate-and-fire dynamics, first introduced in the seminal work of Lapicque \cite{lapique1907}, and still studied mathematically, as e.g. in \cite{Delarue2015}.

The previous type of modeling of the membrane potential typically leads to non-linear Fokker–Planck equations whose large time behavior is often hard to determine analytically. A usual approach in this context (that we follow here) yields more tractable and explicitely solvable models: as spikes are stereotyped, all the information is coded in the duration of time between the spikes. Hence we model the activity of a neuron by a point process where each point represents the time of a spike. In this context, the framework of \emph{Hawkes processes} is particularly relevant since it can account for the dependence of the activity of a neuron on the past of the whole population: the spike of one neuron can trigger others spikes. Hawkes processes have been first introduced in \cite{HAWKES1971} in 1971 to model earthquakes, and have been thoroughly studied since (with applications for instance to seismology  \cite{Ogata1988}). It is not possible to quote the vast mathematical literature on Hawkes  processes since the seminal works of \cite{HAWKES1971, Hawkes1974, bremaud1996stability}, we refer nonetheless to \cite{delattre2016,Hodara2017,Chevallier2017} and references therein.

In this paper, the main issue we concentrate on is the structure of interaction between neurons. There is indeed experimental evidences that neurons are spatially organized \cite{Bosking1997, Mountcastle1997}. The first approach, where this spatial structure is missing, assumes a complete graph of interaction (mean-field framework). Mean-field analysis goes back to \cite{McKean,Sznitman1989}, originally for diffusion models as in \cite{Baladron2012}. The literature on mean-field analysis is huge and does not restrict to neuroscience applications (see the following references as far as neurosciences are concerned: integrate and fire models \cite{Delarue2015}, PDMP \cite{DeMasi2014,Cormier2020}). As for mean-field Hawkes processes, similar models have been considered in \cite{delattre2016, Heesen2021, Hodara2017} and expanded with additional features (age dependence in  \cite{Chevallier2017, Raad2020}, inhibition in \cite{Costa2020, duval:hal-03233710, raad2020stability}). What makes the mean-field analysis for Hawkes processes particularly tractable is that the large population limit is given in terms of an inhomogeneous Poisson process whose intensity solves a convolution equation \cite{delattre2016}. 

The spatial organization in the brain has been originally analysed mathematically from a phenomenological perspective: we may refer to the celebrated neural field equation  \cite{Wilson1972,Amari1977,Bressloff2011}, which has given a macroscopic description of excitable units with non-local interaction. Several works have extended the mean-field framework to take into account the presence of a macroscopic spatial structure in the interaction (originally for diffusion models \cite{Touboul2014, Luon2016, Budhiraja2016}, as well as for Hawkes processes \cite{Ditlevsen2017,CHEVALLIER20191}). More specifically, \cite{CHEVALLIER20191} has given a mesoscopic interpretation of the neural field equation in terms of the limit of spatially extended Hawkes processes interacting through a mesoscopic spatial kernel.

The main contribution of this paper is to go further and provide a microscopic interpretation of this spatial structure in terms of random graphs. We assume that the interaction between neurons is given by a possibly inhomogeneous and diluted graph, where  the probability of presence of an edge depends on the positions of its vertices. The main example that we have in mind concerns the class of $W$-random graph (see \cite{diaconis2007graph,lovasz2012large, janson2011graphons, Borgs2018, Borgs2019}), that includes homogeneous Erd\"os R\'enyi graphs. The only previous works so far on particle systems with similar interaction address the case of diffusions. Law of Large Numbers (LLN) and Large Deviations results on homogeneous Erd\"os R\'enyi graphs have been considered in \cite{DelattreGL2016, Coppini2019, Oliveira2019} and further extended to the inhomogenous case in \cite{bet2020weakly, Luon2020, bayraktar2021graphon, Medvedev2013} on a bounded time interval. The behavior of such systems on a time scale no longer bounded (but may depend on the size of the population) is more difficult, and remains largely open so far (in this direction, see \cite{coppini2019long}). The present work is, to the best of our knowledge, the first paper to address similar issues to Hawkes processes. We address here quenched LLN results on bounded time interval and large time asymptotics of the limiting process. The behavior of the system on unbounded time scale is a working progress. Note also that all the existing works consider graphs with interaction of diverging degrees.  The case with sparse interaction (see \cite{Oliveira2020,lacker2021local} for diffusions) remains open for Hawkes processes and will be the object of future works.

\subsection{Our model}

The aim of this paper is to describe the behavior in large population and large time of a network of particles interacting on a spatially structured random graph. Let $N$ be the size of the population, consider the multivariate Hawkes process $\left( Z_1^{(N)}(t), \cdots, Z_N^{(N)}(t) \right)_{t>0}$: for $i=1\cdots N$, the $i$th neuron is located on $x_i \in I$ where $I\subset\mathbb{R}^d$ represents the spatial domain of the neuron (suppose e.g. that $I=[0,1]$ or $I=\mathbb{R}^d$), $Z_i^{(N)}(t)$ counts the number of spikes during the time interval $[0,t]$. Its intensity at time $t$ conditioned on the past $[0, t)$ is given by
\begin{equation}\label{eq:def_lambdaiN_intro}
\lambda_i^{(N)}(t)= f\left(u_0(t,x_i)+\dfrac{1}{N}\sum_{j=1}^N w_{ij}^{(N)}\int_0^{t-} h(t-s) dZ_j^{(N)}(s)\right).
\end{equation}
Here, $f ~:~  \mathbb{R}  \longrightarrow \mathbb{R}_+$ represents the synaptic integration, $u_0~:~  \mathbb{R}_+ \times I \longrightarrow \mathbb{R}$ a spontaneous activity of the neuron, $h~:~  \mathbb{R}_+ \longrightarrow \mathbb{R}$ a memory function which models how a past jump of the system affects the present intensity. The novelty here is  $w_{ij}^{(N)}$, representing the random inhomogeneous interaction between the neurons $i$ and $j$ that depends on their positions $x_i$ and $x_j$. We refer to Section \ref{S:system} for precise definitions. 

We study the behavior of the process $\left( Z_1^{(N)}(t), \cdots, Z_N^{(N)}(t) \right)_{t>0}$ as $N\to \infty$ and $t\to\infty$. The large population convergence is considered for a fixed realization of the graph (quenched model). Its limit is described in terms of an inhomogeneous Poisson process whose intensity involves the macroscopic spatial structure of the graph. A second aspect of the present work of independent interest will be to analyse the long time dynamics of the macroscopic process. We generalise the phase transition already observed for mean-field linear Hawkes processes \cite{delattre2016}. An important issue will be to understand how the inhomogeneity of the graph influences the long time dynamics. This will be illustrated by different examples and simulations.

\subsection{Organisation of the paper}

After introducing some notation, we start in Section \ref{S:system} by defining formally the process of interest \eqref{eq:def_ZiN}. The well-posedness of such process is treated by Proposition \ref{prop:exis_H_N}. We study the large population behavior of the process $\left( Z_1^{(N)}(t), \cdots, Z_N^{(N)}(t) \right)_{t>0}$ in Section \ref{S:sys-lim}. We show, under suitable hypotheses on the parameters, that the behavior of a neuron located in $x\in I$ within an infinite population is described by an intensity $\lambda(\cdot,x)$ solving
\begin{equation}\label{eq:def_lambdabarre}
\lambda(t,x)= f\left(u_0(t,x)+\int_{I} W(x,y)\int_0^t h(t-s) \lambda(s,y)ds~ \nu(dy)\right).
\end{equation}
Here, $ W ~:~ I \times I \longrightarrow \mathbb{R}_+$ is seen as the limit interaction kernel, and $\nu$, probability measure on $I$ describes  the macroscopic distribution of the positions. Well-posedness and regularity of \eqref{eq:def_lambdabarre} is considered in Theorem \ref{thm:existence_lambda}. In Section \ref{S:conv}, we study the behavior of the process  \eqref{eq:def_ZiN} in large population (Theorems  \ref{thm:cvg_0} and \ref{thm:cvg-sup_0}). The behavior of the empirical measure and respectively the spatial profile (Definition \ref{def:spatial_prof}) is analysed in Section \ref{SS:conv_mes_emp} (resp. Section \ref{S:spatial_prof}). In Section \ref{S:sys-lim_temps}, we study the behavior of \eqref{eq:def_lambdabarre} as $t\to\infty$ in the linear case, that is when $f(x)=x$. We extend the phase transition observed without spatial structure in \cite{delattre2016} to a general interaction kernel $W$. Finally in Section \ref{S:application}, we apply our results to concrete cases and present some simulations.  The proofs are gathered in the remaining Sections.

\subsection*{Acknowledgements}

This is a part of my PhD thesis. I would like to thank my PhD supervisors Eric \textsc{Lu\c con} and Ellen \textsc{Saada} for introducing this subject, for their useful advices and for their encouragement. This research has been conducted within the FP2M federation (CNRS FR 2036), and is supported by ANR-19-CE40-0024 (CHAllenges in MAthematical NEuroscience). I would like to warmly thank  two anonymous Referees for their useful and precise comments, which considerably helped to improve the paper.

\section{A system of \texorpdfstring{$N$}{N} interacting particles on a graph and its limit}\label{S:system}

\subsection{Notation}

For $n\in \mathbb{N}$, we write $\Vert \cdot \Vert$ for the usual Euclidian norm in $\mathbb{R}^n$, $\Vert \left(x_1,\cdots,x_n\right) \Vert = \left( \vert x_1 \vert ^2 + \cdots + \vert x_n \vert ^2 \right) ^\frac{1}{2}$. For $(E,\mathcal{A},\mu)$ a measured space, for a function $h$ in $L^p(E,\mu)$ with $p\geq 1$, we write $\Vert h \Vert_{E,\mu,p}:=\left( \int_E \vert h \vert^p d\mu \right)^\frac{1}{p}$ . When $p=2$, we write as $<f,g>_{E,\mu}=  \int_E fg d\mu$ the scalar product. Without ambiguity, we may omit the subscript $(E,\mu)$ or $\mu$. For instance, for $T>0$ and $h$ in $L^p([0,T])$, we write $\Vert h \Vert_{[0,T],p}:=\left( \int_0^T\vert h(t) \vert^p dt\right)^\frac{1}{p}$. When we omit the notation $[0,T]$, the integration is on $\mathbb{R}_+$. For a  real-valued bounded function $g$ on a space $E$,  we write $\Vert g \Vert _\infty := \Vert g \Vert _{E,\infty}=\sup_{x\in E} \vert g(x) \vert$. If $d$ is a distance on $E$, we denote by $ \Vert f \Vert_L = \sup_{x\neq y} \vert f(x) - f(y) \vert / d(x,y)$ the Lipchitz seminorm of a real-valued function $f$ on $E$. We also denote by $\Vert f \Vert_{BL}:= \Vert f \Vert_L + \Vert f \Vert_{E,\infty}$ the bounded Lipschitz norm of $f$. For $\mu$ and $\nu$ measures on $E$, we define
\begin{equation}\label{eq:def_dBL1}
d_{BL}(\mu,\nu):= \sup_{g, \Vert g \Vert_{BL}\leq 1} \left| \int_E g\left(d\mu - d\nu\right)\right|.
\end{equation}
We denote by $\mathbb{D}\left([0,T],\mathbb{N}\right)$ the space of c\`adl\`ag (right continuous with left limits) functions defined on $[0,T]$ and taking values in $\mathbb{N}$. For any integer $N\geq 1$, we denote by $\llbracket 1, N \rrbracket$ the set $\left\{1,\cdots,N\right\}$. For any distribution $\nu$, $X\sim\nu$ means that the random variable $X$ has distribution $\nu$. We denote by $\mathcal{U}(0,1)$ the uniform distribution on $[0,1]$, and for any $p\in [0,1]$, $\mathcal{B}(p)$ denotes the Bernoulli distribution with parameter $p$. %: $X\sim \mathcal{B}(p)$ means that $X=1$ with probability $p$ and $X=0$ with probability $1-p$.

\subsection{The model} 

\subsubsection{Definitions} 
The graph of interaction for \eqref{eq:def_lambdaiN_intro} is constructed as follows:
\begin{deff}\label{def:espace_proba_bb}
On a common probability space $\left(\widetilde{\Omega}, \widetilde{\mathcal{F}},\mathbb{P}\right)$, we consider a sequence $\left(\left(x_i^{(N)}\right)_{i \in \llbracket 1,N \rrbracket}\right)_{N\geq 1}$ of (possibly random) positions and a family of random variables $\xi^{(N)}=\left( \xi^{(N)}_{ij}\right)_{N\geq 1, i,j \in \llbracket 1,N \rrbracket}$ on $\widetilde{\Omega}$ such that under $\mathbb{P}$, for any $N\geq 1$ and  $i,j \in \llbracket 1,N \rrbracket$, conditioned on the positions $\left(x_1^{(N)},\ldots,x_N^{(N)}\right)$, $\xi^{(N)}$ is a collection of mutually independent Bernoulli random variables such that for $1\leq i,j \leq N$, $\xi_{ij}^{(N)}$ has parameter $W_N(x_i,x_j)$.  We assume that the particles in \eqref{eq:def_lambdaiN_intro} are connected according to the oriented graph $\mathcal{G}^{(N)}= \left( \left\{1,\cdots,N\right\} , \xi^{(N)}\right)$. For any $i$ and $j$, $\xi^{(N)}_{ij}=1$ encodes for the presence of the edge $j\to i$ and $\xi^{(N)}_{ij}=0$ for its absence.
\end{deff}
It is possible to construct via a coupling this graph simultaneously for all $N$: consider an infinite sequence of fixed positions in $I$ $\left(x_1,\ldots, x_N, \ldots\right)$ (that is, for each $N\geq 1, ~x_i^{(N)}=x_i$) and i.i.d. random variables $\left(U_{i,j}\right)_{i,j\in \mathbb{N}}\sim \mathcal{ U}[0, 1]$. Define $\xi_{ij}^{(N)}=\mathbf{1}_{\left\{ U_{i,j} \leq W_N(x_i,x_j)\right\}}$: conditioned on the positions $\left(x_1,\ldots,x_N\right)$, $\xi^{(N)}$ is a collection of independent variables and $\xi_{ij}^{(N)} \sim \mathcal{B}\left( W_N(x_i,x_j) \right)$. We now fix these sequences, and work on a filtered probability space $\left(\Omega,\mathcal{F},\left(\mathcal{F}_t\right)_{t\geq 0},\mathbf{P}\right)$ rich enough for all the following processes can be defined. We denote by $\mathbf{E}$ the expectation under $\mathbf{P}$ and $\mathbb{E}$ the expectation w.r.t. $ \mathbb{ P}$. In the following definitions, $N$ is fixed and we denote by $\underline{x}^{(N)}=\left(x_1^{(N)},\ldots,x_N^{(N)}\right)$ the vector of positions.
\begin{deff}\label{def:H2} Let $\left(\pi_i(ds,dz)\right)_{1\leq i \leq N}$ be a sequence of i.i.d. Poisson random measures on $\mathbb{R}_+\times \mathbb{R}_+$ with intensity measure $dsdz$.
A $\left(\mathcal{F}_t\right)$-adapted multivariate counting process  $\left(Z_1^{(N)}\left(t\right),...,Z_N^{(N)}\left(t\right)\right)_{t\geq 0}$ defined on $\left(\Omega,\mathcal{F},\left(\mathcal{F}_t\right)_{t\geq 0},\mathbf{P}\right)$ is called \emph{a multivariate Hawkes process} with the set of parameters
$\left(N,f,\xi^{(N)},W_N,u_0,h,\underline{x}^{(N)}\right)$ if $\mathbf{P}$-almost surely, for all $t\geq 0$ and $i \in \llbracket 1, N \rrbracket$:
\begin{equation}\label{eq:def_ZiN}
Z_i^{(N)}(t) = \int_0^t \int_0^\infty \mathbf{1}_{\{z\leq \lambda_i^{(N)}(s)\}} \pi_i(ds,dz)
\end{equation}
with $\lambda_i^{(N)}(t)$ defined by
\begin{equation}\label{eq:def_lambdaiN}
\lambda_i^{(N)}(t)= f\left(u_0(t,x_i^{(N)})+\dfrac{\kappa_i^{(N)}}{N}\sum_{j=1}^N \xi_{ij}^{(N)}\int_0^{t-} h(t-s) dZ_j^{(N)}(s)\right).
\end{equation}
\end{deff}
We denote by $\kappa_i^{(N)}\geq 0$ a dilution parameter which may depend on $\underline{x}^{(N)}$, and $\xi^{(N)}$. The idea behind this dilution parameter is that $\kappa_i^{(N)}\simeq \dfrac{N}{\mathbb{E}[deg_N(i)]}$ (where $deg_N (i)=\sum_{j=1}^N \xi_{ij}^{N}$ is the indegree of the particle $i$, that is, the number of edges incident to it), so that the interaction term remains of order 1 as $N\to\infty$. This means that the interaction in \eqref{eq:def_lambdaiN_intro} is fixed as $w_{ij}^{(N)}=\kappa_i^{(N)}\xi_{ij}^{(N)}$. 
\begin{remark}\label{rem:equivalence_def}
By Proposition 3 of \cite{delattre2016}, the process $\left(Z_1^{(N)}\left(t\right),\ldots,Z_N^{(N)}\left(t\right)\right)_{t\geq 0}$ defined by \eqref{eq:def_ZiN} is such that $\mathbf{P}$-almost surely, $Z_i^{(N)}$ and $Z_j^{(N)}$ do not jump simultaneously for all $i\neq j $, and for all $i \in \llbracket 1, N \rrbracket$, the compensator of $Z_i^{(N)}(t)$ is $\int_0^t \lambda_i^{(N)} (s) ds$ (see \cite{jacod2013limit} about compensators of increasing processes).
\end{remark}

\subsubsection{Existence}

We first provide well-posedness results of $\left(Z_1^{(N)},\ldots,Z_N^{(N)}\right)$ given by \eqref{eq:def_ZiN}. We require the following assumptions:
\begin{hyp}\label{hyp:existence_Zin} We suppose that $f$ is Lipschitz continuous with Lipschitz constant $L_f \geq 0$, and that either $f$ is nonnegative or that $f(x)=x$ with $u_{ 0}\geq 0$ and $h\geq 0$ (linear case). We also  suppose that $h$ is locally square integrable on $[0,+\infty)$, that $(t,x) \mapsto u_0(t,x)$ is continuous in $t$ and Lipschitz continuous in $x$ (uniformly in $t$) with Lipschitz constant $L_{u_0} \geq 0$. Moreover $u_{0}$ is supposed bounded uniformly in $(t,x)$ i.e, $\Vert u_0 \Vert_\infty<\infty$.% : there exists $\Vert u_0 \Vert_\infty \geq 0$ such that for all $t\geq 0$ and $x\in\mathbb{R}^d$, $\vert u_0(t,x) \vert \leq  \Vert u_0 \Vert_\infty$.
\end{hyp}
\begin{prop}\label{prop:exis_H_N} Under Hypothesis \ref{hyp:existence_Zin}, for a fixed realisation of the family $\left(\pi_i\right)_{1\leq i \leq N}$,  there exists a pathwise unique multivariate Hawkes process (in the sense of Definition \ref{def:H2}) such that for any $T<\infty$, $\sup_{t\in [0,T]} \sup_{1\leq i \leq N} \mathbf{E}[Z_i^{(N)}(t)] <\infty$.
\end{prop}
The proof of Proposition \ref{prop:exis_H_N} will be given in Section \ref{S:proof_prop:exis_H_N}. 

\subsection{Large population limit process}\label{S:sys-lim}

We want to study the behavior of the process defined in Definition \ref{def:H2} when $N \to \infty$ on bounded time interval. After some heuristics, we show the well-posedness of the limit of the system \ref{eq:def_ZiN}.
 
\subsubsection{Heuristics}

In this paragraph, we motivate the proper limit for the particle system \eqref{eq:def_ZiN} as $N\to\infty$. A minimal requirement is that the empirical distribution of the positions $\nu^{(N)}:= \frac{1}{N}\sum_{i=1}^N \delta_{x_i}$ has itself a macroscopic limit $\nu$. We will consider below different scenarios under which such LLN holds. Concerning the macroscopic behavior of the graph, another minimal requirement is that in a way to define later on, the graph $\mathcal{G}^{(N)}$ given in Definition \ref{def:H2} converges towards a macroscopic interaction kernel $ W ~:~ I \times I \longrightarrow \mathbb{R}_+$. We refer to Section \ref{S:hyp_cvg} for more precise statements. Then, as $N \to \infty$, an informal LLN argument shows that the empirical mean in \eqref{eq:def_lambdaiN} becomes an expectation w.r.t both the candidate limit for $Z_i^{(N)}$ and w.r.t the macroscopic law $\nu$ of the positions: we can replace the sum in \eqref{eq:def_lambdaiN_intro} by the integral in \eqref{eq:def_lambdabarre}, the microscopic interaction term $w_{ij}$ in \eqref{eq:def_lambdaiN_intro} by the macroscopic term $W(x,y)$ in \eqref{eq:def_lambdabarre} (where $y$ describes the macroscopic distribution of the positions), and the past activity of the neuron $dZ_j^{(N)}(s)$ by its intensity in large population. Hence, the macroscopic description of a neuron at position $x\in I$ should be described in terms of its intensity $\lambda(t,x)$ solving \eqref{eq:def_lambdabarre}. This heuristics gives a limit process at position $x$ defined as an inhomogeneous Poisson point process with deterministic intensity $\lambda(\cdot,x)$ satisfying \eqref{eq:def_lambdabarre}. 

\subsubsection{Well-posedness of the macroscopic limit}

We propose a framework under which  \eqref{eq:def_lambdabarre} is well-posed, with more hypotheses on the regularity of $(f,u_0,W)$.
\begin{hyp}\label{hyp:existence_lambda_barre} Assume that the macroscopic indegree at position $x$ defined by 
\begin{equation}\label{eq:def_D(x)}
D(x)=\int_I W(x,y)\nu(dy)
\end{equation}
has a H\"older regularity and is uniformly bounded on $I$: there exist $C_w>0$ and $\vartheta \in ]0,1]$ such that
\begin{align}
&\int_I \vert W(x,y)-W(x',y) \vert \nu(dy) \leq C_w \Vert x - x'\Vert^\vartheta,\ x,x^{ \prime}\in I \quad \text{ and} \label{eq:hyp_W_pseudolip_theta}\\
&\sup_{x\in I} D(x) =: C_W^{(1)} < \infty.\label{eq:hyp_w_int_nu_y}
\end{align}
\end{hyp}
\begin{thm}\label{thm:existence_lambda}
Let $T>0$. Under Hypotheses \ref{hyp:existence_Zin} and \ref{hyp:existence_lambda_barre}, there exists a unique solution  $\lambda$ to \eqref{eq:def_lambdabarre} that is continuous and bounded on $\left[0,T\right]\times I$ and this solution is nonnegative. Moreover, there exists  $C_\lambda>0$ depending on $(f,u_0,W,h,\nu,T)$ such that for all $(t,x,z)\in [0,T]\times I \times I$,
\begin{equation} \label{eq:pt_fixe_lips_espace}
\vert \lambda(t,x) - \lambda(t,z) \vert \leq C_\lambda \left( \Vert x-z \Vert + \Vert x-z \Vert^\vartheta\right) =: C_\lambda \phi \left(\Vert x-z \Vert\right) .
\end{equation}
In the linear case $f(x)=x$, $u_{ 0}, h\geq0$, if $u_0$ is continuously differentiable in time and $\dfrac{\partial u_0}{\partial t}$ is bounded on $[0,T]\times I$, $h$ is continuous and piecewise continuously differentiable, then $\lambda$ is differentiable in time and
\begin{equation} \label{eq:pt_fixe_derivee}
\dfrac{\partial \lambda}{\partial t}(t,x) = \dfrac{\partial u_0}{\partial t} (t,x) + h(t) \int_I W(x,y) \lambda(0,y) \nu(dy) + \int_I \int_0^t h(t-s) W(x,y)\dfrac{\partial \lambda}{\partial t} (s,y) \nu(dy) ds,
\end{equation}
and $\left| \dfrac{\partial \lambda}{\partial t}\right| $ is bounded on $[0,T]\times I$.
\end{thm}
Theorem \ref{thm:existence_lambda} will be proved in Section \ref{S:proof_thm:existence_lambda}. Note that Theorem \ref{thm:existence_lambda} provides the existence of a unique solution  $\lambda$ of \eqref{eq:def_lambdabarre} that is continuous on $\mathbb{R}_+\times I$ and locally bounded.
%----------------------------------------------------------------------------------------
%----------------------------------------------------------------------------------------
\section{Convergence of the model in large population}\label{S:conv}

\subsection{Coupling} 
From now on, $\lambda$ refers to the unique solution to \eqref{eq:def_lambdabarre}. To check that our heuristics about the large population behavior is correct, we introduce a suitable coupling between the process defined in \eqref{eq:def_ZiN} (at positions $x_i$) and a Poisson process with intensity $\lambda(\cdot, x_i)$ at the same position $x_i$.
\begin{deff}\label{def:couplageZZbarre}
For the family $\left(\pi_i\left(ds,dz\right)\right)_{1\leq i \leq N}$ of i.i.d. Poisson random measures on  $\mathbb{R}_+\times\mathbb{R}_+$ from Definition \ref{def:H2}, we construct for all $i$ in $\llbracket 1, N \rrbracket$:
\begin{equation}\label{eq:construction_lim_Zbarre_i}
 \overline{Z}_i(t) = \int_0^t \int_0^\infty \mathbf{1}_{\{z\leq \lambda(s,x_i)\}} \pi_i(ds,dz)
\end{equation}
with $\lambda$ satisfying \eqref{eq:def_lambdabarre}. Each process $\overline{Z}_i$ is an inhomogenous Poisson process with (deterministic) intensity $\lambda(\cdot,x_i)$, and as the family $\left(\pi_i\right)$ is independent, the processes $\left(\overline{Z}_i\right)_{i=1,\cdots,N}$ are also independent. 
\end{deff}

\subsection{Hypotheses}\label{S:hyp_cvg}

Regarding the behavior of the graph when $N\to \infty$, we use here the formalism of graph convergence developped in \cite{lovasz2012large} and introduce different norms on $I^2$. The key notion is to represent graphs in term of graphons, that are positive kernels defined on $I^2$. Note that we will not necessarily restrict ourselves to the symmetric case and bounded graphons.
\begin{deff}\label{def:norm_graph_mult}
Let $W$ be a $\mathbb{R}$-valued function defined on $I\times I$, where $I$ is endowed with some probability measure $\nu$. When the following terms are correctly defined, we write: 
\begin{align}
\Vert W \Vert_{\Box,\nu} :&= \sup_{S,T\subset I} \left| \int_{S\times T} W\left(x,y\right)\nu(dx)\nu(dy) \right|,\label{eq:cutnorm_graphon}\\
\Vert W \Vert_{\infty\to 1,\nu} :&= \sup_{\Vert g \Vert_{\infty}\leq 1} \int_I \left| \int_I W(x,y) g(y) \nu(dy) \right| \nu(dx),\label{eq:normeinf1_graphon}\\
\Vert W \Vert_{\infty\to \infty,\nu} :&= \sup_{\Vert g \Vert_{\infty}\leq 1} \sup_{x\in I} \left| \int_I W(x,y) g(y) \nu(dy) \right|\label{eq:normeinfinf_graphon}.
\end{align}
\end{deff}
These norms go back to the formalism of graph convergence introduced in \cite{lovasz2012large, diaconis2007graph} and further developed in \cite{Borgs2011,Borgs2018, Borgs2019} (and references therein). The last two norms can be seen as the norms of the linear operator $T_W:g\mapsto \left( x\longmapsto  \int_I W(x,y) g(y) \nu(dy) \right)$ when considering respectively $T_W : L^\infty(I,\nu) \to L^1(I,\nu)$ and $T_W : L^\infty(I,\nu) \to L^\infty(I,\nu)$. We also define the \emph{cut-distance} between two functions by
\begin{equation}\label{eq:def_cut_distance}
d_{\Box,\nu}\left(W_1,W_2\right)=\Vert W_1-W_2 \Vert_{\Box,\nu}.
\end{equation}
\begin{remark}\label{rem:equ_norm_graphon}
Lemma 8.11 of \cite{lovasz2012large} gives that $ \Vert \cdot \Vert_{\Box,\nu}$ and $ \Vert \cdot \Vert_{\infty\to 1,\nu}$ are equivalent: if $W$ is a function defined on $I^2$ with values in $\mathbb{R}$, then
\begin{equation}\label{eq:equ_norm_graphon_nu}
 \Vert W \Vert_{\Box,\nu} \leq \Vert W \Vert_{\infty\to 1,\nu} \leq 4 \Vert W \Vert_{\Box,\nu} .
\end{equation}
As $\Vert W \Vert_{1,\nu}:=\int_{I^2}\left| W(x,y)\right| \nu(dx)\nu(dy)$, we always have $\Vert W \Vert_{\Box,\nu}\leq \Vert W \Vert_{1,\nu}$.
\end{remark}
Usual representations of graphons consist in taking $I=[0,1]$ endowed with Lebesgue measure. We extend this definition to the general case where $\nu$ is a probability measure on $I$. To do this, we require the following assumption for the whole article.
\begin{hyp}\label{hyp:nu_AC/leb}
The probability measure $\nu$ is absolutely continuous w.r.t Lebesgue measure on $\mathbb{R}^d$.
\end{hyp}
\begin{lem}\label{lem:partition_I}
Under Hypothesis \ref{hyp:nu_AC/leb}, for any $N\geq 1$, there exists a partition  $\mathcal{P}_N:=\left(B_i^{(N)}\right)_{i=1,\cdots,N}$ of $I$ (and we use the notation $I=\bigsqcup_{i=1}^{N}B_i^{(N)}$) such that for all $i = 1, \cdots ,N$ , $\nu\left(B_i^{(N)}\right) = \dfrac{1}{N}$.
\end{lem}
Without ambiguity, we will forget the upper index $^{(N)}$ and only write $\left(B_1,\cdots,B_N\right)$.

 \begin{proof}
Denote by $\nu^{(1)}$ the first marginal of $\nu$ that is absolutely continuous w.r.t. Lebesgue measure on $\mathbb{R}$. Let $F_{ 1}$ be its continuous probability distribution function. Then for all $i=1, \ldots, N$, defining $B_i:= \left( F_{ 1}^{ -1} \left( \left( \frac{ i-1}{ N}, \frac{ i}{ N}\right]\right) \times \mathbb{R}^{d-1} \right)\cap I$ gives the result.
 \end{proof}

For every weighted graph $\mathcal{G}$ with weights $\left( g_{ij} \right)$, we associate a step-function $W^{\mathcal{G}}$ constructed, upon this partition, as follows (see e.g. \cite{lovasz2012large,Borgs2011}):
 \begin{equation}\label{eq:graphonG}
W^{\mathcal{G}}(u,v)= \sum_{i=1}^N \sum_{j=1}^N g_{ij} \mathbf{1}_{B_i}(u) \mathbf{1}_{B_j}(v), \quad \left(u,v\right)\in I^2.
\end{equation}

\begin{deff}\label{def:graphs_G1}\
We denote by $\mathcal{G}_N^{(1)}$ the directed weighted graph with vertices $\left\{ 1, \cdots,N\right\}$ such that every edge $j\rightarrow i$ is present, and with weight $\kappa_i^{(N)} W_N (x_i,x_j)$.
\end{deff}
Here $\mathcal{G}_N^{(1)}$ represents the average version of the graph $\mathcal{G}^{(N)}$ (where $\xi_{ij}\sim\mathcal{B}\left(W_N(x_i,x_j)\right)$ has been replaced by $\mathbb{E}\left(\xi_{ij}\right)$), renormalized by the dilution coefficient $\kappa_i^{(N)}$. 
A key argument of Theorems \ref{thm:cvg_0} and \ref{thm:cvg-sup_0} will be to show that $\mathcal{G}^{(N)}$ and $\mathcal{G}_N^{(1)}$ are close as $N\to\infty$ through concentration arguments that require the following uniformity assumptions on $W_N$.
\begin{hyps}\label{hyp:conv_graph_concentration} We suppose that there exist $\kappa_N\geq 1$ and $w_N\in]0,1]$ such that:
\begin{align}
&\max_{i\in \llbracket 1, N \rrbracket} \left( \kappa_i^{(N)} \right)\leq \kappa_N, \label{eq:kappa}\\
&\max_{i,j \in \llbracket 1, N \rrbracket} \left( W_N(x_i,x_j) \right)\leq w_N,\label{eq:w_n} \\
&\dfrac{1}{\kappa_N}\leq w_N \leq 1,\quad \text{and asymptotically: }\label{eq:kappaw_n1}\\
&\kappa_N^2 w_N \underset{N\to\infty}{=} o\left( \dfrac{N}{\log(N)}\right) \quad \text{and} \quad \dfrac{\kappa_N}{N}\xrightarrow[N\to\infty]{} 0.\label{eq:kappaw_nlog}
\end{align}
We also suppose that there exists $C_W>0$ independent of $N$ such that
\begin{equation}\label{eq:control_i}
\sup_{i\in\llbracket 1, N \rrbracket} \dfrac{1}{N} \sum_{j=1}^N \kappa_i^{(N)} W_N(x_i,x_j)\leq C_W.
\end{equation}
\end{hyps}
To illustrate the above conditions, think of the case where $W_N=\rho_N$ is a constant with $\rho_N \xrightarrow[N\to\infty]{}0$. This corresponds to a diluted Erd\"os-R\'enyi graph random graph. In this case, we can take $w_N=\rho_N$ and $\kappa_i^{(N)}=\kappa_N=1/\rho_N$. Then \eqref{eq:kappaw_nlog} boils down to $\rho_N \gg \log(N)/N$. Inequality \eqref{eq:control_i} is the microscopic counterpart of \eqref{eq:hyp_w_int_nu_y}: we require that the weighted indegrees of vertices in $\mathcal{G}_N^{(1)}$ are uniformly bounded.

\subsection{Convergence}

We study the proximity between the particle systems \eqref{eq:def_ZiN} and its macroscopic limit \eqref{eq:construction_lim_Zbarre_i}. We show two theorems that require different sets of hypotheses on the parameter functions, under two main scenarios.
\begin{deff}\label{def:scenarios}
We consider two different frameworks for the choices of the positions:
\begin{enumerate}
\item \textbf{Random spatial distribution: } For $(\widetilde{x_1},\widetilde{x_2},\cdots, \widetilde{x_N}, \cdots)$ a random sequence of i.i.d. variables distributed according to $\nu$ on $I$, we set for all $N\geq 1$ $\underline{x}_N=\left(x_1,\cdots, x_N\right)$ as the lexicographic ordering of the $N$ first positions $(\widetilde{x_1},\widetilde{x_2},\cdots, \widetilde{x_N})$. We assume that there exists some $\chi>5$ such that $\Vert W \Vert _ {L^\chi(I^2,\nu \otimes \nu)}<\infty$.
\item \textbf{Deterministic regular distribution of the positions:} For every $N\geq 1$ and $1\leq i \leq N$, we set $x_i^{(N)}=i/N$ and $I=[0,1]$ endowed with $\nu(dx)=dx$. We assume that $W$ is piecewise continuous on $[0,1]^2$.
\end{enumerate}
\end{deff} 
The assumption $ \chi>5$ of Scenario (1) is required in Proposition \ref{prop:toolbox_iid}, as a sufficient hypothesis for a Borel-Cantelli argument.

\subsubsection{First case: convergence in average}
\begin{hyp}\label{hyp:cvg_graph} We suppose that the annealed graph $\mathcal{G}_N^{(1)}$ converges to $W$ for the cut-distance:
\begin{align}
d_{\Box,\nu}\left( W^{\mathcal{G}_N^{(1)}}, W \right) \xrightarrow[N\to\infty]{} 0, \text{ as well as} \label{eq:hyp_cut_cvg}\\
\sup_{j\in\llbracket 1, N \rrbracket} \dfrac{1}{N} \sum_{i=1}^N \kappa_i^{(N)} W_N(x_i,x_j)\leq C_W. \label{eq:control_j}
\end{align}
\end{hyp}
Note that \eqref{eq:hyp_cut_cvg} implies that $ \displaystyle \lim_{N\to\infty}\Vert W^{\mathcal{G}_N^{(1)}} - W \Vert_{\infty\to 1,\nu}=0 $ (see Remark \ref{rem:equ_norm_graphon}). The hypothesis \eqref{eq:control_j} differs from \eqref{eq:control_i} in the sense that \eqref{eq:control_j} asks for a uniform bound on the outdegree (that is, the number of tail ends adjacent to a vertex) whereas \eqref{eq:control_i} relates to a uniform bound on the indegree.
\begin{thm}\label{thm:cvg_0}
Let $T>0$. Suppose that the sequence of positions $\left(\underline{x}_N\right)_N$ satisfies one of the scenarios of Definition \ref{def:scenarios}. Then, under the set of Hypotheses \ref{hyp:existence_Zin}, \ref{hyp:existence_lambda_barre}, \ref{hyp:conv_graph_concentration}, \ref{hyp:nu_AC/leb} and \ref{hyp:cvg_graph}, we have
\begin{equation}\label{eq:moy_ecart_limite_nulle}
\frac{1}{N} \sum_{i=1}^N \mathbf{E} \left[ \sup_{t\in [0,T]} \left| Z_i^{(N)}(t) - \overline{Z}_i(t) \right| \right] \xrightarrow[N\to \infty]{} 0
\end{equation}
for $\mathbb{P}$-almost realisations of the connectivity sequence $\left(\xi^{(N)}\right)_{N\geq 1}$ and positions $\left(\underline{x}_N\right)_{N\geq 1}$.
\end{thm}
The proof of Theorem \ref{thm:cvg_0} will be given in Section \ref{S:proof_thm:cvg_0}.

\subsubsection{Second case: convergence of the supremum}
Some graphs do not satisfy \eqref{eq:control_j}, see the examples of Section \ref{S:general_class_ex_div}. We propose here another result of convergence that does not require the control \eqref{eq:control_j}, but ask in return for a stronger convergence of the graphons.
\begin{hyp}\label{hyp:cvg_graph_infinf}
We suppose that
\begin{equation}\label{eq:hyp_infinf_cvg}
\Vert W^{\mathcal{G}_N^{(1)}}- W  \Vert_{\infty \to \infty, \nu}\xrightarrow[N\to\infty]{} 0 .
\end{equation}
\end{hyp}
\begin{thm}\label{thm:cvg-sup_0}
Let $T>0$. Suppose that the sequence of positions $\left(\underline{x}_N\right)_N$ satisfies one of the scenarios of Definition \ref{def:scenarios}. Consider the coupling introduced in Definition \ref{def:couplageZZbarre}. Then, under the set of Hypotheses \ref{hyp:existence_Zin}, \ref{hyp:existence_lambda_barre}, \ref{hyp:conv_graph_concentration}, \ref{hyp:nu_AC/leb} and  \ref{hyp:cvg_graph_infinf}, we have
\begin{equation}\label{eq:sup_ecart_limite_nulle}
\max_{1\leq i \leq N} \mathbf{E} \left[ \sup_{t\in [0,T]} \left| Z_i^{(N)}(t) - \overline{Z}_i(t) \right| \right] \xrightarrow[N\to \infty]{} 0.
\end{equation}
$\mathbb{P}$-almost surely.
\end{thm}
The proof of Theorem \ref{thm:cvg-sup_0} will be given in Section \ref{S:proof_thm:cvg-sup_0}. 
\begin{remark}\label{rem:vitesse_conv}
Theorems \ref{thm:cvg_0} and \ref{thm:cvg-sup_0} are quenched results, and do not provide any speed of convergence. In this case, the speed of convergence is unknown. Nevertheless, if we integrate also with respect to the graph (annealed case), one can obtain explicit speed of convergence as follows:
\begin{align*}
\max_{1\leq i \leq N} \mathbb{E}\mathbf{E} \left[ \sup_{t\in [0,T]} \left| Z_i^{(N)}(t) - \overline{Z}_i(t) \right| \right] &\leq C_T \dfrac{\kappa_N \sqrt{w_N}}{\sqrt{N}},\\
\frac{1}{N} \sum_{i=1}^N   \mathbb{E}\mathbf{E} \left[ \sup_{t\in [0,T]} \left| Z_i^{(N)}(t) - \overline{Z}_i(t) \right| \right] &\leq C_T \dfrac{\kappa_N \sqrt{w_N}}{\sqrt{N}}.
\end{align*}
Working in the annealed case simplifies considerably the proof (left to the reader), the previous estimates can be easily derived from the calculation done in the proofs of the previous theorems: we no longer have to deal with concentration estimates (see the term $A^{(N)}_{i,t,3}$ below in \eqref{eq:def_A3} which becomes a simple variance term).
\end{remark}

\subsection{Consequence on the empirical measure}\label{SS:conv_mes_emp}

A direct consequence of Theorems \ref{thm:cvg_0} and \ref{thm:cvg-sup_0} concerns the behavior as $N\to\infty$ of the empirical distribution on the space $S:=\mathbb{D}\left([0,T],\mathbb{N}\right)\times I$ of trajectories and positions.
\begin{deff}\label{def:mesure_process}
We define the following probability measures on $S$:
\begin{align}
\mu_N(d\eta,dx)&:= \dfrac{1}{N} \sum_{i=1}^N \delta_{\left(Z_i^{(N)}\left([0,T]\right),x_i^{(N)}\right)}(d\eta,dx), \text{ and}\label{eq:def_mesure_empirique_process}\\
\mu_\infty(d\eta,dx)&:= P_{[0,T],\infty}\left( d\eta \vert x\right) \nu(dx),\label{eq:def_mesure_limite_process}
\end{align}
where $P_{[0,T],\infty}\left( \cdot \vert x\right)$  is the law of an inhomogeneous Poisson point process with intensity $\left(\lambda(t,x)\right)_{0\leq t \leq T}$ (solution of \eqref{eq:def_lambdabarre}). Note that $\mu_N$ is random.
\end{deff}
\begin{thm}\label{thm:cvg_dBL}
Under the assumptions of Theorem \ref{thm:cvg_0} or Theorem \ref{thm:cvg-sup_0}, we have
\begin{equation}\label{eq:thm_cvg_dBL}
\mathbf{E} \left[ d_{BL}\left(\mu_N,\mu_\infty\right)\right] \xrightarrow[N\to\infty]{} 0.
\end{equation}
for $\mathbb{P}$-almost realisations of the connectivity sequence $\left( \xi^{(N)}\right)_{N\geq 1}$ and positions $\left(\underline{x}_N\right)_{N\geq 1}$ under scenarios of Definition \ref{def:scenarios}, where $d_{BL}$ is the bounded Lipschitz distance introduced in \eqref{eq:def_dBL1}.
\end{thm}
The proof of Theorem \ref{thm:cvg_dBL} will be given in Section \ref{S:proof_thm:cvg_dBL}. We can see this result as an extension of Theorems 1 and 2 of \cite{CHEVALLIER20191}, where the memory function is an exponential kernel and the interaction comes from a fixed interaction kernel that depends on the positions.

\subsection{Spatial profile}\label{S:spatial_prof}
Here we are under the conditions of Scenario (2) of Definition \ref{def:scenarios}, where $\lambda$ solves \eqref{eq:def_lambdabarre}. 
\begin{deff}\label{def:spatial_prof}
Define the random profile
\begin{align}
U_N(t,x)&:=\sum_{i=1}^N U_{i,N}(t) \mathbf{1}_{x\in\left(\frac{i-1}{N}, \frac{i}{N}\right]}, \text{ where} \label{eq:def_UN} \\
U_{i,N}(t)&:= u_0(t,x_i)+\dfrac{\kappa_i^{(N)}}{N}\sum_{j=1}^N \xi_{ij}^{(N)}\int_0^{t-} h(t-s) dZ_j^{(N)}(s),\label{eq:def_UiN}
\end{align}
and the deterministic profile
\begin{equation}\label{eq:def_u_lim}
u(t,x):=u_0(t,x)+\int_{I} W(x,y)\int_0^t h(t-s) \lambda(s,y)ds~ \nu(dy).
\end{equation}
We see from Theorem \ref{thm:existence_lambda} that $u$ is continuous and bounded.
\end{deff}
Note that $\lambda_i^{(N)}(t)=f\left( U_{i,N}(t)\right)$ and that $U_{i,N}(t)$ describes the accumulated activity of neuron $i$ up to time $t$. A similar quantity has already been considered in \cite{CHEVALLIER20191} for $h(t)=e^{-\alpha t}$ with a deterministic graph of interaction. In this case, with $u_0(t,x)=e^{-\alpha t} \widetilde{u}(x)$, \eqref{eq:def_u_lim} is the solution of the scalar neural field equation
$$\partial_t u(t,x) = -\alpha u(t,x) + \int_I W(x,y) f( u(t,y)) \nu(dy).$$ It has been extensively studied in the literature as an important example of macroscopic structured model with non local interaction (see \cite{Amari1977, Wilson1972,Bressloff2011}).
\begin{prop}\label{prop:cvg_profil_spatial} 
Under the Hypotheses of Theorem \ref{thm:cvg_0},%\ref{thm:cvg_0},
\begin{equation}\label{eq:cvg_profil_spatial}
\mathbf{E}\left[  \int_0^T  \int_0^1 \left| U_N(t,x) - u(t,x) \right| dx ~dt \right] \xrightarrow[N\to\infty]{} 0, 
 \end{equation}
for $\mathbb{P}$-almost realisations of the connectivity sequence $\left( \xi^{(N)}\right)_{N\geq 1}$ and positions $\left(\underline{x}_N\right)_{N\geq 1}$.
\end{prop}

The proof of Proposition \ref{prop:cvg_profil_spatial}  will be given in Section \ref{S:proof_prop:cvg_profil_spatial}.

%----------------------------------------------------------------------------------------
%----------------------------------------------------------------------------------------
%----------------------------------------------------------------------------------------

\section{Large time behavior of the limit process in the linear case}\label{S:sys-lim_temps}

We want to see how the limiting intensity \eqref{eq:def_lambdabarre} behaves as $t\to\infty$. We restrict here to the following \emph{linear case}, that is, when $f(x)=x$:
\begin{equation}\label{eq:def_lambdabarre_linear}
\lambda(t,x) = u_0(t,x) + \int_I W(x,y) \int_0^th(t-s)\lambda(s,y)ds~\nu(dy)
\end{equation}
on $\mathbb{R}_+\times I$. The case without spatial interaction, that is, $\lambda(t)=u_0 + \int_0^t h(t-s) \lambda(s) ds$ is standard and has been thoroughly studied in \cite[Th.~10 and~11]{delattre2016}. Depending on the value of $\Vert h \Vert_1$, there is a phase transition in the behavior of such $\lambda$ when $t\to\infty$: in the subcritical case ($\Vert h \Vert_1 <1$), $\lambda(t) \xrightarrow[t\to\infty]{} \dfrac{u_0}{1-\Vert h \Vert_1}$ and in the supercritical case ($\Vert h \Vert_1 >1$), $\lambda(t) \xrightarrow[t\to\infty]{} \infty$. The point of the present paragraph is to extend this result to the spatial case. We require the following assumptions:
\begin{hyp}\label{hyp:cas_lin} Suppose that we are in the linear case of Hypothesis \ref{hyp:existence_Zin}. In addition to Hypotheses \ref{hyp:existence_Zin} and \ref{hyp:existence_lambda_barre}, we suppose that $h$ is in $L^1(\mathbb{R}^+)$ and piecewise continuously differentiable. We also suppose that $u_0$ is continuously differentiable in time,  that there exists $C_{u_0}>0$ such that 
\begin{equation}\label{eq:borne_partial_u0}
\sup_{x\in I} \left\Vert \dfrac{\partial u_0}{\partial_t}(\cdot,x)\right\Vert_{1}=\sup_{x\in I}\int_{\mathbb{R}^+} \left\vert\dfrac{\partial u_0}{\partial s}(s,x)\right\vert ds\leq C_{u_0} <\infty.
\end{equation}
We also suppose that there exists $u$ Lipschitz continuous on $I$ such that $\displaystyle \lim_{N\to\infty}\sup_{x\in I} \left| u_0(t,x)-u(x)\right|=0$. Hence, when $u_0$ does not depend on time, we simply suppose $u_0=u$.
\end{hyp}
To describe the phase transition, we introduce the following linear operator
\begin{equation} \label{eq:def_operator_T}
  \begin{array}{rrcl}
T_W:&    L^\infty (I) & \longrightarrow &L^\infty (I) \\
   & g & \longmapsto & \left(T_Wg : x \longmapsto \int_I W(x,y) g(y) \nu(dy) \right).
  \end{array}
\end{equation}
The continuity of $T_W$ follows directly from \eqref{eq:hyp_w_int_nu_y}, and we have $\Vert T_W \Vert \leq C_W^{(1)}$. We denote by $r_\infty(T_W)$ the spectral radius of $T_W$: 
\begin{equation}\label{eq:def_spectral_radius}
r_\infty:=r_\infty(T_W) = \sup_{\sigma \in Sp(T_W)} \vert \sigma \vert = \lim_{n\to\infty} \Vert T_W^n \Vert ^\frac{1}{n}.
\end{equation}
The phase transition is given in terms of  $\Vert h \Vert_1 r_\infty <1 $ (subcritical) and  $\Vert h \Vert_1 r_\infty >1 $ (supercritical). The two cases are described separately below, after dealing with the usual exponential case.

\subsection{The exponential case} \label{SS:exp}

Previous works \cite{CHEVALLIER20191} have considered $h(t)=e^{-\alpha t}$ with $\alpha>0$ (hence $\Vert h \Vert_1 = 1/\alpha$). The term $\alpha$ is then called the \emph{leakage rate}. Note that in this case, the dynamics becomes Markovian \cite{dion2020exponential}. At the large population limit, the spatial profile seen in Section \ref{S:spatial_prof} is in this case linked to the scalar neural field equation \cite{CHEVALLIER20191}. In the exponential case, with the introduction of the operator $T_W$  we can give an explicit solution of \eqref{eq:def_lambdabarre_linear}.
\begin{prop} In the exponential case $h(t)=e^{-\alpha t}$, the solution of  \eqref{eq:def_lambdabarre_linear} when $u_0$ does not depend on time is explicitly  given by
\begin{equation}\label{eq:solution_lambdabarre_exp}
\lambda(t,x) = e^{-\alpha t} e ^{tT_W} u_0(x) + \alpha \int_0^t e^{-\alpha (t-s)} e^{(t-s)T_W}  u_0(x) ds,
\end{equation} 
where $e^{tT_W}$, $t\geq 0$ is the semigroup of the bounded operator $T_W$ defined as 
\begin{equation}\label{eq:eTW}
e^{tT_W}v:= \sum_{k=0}^\infty \dfrac{t^k}{k!} T_W^k v, \quad v\in L^\infty(I).
\end{equation}
\end{prop}
\begin{proof}
Define for $t\geq 0$ $A(t):= x \mapsto e^{\alpha t} \lambda(t,x)$. Multiplying \eqref{eq:def_lambdabarre_linear} by $e^{\alpha t}$, we obtain that $A(t)$ solves in $L^\infty(I)$ the differential equation $\dfrac{d}{dt} A(t) = \alpha  e^{\alpha t}u_0 + T_W A(t)$
with $A(0)=\lambda(0,\cdot)=u_0$. As $t\to e^{tT_W}v$ is the unique solution of $X'(t)=T_W X(t)$ with initial condition $X(0)=v$ for $v\in L^\infty(I)$, a variation of constants formula gives
$A(t)=e^{tT_W}u_0 + \alpha \int_0^t e^{(t-s)T_W}  e^{\alpha s} u_0 ds$, and \eqref{eq:solution_lambdabarre_exp} follows by definition of $A$.
\end{proof}
\begin{example}\label{ex:EDD_exp} Consider the particular case of Expected Degree Distribution (EED) (see \cite{Chung2002,Ouadah2019}): where $W(x,y)=f(x)g(y)$ with $f$ and $g$ two positive functions on $I$ such that $f,g \in L^2(I,\nu)$. Without any loss of generality, we assume $\int_I g d\nu=1$ and then $D(x)=f(x)$. We have then $r_\infty=\langle f,g\rangle$. When $\alpha \neq \langle  f, g \rangle$, the solution of \eqref{eq:solution_lambdabarre_exp} is given by
$$ \lambda(t,x) = u_0(x) + \dfrac{\langle g, u_0\rangle}{\alpha - \langle f,g\rangle } \left( 1 - e^{t\left( \langle f,g\rangle- \alpha \right)} \right) f(x).$$
The large time behavior depends then explicitly on the sign of $\langle f,g\rangle - \alpha$:
\begin{align*}
\langle f,g\rangle > \alpha \Rightarrow \forall x\in I,& \quad \lambda(t,x) \xrightarrow[t \to \infty]{} + \infty \text{ and }\\
\langle f,g\rangle < \alpha \Rightarrow \forall x\in I, &\quad \lambda(t,x) \xrightarrow[t \to \infty]{} u_0(x) +  \dfrac{\langle g,u_0\rangle}{\alpha - \langle f,g\rangle}  f(x).
\end{align*}
\end{example}
\begin{proof}
Recall that we have here $u_0(t,x)=u_0(x)$. By induction, we have explicitly that for $k\geq 1$, $T^k_W u_0 = f \langle  g, u_0 \rangle  \langle f,g\rangle^{k-1}$. Since $\displaystyle e^{tT_W}v=v +\sum_{k=1}^\infty \dfrac{t^k}{k!} T_W^k v$,  when $v\in L^\infty(I)$, together in \eqref{eq:solution_lambdabarre_exp}, we obtain
\begin{align*}
\lambda(t,x) &= e^{-\alpha t}\left(u_0(x) + \sum_{k=1}^\infty \dfrac{t^k}{k!}f(x) \langle  g, u_0 \rangle  \langle f,g\rangle^{k-1} \right) \\&\quad+ \alpha \int_0^t e^{-\alpha (t-s)} \left( u_0(x) + \sum_{k=1}^\infty \dfrac{(t-s)^k}{k!} f(x) \langle  g, u_0 \rangle  \langle f,g\rangle^{k-1} \right) ds\\
&= u_0(x) +  f(x) \dfrac{\langle  g, u_0 \rangle}{\langle  f, g \rangle} \left( e^{-t(\alpha-\langle  f, g \rangle)} - e^{-\alpha(t)}+ \dfrac{\alpha\left(1-e^{-t(\alpha-\langle  f, g \rangle)}\right)}{\alpha-\langle  f, g \rangle)}-1+e^{-\alpha t}\right)
\end{align*}
which gives then the result.
\end{proof}
We now consider the general case.

\subsection{Subcritical case}

We assume that:
\begin{equation}\label{eq:CS_sous_critical}
\Vert h \Vert_1 r_\infty <1.
\end{equation}
The main result is the following
\begin{thm}\label{thm:lambda_temps_long_sous_critique}
Assume \eqref{eq:CS_sous_critical}. Under Hypotheses \ref{hyp:existence_lambda_barre} and \ref{hyp:cas_lin} 
\begin{itemize}
\item there exists a unique function $\ell:I\mapsto \mathbb{R}^+$ solution of
\begin{equation}\label{eq:def_l_lim}
\ell(x) = u(x) + \Vert h \Vert_1 \int_I W(x,y)\ell(y)\nu(dy),
\end{equation}
continuous and bounded on $I$. Moreover, there exists $C_\ell>0$ such that for all $(x,y)\in I^2$,
\begin{equation} \label{eq:continuite_ell}
\vert \ell (x)-\ell(y) \vert \leq C_\ell \phi \left(\Vert x-y \Vert\right),
\end{equation}
where $\phi$ is given in \eqref{eq:pt_fixe_lips_espace}.
\item for any $x \in I$, we have the convergence
\begin{equation}\label{eq:cvg_ell}
\lambda(t,x) \xrightarrow[t\to\infty]{} \ell(x).
\end{equation}
\end{itemize} 
\end{thm}
The proof of Theorem \ref{thm:lambda_temps_long_sous_critique} will be given in Section \ref{S:proof_thm:lambda_temps_long_sous_critique}. We are now in position to address the question that motivates our paper: to what extent does the inhomogeneity of the underlying graph influence the macroscopic dynamics?
\begin{prop}\label{prop:ell_resultat_souscritique}
In the subcritical case \eqref{eq:CS_sous_critical}, $\ell$ solution of \eqref{eq:def_l_lim} is explicitly defined by
\begin{equation}\label{eq:explicit_ell_sscriti}
\ell = \sum_{k=0}^\infty \Vert h \Vert_1^k T_W^k u.
\end{equation} 
In particular, if $u_0$ is constant (i.e. for all $(t,x),~  u_0(t,x)=u(x)=u_0$), $\ell$ is uniform (i.e. $\ell(x)=\ell$ for every $x\in I$) if and only if the indegree is uniform (i.e. $D(x)=\int_I W(x,y) \nu(dy)=D$ for every $x\in I$). In such case, $r_\infty=D$.
\end{prop}
Note that \eqref{eq:explicit_ell_sscriti} informs us about the influence of the macroscopic graph $W$ on the dynamics: when $u_0$ is constant (thus $u_0=u$), we have
\begin{equation}\label{eq:ell_degre_DL}
\ell (x) =u_0 \left( \sum_{k=0}^\infty \Vert h \Vert_1 ^k D^{(k)}(x)\right),
\end{equation}
where $D^{(0)}=1$, $D^{(1)}=D(x)$ and $D^{(k+1)}=T_WD^{(k)}$. We see from \eqref{eq:ell_degre_DL} that in order to understand $\ell(x)$, one needs to explore the structure of the macroscopic graph around $x$.
\begin{proof}
Equation \eqref{eq:def_l_lim} can be written $\ell=u+\Vert h \Vert_1 T_W\ell$ which leads to $\Vert h \Vert_1 \left( \frac{Id}{\Vert h \Vert_1} - T_W \right) \ell = u$. As $r_\infty<\frac{1}{\Vert h \Vert_1}$ in the subcritical case, $\left( \frac{Id}{\Vert h \Vert_1} - T_W \right) $ is invertible (recall that $r_\infty= \sup_{\sigma \in Sp(T_W)} \vert \sigma \vert$) and then $\ell = \left( Id - \Vert h \Vert_1 T_W \right) ^{-1} u = \sum_{k=0}^\infty \Vert h \Vert_1^k T_W^k u$. 

We take now $u_0$ constant. Theorem \ref{thm:lambda_temps_long_sous_critique} gives the existence of a unique $\ell$ satisfying \eqref{eq:def_l_lim}. Assume that this solution is a constant function $\ell_0$, then for all $x\in I$ we have from \eqref{eq:def_l_lim} $\ell_0 = u_0 + \Vert h \Vert_1 \ell_0 \int_I W(x,y)\nu(dy)$ thus $\int_I W(x,y)\nu(dy)$ is constant and is equal to $\dfrac{\ell_0 - u_0}{\ell_0 \Vert h \Vert_1}$. Conversely, assume  $\int_I W(x,y)\nu(dy)$ constant and equal to $D$. Then, a direct computation gives $\Vert T_Wf \Vert_\infty \leq \Vert f \Vert_\infty D$ hence (as $r_\infty=\lim_{n\to\infty} \Vert T_W^n \Vert ^\frac{1}{n}$) $r_\infty\leq D$. As $T_W \mathbf{1}= D \mathbf{1}$ (where $\mathbf{1}(x)\equiv 1$), we have $D\leq r_\infty$ thus $D=r_\infty$. The subcritical case can then be written as $\Vert h \Vert _1 D <1$ and we can define $\ell_0:= \dfrac{u_0}{1-\Vert h \Vert_1 D}>0$. The constant function $\ell_0$ is continuous, bounded and solution of \eqref{eq:def_l_lim} which is unique: thus the solution of \eqref{eq:def_l_lim} is indeed constant.
\end{proof}

\subsection{Supercritical case}

We assume that:
\begin{equation}\label{eq:CS_sur_critical}
\Vert h \Vert_1 r_\infty >1.
\end{equation}
Note again that, without space interaction (i.e. $W=1$), \eqref{eq:CS_sur_critical} reduces to $\Vert h \Vert_1>1$ and it can be shown (see \cite{delattre2016}, Theorem  11) that $\lambda(t) \sim \alpha e^{\beta t} \to \infty$ for some $\alpha,\beta>0$. In our context with nontrivial $W$, one does not expect to have $\lambda(t,x) \xrightarrow[t\to\infty]{}\infty$ uniformly on $x$ as one can see from the obvious following example: take $W(x,y)=\alpha \mathbf{1}_{[0,\frac{1}{2})^2}(x,y)+\beta  \mathbf{1}_{[\frac{1}{2},1]^2}(x,y)$ for $\alpha>\beta$, then $r_\infty=\frac{\alpha}{2}$. This corresponds to two disconnected mean-field components $A$ (for neurons with positions in $I_A=[0,\frac{1}{2})$) and $B$ (for neurons with positions in $I_B=[\frac{1}{2},1]$). The critical parameter for population $A$ (resp. $B$) is hence $\alpha_c=\frac{2}{\Vert h \Vert_1}$ (resp. $\beta_c=\frac{2}{\Vert h \Vert_1}$). Taking now $\alpha>\alpha_c$ and $\beta<\beta_c$, \eqref{eq:CS_sur_critical} is satisfied but one does not have $\lambda(t,x) \xrightarrow[t\to\infty]{}\infty$ uniformly on $x$ as the population $B$ is subcritical, we only have $\lambda(t,x) \xrightarrow[t\to\infty]{}\infty$ for $x\in I_A$. 

In order to avoid such trivial examples, we assume that the graphon $W$ is sufficiently connected in the following way. Defining for $k\geq 1$: $$W^{(k)} (x,y):= \int_{I \times \cdots \times I} W(x,x_1) \cdots W(x_{k-1},y) dx_1\cdots dx_{k-1},$$ we assume primitivity of $W$ i.e. that there exists $k$ such that 
\begin{equation}\label{eq:Wkpos}
W^{(k)}>0.
\end{equation}
Note that $W^{(k)}$ is the kernel of the operator $T_W^k$. To understand \eqref{eq:Wkpos}, think of the finite dimensional case with $N$ particles interacting through a connectivity matrix $A$. In this context, $A$ being primitive means the existence of some $k\geq 1$ such that $A^k (i,j)>0$ for all $i,j$. Hypothesis \eqref{eq:Wkpos} is the exact counterpart in infinite dimension. We also assume the more technical assumptions:
\begin{hyp}\label{hyp:surcritique+}
\begin{equation}\label{eq:W_deg2}
\sup_x \int_I W(x,y)^2 \nu(dy) =: C_W^{(2)} < \infty,
\end{equation}
and
\begin{equation}\label{eq:W_sym}
\forall (x,y) \in I^2, \quad W(x,y)=W(y,x).
\end{equation}
We also assume that we can define the Laplace transform of $h$ for any $z\geq 0$ : $\mathcal{L}(h)(z):=\int_0^\infty e^{-tz} h(t)dt<\infty$.
Having $h$ of polynomial growth works for instance.
\end{hyp}
\begin{prop}
\label{prop:spectral_TW_L2}
Under Hypothesis \ref{hyp:surcritique+}, for all $p\geq1$, the linear operator $T_{ W}^{ p}$ is continuous  from $L^{ 2}(I)$ to $L^{ 2}(I)$, is compact, self-adjoint, its spectrum is the union of $\{0\}$ and a discrete sequence of eigenvalues $(\mu_{ n}^{(p)})_{ n\geq1}$ such that $ \mu_{ n}^{ (p)}\to0$ as $n\to\infty$. Moreover, the spectral radius $r_{ 2}(T_{ W}^{ p})$ verifies
\begin{equation}
\label{eq:spectral_radii_equal}
r_{ 2}(T_{ W}^{ p})= r_{ \infty}^{ p}
\end{equation}
where $r_{ \infty}$ defined in \eqref{eq:def_spectral_radius}.\\
Secondly, if one assumes further hypothesis \eqref{eq:Wkpos} for $p=k$, $ \mu_{ 0}^{ (k)}:= r_{ 2}(T_{ W}^{ k})>0$ is an eigenvalue of $T_{ W}^{ k}$ with a unique normalized eigenfunction $h_{ 0}^{ (k)}$ that is bounded, continuous and strictly positive on $I$. Moreover, every other eigenvalue $ \mu_{ n}^{ (k)}$ of $T_{ W}^{ k}$ has modulus $ \left\vert \mu_{ n}^{ (k)} \right\vert< r_{ 2}(T_{ W}^{ k})$.
\end{prop}

\begin{prop}\label{prop:lambda_temps_long_sur_critique}
Suppose that we are in the supercritical case \eqref{eq:CS_sur_critical}. Under Hypotheses \ref{hyp:cas_lin} and  \ref{hyp:surcritique+}, %if \eqref{eq:Wkpos}, \eqref{eq:W_deg2} and \eqref{eq:W_sym} are satisfied, 
$\int_I \lambda(t,x)^2 \nu(dx) \xrightarrow[t\to\infty]{} \infty$.
\end{prop}
The proofs of Propositions \ref{prop:spectral_TW_L2} and\ref{prop:lambda_temps_long_sur_critique} will be given in Section \ref{S:proof_prop:lambda_temps_long_sur_critique}.
\begin{remark}\label{rem:surcritical} Proposition \ref{prop:lambda_temps_long_sur_critique} provides a divergence result in $L^2$ norm, that is not uniform in $x$. But under more restrictive hypotheses on the connectivity of $W$ (without supposing $W$ symmetric), one can easily derive uniform divergence result. Assume $0<\inf_{x\in I} u(x)=: \underline{u}<\infty$  and $\Vert h \Vert_1 \inf_{x\in I} \Vert W(x,\cdot) \Vert_{1,\nu} >1$, we have then that $\inf_{x\in I} \lim_{t\to\infty} \lambda(t,x) = + \infty$. Note that by Fatou Lemma, $\inf_{x\in I} \Vert W(x,\cdot) \Vert_{1,\nu}  \leq r_\infty$, hence it also implies the result of Proposition \ref{prop:lambda_temps_long_sur_critique}.
\end{remark}
\begin{proof}[Proof of Remark \ref{rem:surcritical}.]
Let $v(t,x):=\inf_{s\geq t} \lambda(s,x)$. For all $x \in I$, set $$\underline{\ell}(x)=\liminf_{t\to\infty} \lambda(t,x).$$
We have for all $t>0$, using the positivity of $W,h$ and $\lambda$, and the fact that $\lambda(s,y)\geq v(\frac{t}{2},y)$ for all $s\in [\frac{t}{2},t]$,
\begin{align*}
\lambda(t,x) &= u_0(t,x) +\int_0^{\frac{t}{2}} \int_I W(x,y)h(t-s)\lambda(s,y)~ \nu(dy)~ ds+ \int_{\frac{t}{2}}^t \int_I W(x,y)h(t-s)\lambda(s,y) ~ \nu(dy) ~ ds\\
&\geq u_0(t,x) +  \int_0^{\frac{t}{2}} h(s)ds \int_I W(x,y) v\left(\frac{t}{2},y\right) \nu(dy),
\end{align*}
then taking $\displaystyle\lim\inf_{ t\to\infty}$, we obtain as $v(\cdot,y)$ is non decreasing by monotone convergence
\begin{align*}
%\underline{\ell}(x) &\geq \underline{u} + \Vert h \Vert _1 \int_I W(x,y) \underline{\ell}(y) \nu(dy),\\
\inf_{x\in I} \underline{\ell}(x)   &\geq \underline{u}+ \inf_{y\in I} \underline{\ell}(y) \Vert h \Vert _1  \inf_{x\in I}\int_I W(x,y) \nu(dy).
\end{align*}
As $u$ is positive and $\Vert h \Vert_1 \inf_{x\in I} \Vert W(x,\cdot) \Vert_{1,\nu} >1$ (in the subcritical case), it implies that $\inf_{x\in I} \underline{\ell}(x)  = \inf_{x\in I} \lim_{t\to\infty}  \inf_{s\geq t} \lambda(s,x)= + \infty	$ hence the result.
\end{proof}

\section{Applications}\label{S:application}

We give here examples of graphs $\left(\mathcal{G}^{(N)}\right)$ and corresponding graphons that satisfy the hypothesis of the paper. The main class of examples we have in mind fall into the framework of $W$-random graphs, see \cite{Lovsz2006,Luon2020}.

\subsection{A general class of examples}\label{S:general_class_ex}

Given a positive measurable kernel $(x,y)\mapsto \mathcal{P}(x,y)$ on $I^2$, for any $N\geq 1$  we consider the interaction kernel
\begin{equation}\label{eq:def_WN_P}
W_N(x,y):= \rho_N \min \left( \dfrac{1}{\rho_N}, \mathcal{P}(x,y) \right)
\end{equation}
with $ \rho_{ N}>0$. If $ \mathcal{ P}$ is bounded, by modifying $ \rho_{ N}$, we can suppose with no loss of generality $\left\Vert \mathcal{ P} \right\Vert_{ \infty}=1$ and $ W_{ N}(x,y)= \rho_{ N} \mathcal{ P}(x,y)$ whenever $ \rho_{ N} \leq \frac{ 1}{ \left\Vert \mathcal{ P} \right\Vert_{ \infty}}$. Then, one distinguish the dense case when $\lim_{N\to\infty} \rho_N= \rho>0$ and the diluted case when $\rho_N \to 0$.

\subsubsection{Uniformly bounded degrees} 
Suppose $\sup_x \int_I \mathcal{P}(x,y)\nu(dy)<\infty$. Recall that the prefactor $\kappa_i^{(N)}$ in \eqref{eq:def_lambdaiN} was here to ensure that the interaction remains of order 1 as $N\to\infty$. In the dense case renormalization is not necessary, one can take $\kappa_i^{(N)}=1$; and in the diluted case we can take $\kappa_i^{(N)}=\frac{1}{\rho_N}$. In either case, we take $w_N=\rho_N$. To satisfy Hypothesis \ref{hyp:conv_graph_concentration}, we require $\dfrac{N\rho_N}{\log(N)}\xrightarrow[N\to\infty]{} +\infty$. Hypothesis \ref{hyp:cvg_graph} or Hypothesis \ref{hyp:cvg_graph_infinf} with $W=\mathcal{P}$ are satisfied under regularity assumption on $\mathcal{P}$, see Propositions 3.2, 3.4, 3.6 and 3.9 of \cite{Luon2020}. Note that if $\rho_N=1$, it is a direct consequence of Proposition \ref{prop:scenario_hyp} (in this case $W^{\mathcal{G}_N^{(1)}}=W^{\mathcal{G}_N^{(2)}}$, see Definition \ref{def:graphs_G2} for the graph $\mathcal{G}_N^{(2)}$). Typical examples include the classic Erdös-Rényi graph with $\mathcal{P}=1$ (hence $W_N=\rho_N$ is uniform), interaction with the P-nearest neighbors (see \cite{Omelchenko2012}), or the EDD model previously defined in Example \ref{ex:EDD_exp}. These examples are thoroughly detailed in the next part.

\subsubsection{Unbounded degrees}\label{S:general_class_ex_div}
Suppose that $\mathcal{P}$ satisfies for all $x\in I$ : $\int_I \mathcal{P}(x,y)^2 \nu(dy) < \infty$ and $\mathcal{P}_*:=\inf_{z\in I} \int \mathcal{P}(z,y)\nu(dy)>0$, but  $\sup_x \int_I \mathcal{P}(x,y)\nu(dy)=\infty$. Then we take $\kappa_i^{(N)}=N\left(\rho_N \sum_{j=1}^N \min \left( \dfrac{1}{\rho_N}, \mathcal{P}(x_i,x_j)\right)\right)^{-1}$, and the macroscopic interaction kernel is $\displaystyle W(x,y)=\dfrac{\mathcal{P}(x,y)}{\int_I \mathcal{P}(x,z)\nu(dz)}$.
For such examples, see \cite{Luon2020}, Section 3.4. For instance, consider $\mathcal{P}(x,y)=\dfrac{1}{x^\alpha} g(y)$ with $g$ a probability measure on $[0,1]$ and $\alpha<\frac{1}{2}$.
\medskip

We present in the following different concrete examples of application of our results. We focus on the framework $I=[0,1]$ with the regular distribution of the positions $x_i^{(N)}=\frac{i}{N}$, $1\leq i \leq N$ and $\nu$ the Lebesgue measure. We take $f(x)=x$ to apply the results of Section \ref{S:sys-lim_temps}.

\subsection{Example: Erd\"os-R\'enyi graph}\label{S:ex_ER_cst}

Taking $ \mathcal{ P}\equiv 1$ with $ \rho_{ N}\in[0, 1]$, \eqref{eq:def_WN_P} becomes $W_{ N}\equiv \rho_{ N}$. This corresponds to the case where $ \mathcal{ G}^{ (N)}$ is a (possibly diluted) Erd\"os-R\'enyi random graph: the dense case  corresponds to $ \rho_{ N}\to \rho \in(0, 1]$ (and one takes $ \kappa_{ i}\equiv 1$ for all $i$) whereas the diluted case corresponds to $ \rho_{ N}\to 0$ (and one chooses $ \kappa_{ i} \equiv \frac{ 1}{ \rho_{ N}}$). The dilution condition (3.10) reduces to 
\begin{equation}
\label{eq:dilution_ER}
\frac{ N\rho_{ N}}{ \log N} \xrightarrow[ N\to\infty]{}+\infty.
\end{equation} 
Note that the condition \eqref{eq:dilution_ER} is the very same condition already met in the similar context of diffusions interacting on Erd\"os-R\'enyi random graphs (see \cite{DelattreGL2016,bet2020weakly}), in the quenched case (i.e. where the randomness of the graph is frozen). In the (technically simpler) annealed case (where one integrates also w.r.t. the randomness of the graph), it is possible to get rid of this supplementary $\log N$ term (that is  required, in the present quenched setting, for our Borel-Cantelli arguments to work) and assume only $ N \rho_{ N} \to \infty$ (as this has been done for diffusions in an annealed setting e.g. in \cite{Bhamidi2016,Coppini2019,bayraktar2021graphon}). Here, the limiting graphon is given by $ W \equiv \rho$ (with $ \rho=1$ in the diluted case). Condition \eqref{eq:hyp_W_pseudolip_theta} is then trivially satisfied and one can apply Theorems \ref{thm:cvg_0} and \ref{thm:cvg-sup_0} (the convergence of graphs is seen in Proposition \ref{prop:scenario_hyp}). As the degree is constant, Proposition \ref{prop:ell_resultat_souscritique} gives $r_\infty=\rho$. As Hypothesis \ref{hyp:surcritique+} is satisfied, there is a transition phase around $\rho_c=\frac{1}{\Vert h \Vert_1}$.

In the subcritical case $\Vert h \Vert_1 \rho<1$, Theorem \ref{thm:lambda_temps_long_sous_critique} gives that for any $x\in I$, $\lambda(t,x) \xrightarrow[t\to\infty]{} \ell(x)=\dfrac{u(x) (1-\Vert h \Vert_1 \rho) + \Vert u \Vert_{I,\nu,1} \Vert h \Vert_1 \rho}{1-\Vert h \Vert_1 \rho}$. Note that if $u_0$ is constant, $\ell=\dfrac{u}{1-\Vert h \Vert_1 \rho}$. Corresponding simulations are given in Figures~\ref{fig:simC} and~\ref{fig:simB}.

\begin{figure}[h]
\centering
\subfloat[Matrix of $\mathcal{G}^{(N)}$]{\includegraphics[width=0.4\textwidth]{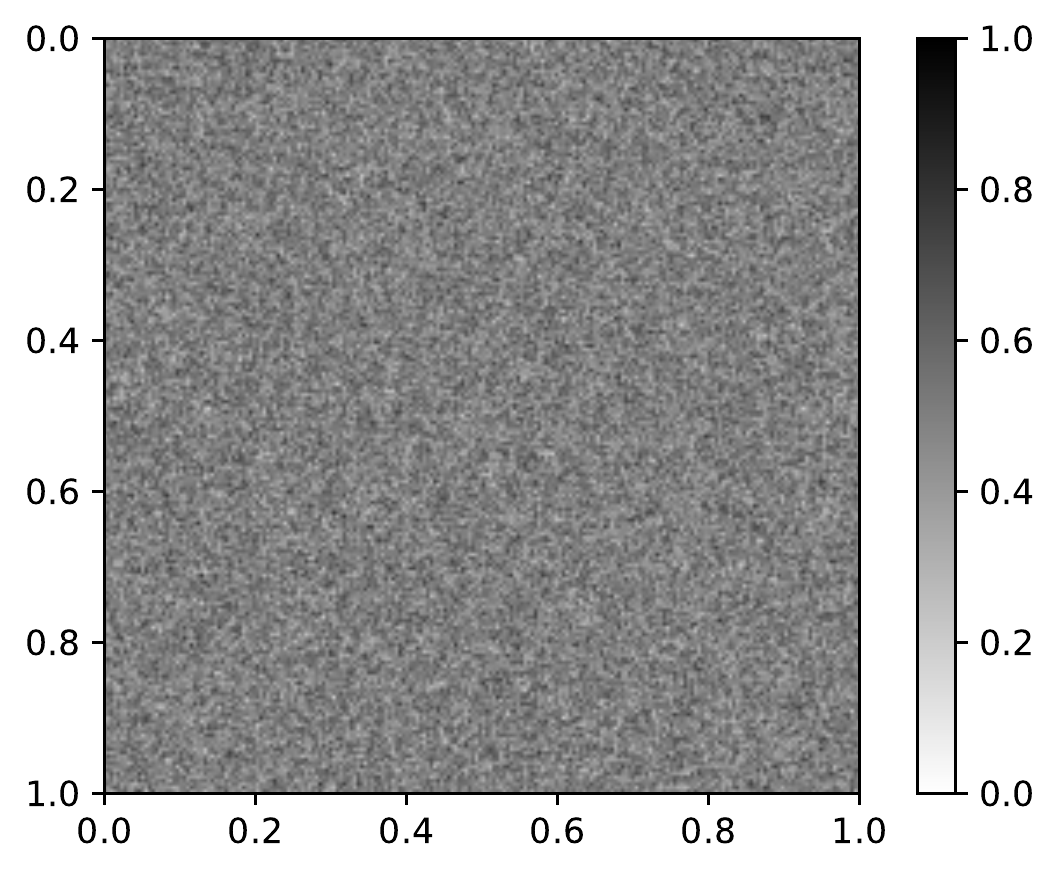}\label{subfig:simC_graphe}}
\quad
\subfloat[Each dot represents $\lambda_N(T,x)$ for $x\in \underline{x}^{(N)}$, and the plain line corresponds to the macroscopic limit $\ell(x)$.]{\includegraphics[width=0.5\textwidth]{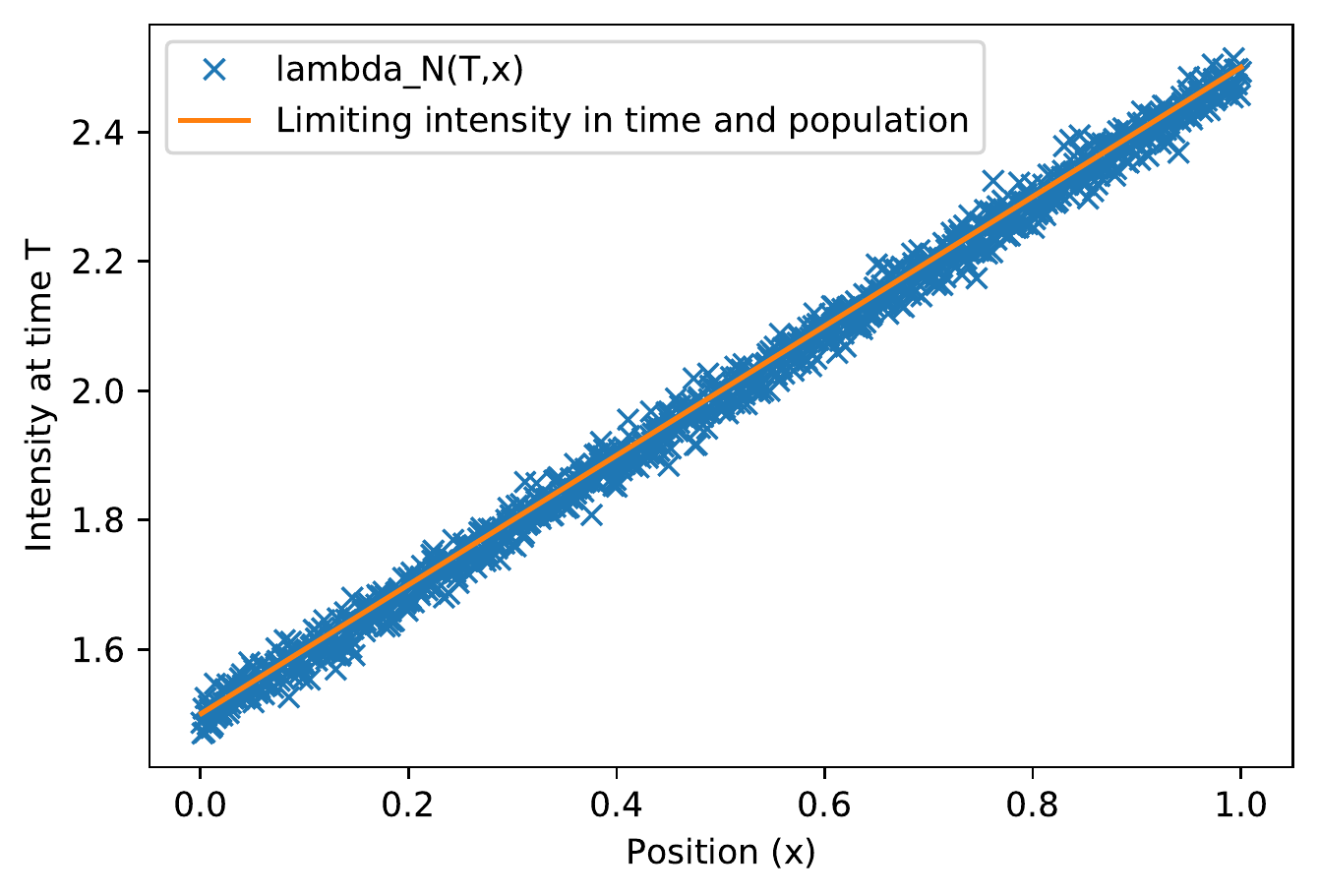}\label{subfig:simC_lambda(T)}}\\\subfloat[Evolution of microscopic and macroscopic intensities of three particles at positions $x=$0.25 (blue - the lowest), 0.5 (red) and 0.75 (green - the highest). In each case, the colored line represents $\lambda_N(t,x)$, the dashed line represents $\lambda(t,x)$ and the dotted line represents the limit $\ell (x)$.]{\includegraphics[width=0.60\textwidth]{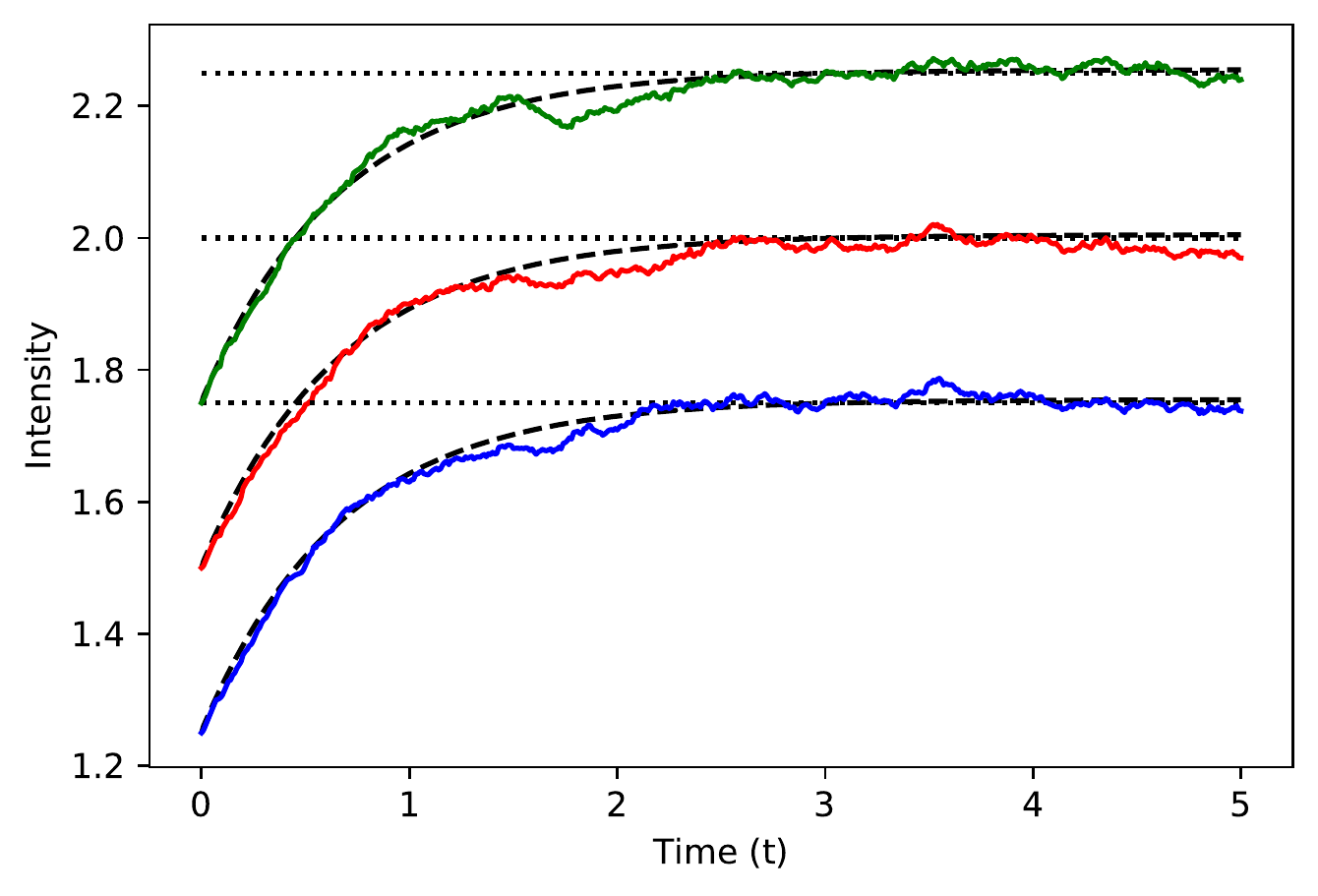}\label{subfig:simC_intensite(t)}}
\caption{Simulation of Example \ref{S:ex_ER_cst} with inhomogeneous $u_0$}%
\label{fig:simC}%
\fnote{We chose $h(t)=e^{-\alpha t}$ with $\alpha=2$, $p=0.5$ for the Erd\"os R\'enyi graph and $u_0(t,x)=x+1$. We are in the subcritical case $\Vert h \Vert_1 p <1$ and the limiting intensity is given by $\ell(x) =% u(x) + \dfrac{\Vert u \Vert_{1,\nu} p}{\alpha-p}=
x+\frac{1}{2}$. We run a simulation for $N=1000$ and $T=5$: in \ref{subfig:simC_graphe}, we show the matrix of the Erd\"os-R\'enyi graph $\mathcal{G}^{(N)}$. In \ref{subfig:simC_lambda(T)}, we represent the spatial distribution of intensities at fixed time $T$. In \ref{subfig:simC_intensite(t)}, we show the time evolution of the intensities for different positions. Note here that the inhomogeneity of $\ell(x)$ is due to the inhomogeneity of the $u_0$, not of the graph.}
\end{figure}

\begin{figure}[h!]
\centering
\subfloat[Matrix of $\mathcal{G}^{(N)}$]{\includegraphics[width=0.25\textwidth]{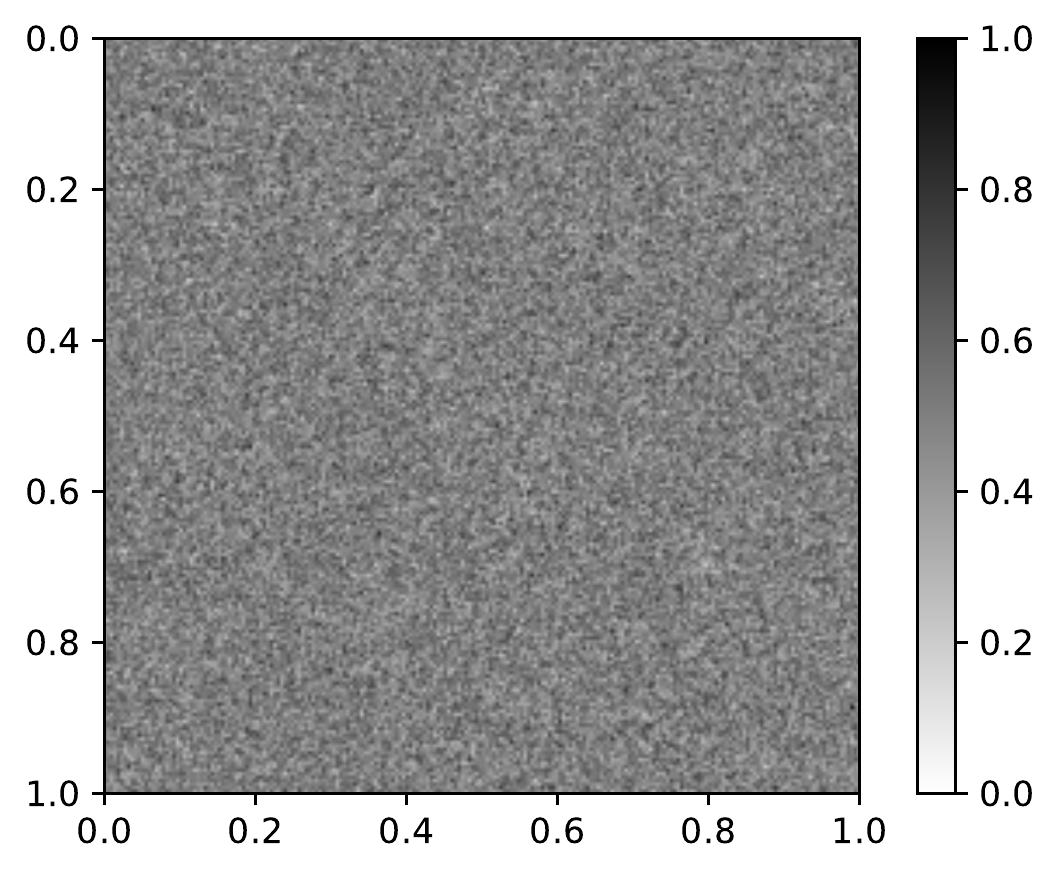}\label{subfig:simB_graphe}}
\quad
\subfloat[Each dot represents $\lambda_N(T,x)$ for $x\in \underline{x}^{(N)}$, and the plain line corresponds to the macroscopic limit $\ell(x)$.]{\includegraphics[width=0.50\textwidth]{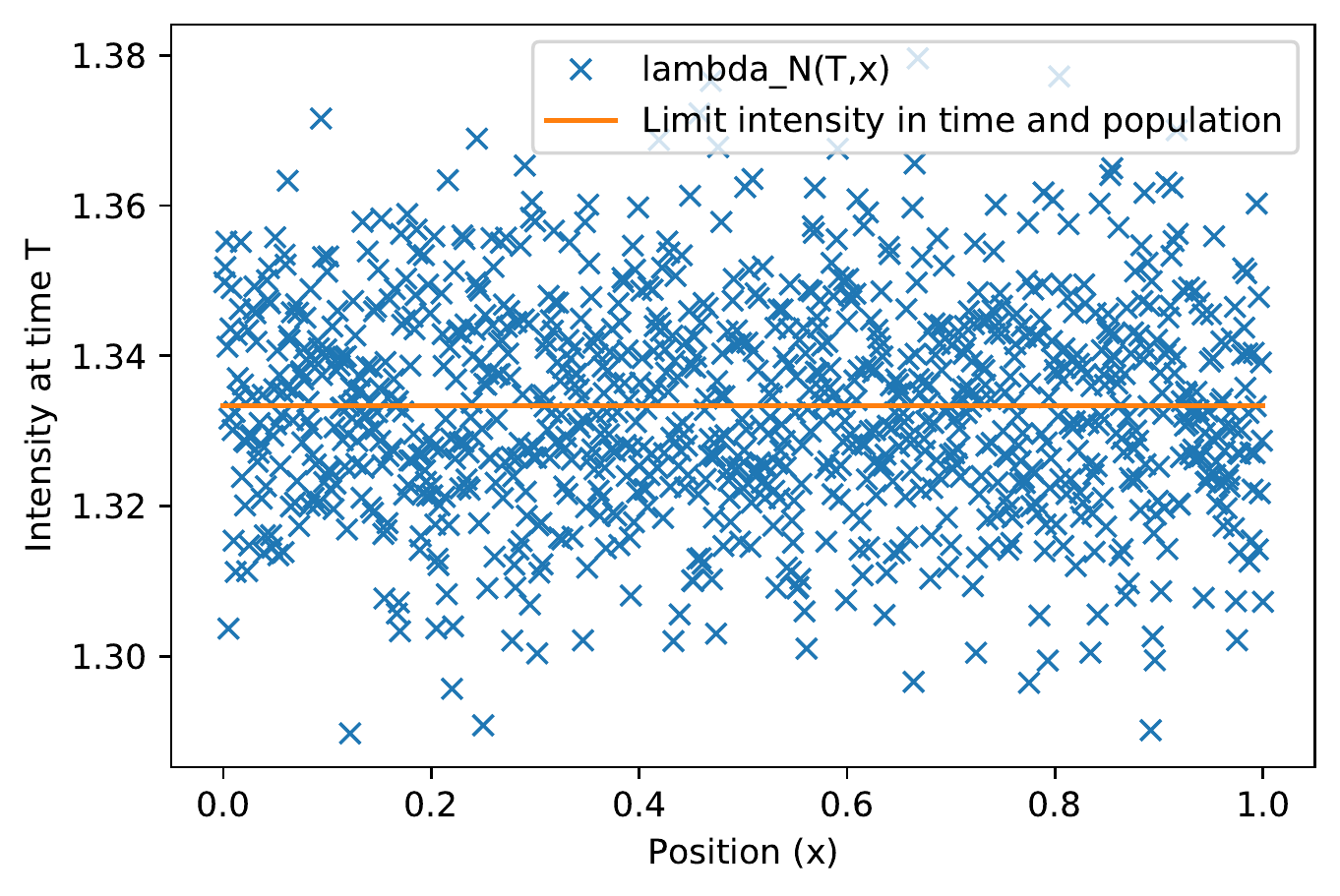}\label{subfig:simB_lambda(T)}}
\quad
\subfloat[Evolution of microscopic and macroscopic intensities of two particles at positions $x=$0.5 (red) and 0.75 (blue). The colored lines represent $\lambda_N(t,x)$, the dashed line represents $\lambda(t)$ and the dotted line represents the limit $\ell$.]{\includegraphics[width=0.60\textwidth]{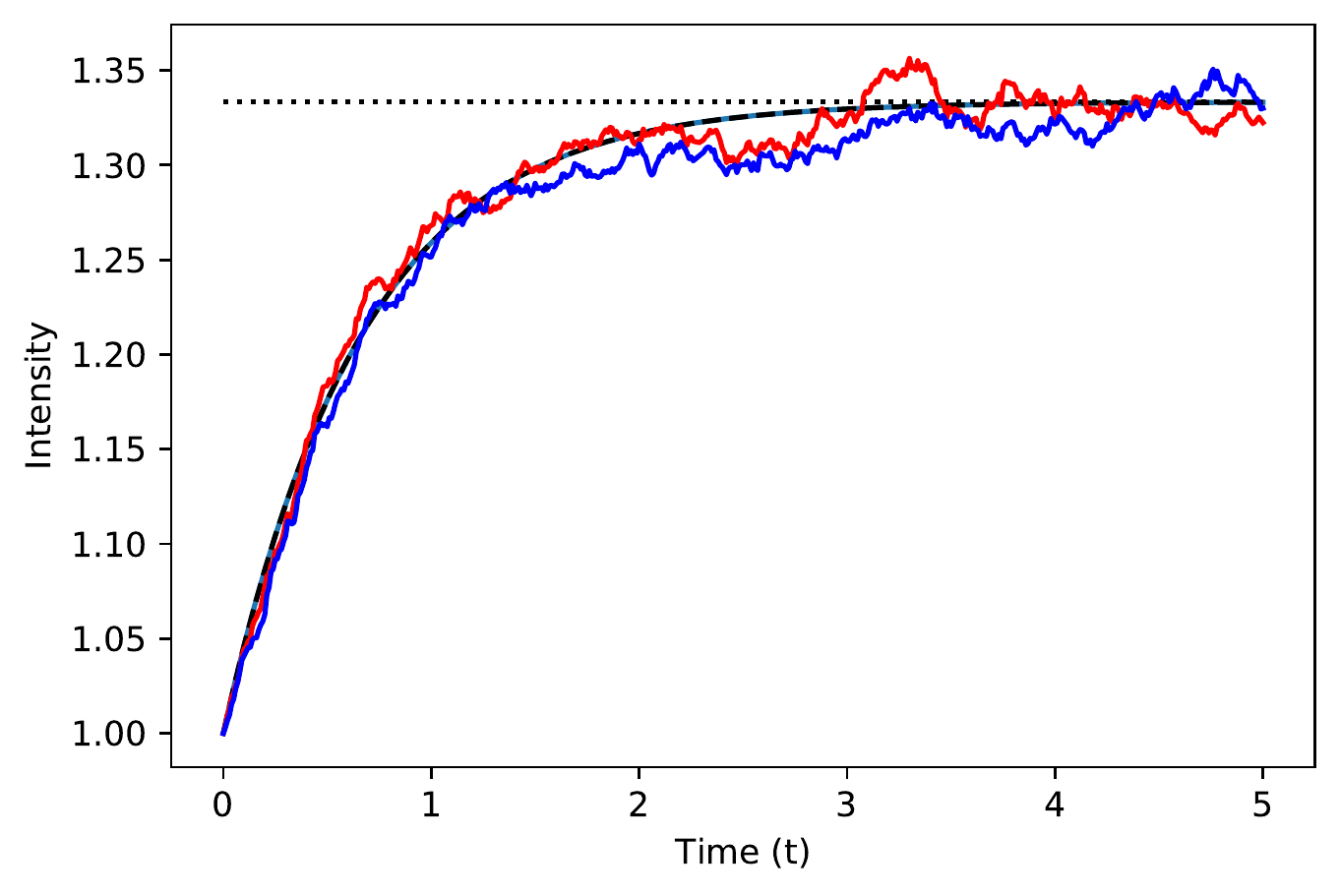}\label{subfig:simB_intensite(t)}}
\caption{Simulation of Example \ref{S:ex_ER_cst} with homogeneous $u_0$}%
\label{fig:simB}%
\fnote{We chose $h(t)=e^{-\alpha t}$ with $\alpha=2$, $p=0.5$ for the Erd\"os R\'enyi graph and $u_0(t,x)=1$, we are in the subcritical case ($\Vert h \Vert_1 p = \frac{1}{4}<1$). As the graph is homogeneous in space and the self-activity is constant, the limit solution of \eqref{eq:def_lambdabarre} dos not depend of the position: $\lambda(t)=%\dfrac{u_0}{p-\alpha}\left(pe^{-(\alpha-p)t}-\alpha\right)=
\frac{4}{3} - \frac{1}{3}e^{-\frac{3}{2}t}$. The limiting intensity is constant $\ell = \frac{4}{3}$. We run a simulation for $N=1000$ and $T=5$. In \ref{subfig:simB_graphe}, we show the matrix of the Erd\"os-R\'enyi graph $\mathcal{G}^{(N)}$. In \ref{subfig:simB_intensite(t)}, we show the time evolution of the intensities for different positions. In \ref{subfig:simB_lambda(T)}, we represent the spatial distribution of intensities at fixed time $T$.}
\end{figure}

In the supercritical case $\Vert h \Vert_1 \rho>1$, as $W$ is constant, we can directly apply Remark \ref{rem:surcritical} and obtain  $\inf_{x\in I} \lim_{t\to\infty} \lambda(t,x) = + \infty$.
 
\subsection{Example: P-nearest neighbor model \cite{Omelchenko2012}}\label{S:ex_degW_cst}  

Consider the kernel $W(x,y)=\mathbf{1}_{d_{\mathcal{S}_1}(x,y)< r}$ for any $(x,y)\in I^2$ for some fixed $r\in (0,\frac{1}{2})$ and with
\begin{equation}\label{eq:def_dS1}
d_{\mathcal{S}_1}(x,y)=\min(\vert x-y \vert,1-\vert x-y \vert).
\end{equation} 
It means that the particles at positions $x$ and $y$ interact if and only if they are at distance less than $r$ on the circle $\mathcal{S}_1:=\mathbb{R}_{/  [0,1]}$. This corresponds to a deterministic graph. As \eqref{eq:hyp_W_pseudolip_theta} is satisfied - for any $(x,x')\in I^2$, $\int_I \vert W(x,y)-W(x',y) \vert \nu(dy) = \int_0^1 \left| \mathbf{1}_{\vert x-y \vert< r} - \mathbf{1}_{\vert x'-y \vert< r}\right|dy\leq 4 \vert x-x'\vert$,  we can apply Theorems \ref{thm:cvg_0} and \ref{thm:cvg-sup_0}. As for any $x\in I$, $\int_I W(x,y) dy = 2r$, Proposition \ref{prop:ell_resultat_souscritique} gives that $r_\infty=2r$. The assumptions \eqref{eq:W_deg2} and \eqref{eq:W_sym} are trivially verified, and as $W^{(k)}$ is positive for $k:=\inf \left\{ n \geq 0, nr\geq \frac{1}{2}\right\}$, Hypothesis \ref{hyp:surcritique+} is satisfied and there is a transition phase around $r_c=\frac{1}{2 \Vert h \Vert_1}$. In the subcritical case ($r <r_c$), Proposition \ref{prop:ell_resultat_souscritique} gives that when $u_0$ is constant the limiting intensity is explicit and $\ell=\dfrac{u_0}{1-2r\Vert h \Vert_1}$. We give an example of simulation in this case in Figure \ref{fig:simD}. In the supercritical case ($r>r_c$), as the degree is constant, we can directly apply Remark \ref{rem:surcritical} and obtain  $\inf_{x\in I} \lim_{t\to\infty} \lambda(t,x) = + \infty$. 
 
\begin{figure}[h!]
\centering
\subfloat[Matrix of $\mathcal{G}^{(N)}$]{\includegraphics[width=0.25\textwidth]{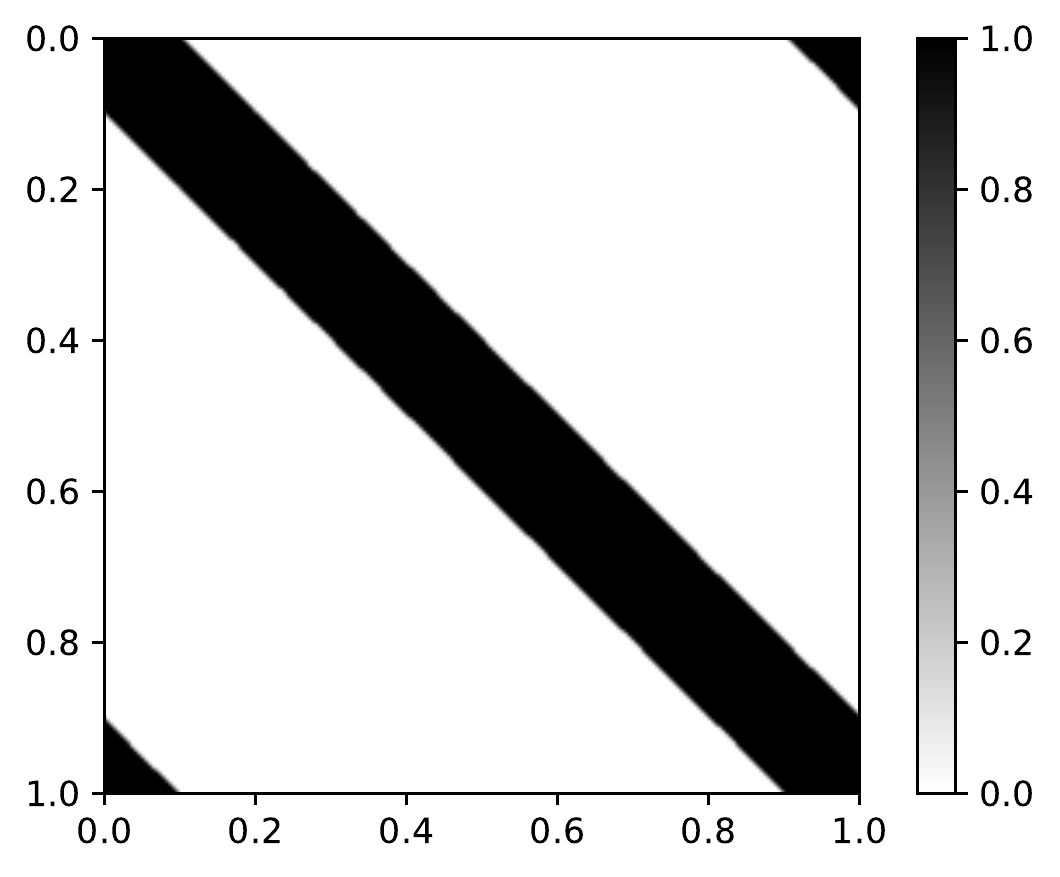}\label{subfig:simD_graphe}}
\quad
\subfloat[Evolution of microscopic and macroscopic intensities of two particles at positions $x=$0.5 (red) and 0.1 (blue). The colored lines represent $\lambda_N(t,x)$, the dashed line represents $\lambda(t)$ and the dotted line represents the limit $\ell$.]{\includegraphics[width=0.60\textwidth]{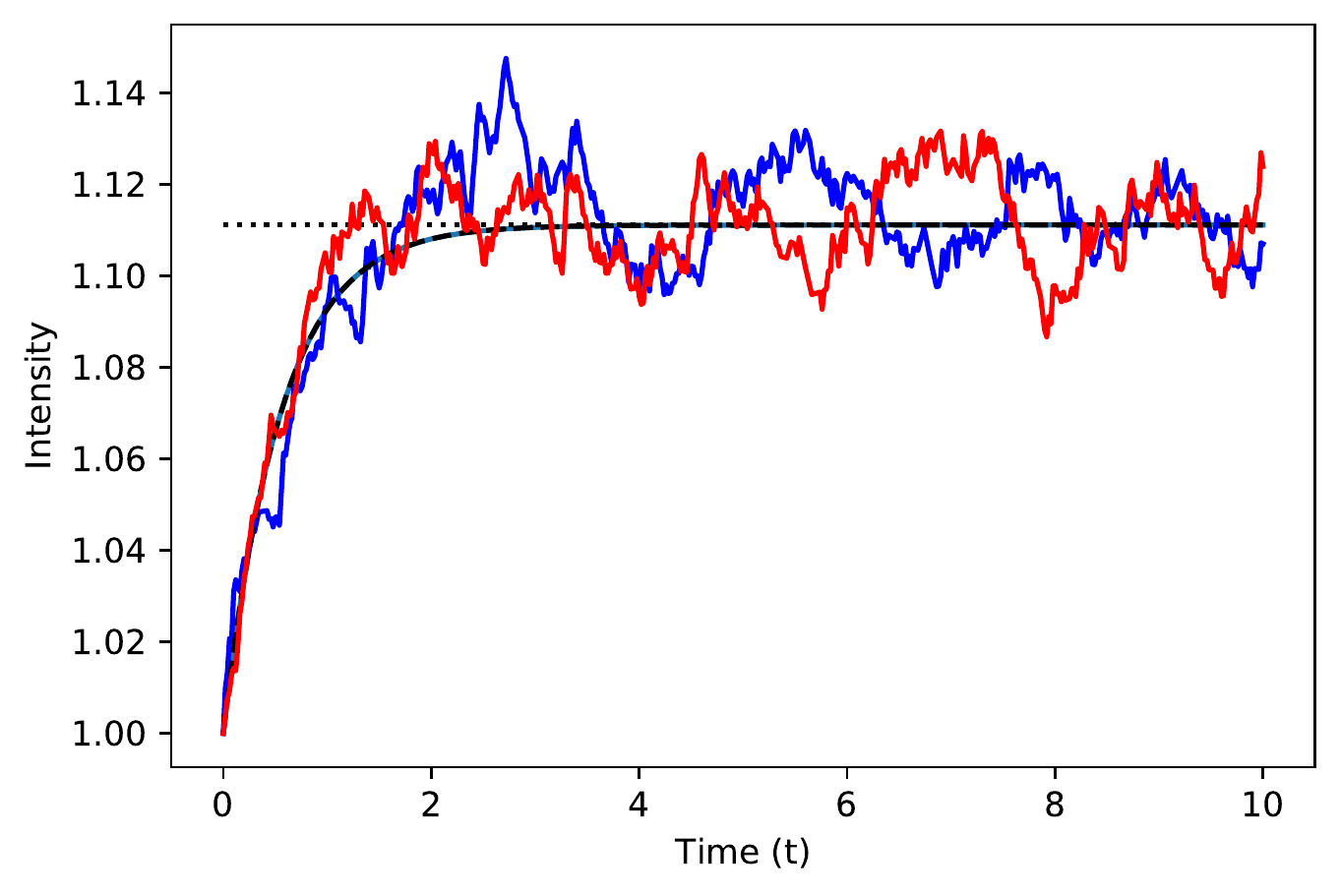}\label{subfig:simD_intensite(t)}}
\caption{Simulation of Example \ref{S:ex_degW_cst} in the subcritical case}%
\label{fig:simD}%
\fnote{We chose $h(t)=e^{-2t}$, $r=0.1$ and  $u_0(t,x)=1$, hence we are in the subcritical case as $2r \Vert h \Vert_1<1$. The graph is not homogeneous in space but has a symmetry and the self-activity $u_0$ is constant, hence the solution of \eqref{eq:def_lambdabarre} does not depend of the position: $\lambda(t)%=\dfrac{u_0}{2r-\alpha}\left(2re^{-(\alpha-2r)t}-\alpha\right)
=\frac{10}{9} - \frac{1}{9}e^{-\frac{9}{5}t}$. The limiting intensity is constant $\ell = \frac{10}{9}$. We run a simulation for $N=500$ particles and a final time $T=10$: in \ref{subfig:simD_graphe}, we show the matrix of the graph $\mathcal{G}^{(N)}$ obtained.  In \ref{subfig:simD_intensite(t)}, we show the time evolution of the intensities for different positions. We see that the simulated intensities follow indeed the behavior expected, as they are close to $\lambda(t,x)$ and converge toward a constant limit $\ell$.}
\end{figure}

\subsection{Example: Inhomogeneous graph with EDD  \cite{Chung2002}}\label{S:ex_W=fg} 

Recall Example \ref{ex:EDD_exp}: $W(x,y)=f(x)g(y)$ with $f$ and $g$ two positive bounded functions on $I$ such that $f,g \in L^2(I,\nu)$ and $\int_I g d\nu=1$. We also suppose that $f$ satisfies a H\"older condition for $\vartheta\in (0,1]$ and is bounded. Note that the indegree is $D(x)=f(x)$. Hypothesis \ref{hyp:existence_lambda_barre} is satisfied and we can apply Theorems \ref{thm:cvg_0} and \ref{thm:cvg-sup_0}. The operator $T_W$ is then defined as $T_Wk(x)=f(x) \langle g,  k\rangle$ for $k\in L^\infty$. An iteration gives $T_W^n  = \langle f,  g \rangle ^{n-1} T_W$ for all $n\geq 1$, and then $r_\infty=\langle f,  g \rangle$, so that the phase transition is given in term of $\langle f,  g \rangle\Vert h \Vert_1<1$ or $\langle f,  g \rangle\Vert h \Vert_1>1$ (and we retrieve Example \ref{ex:EDD_exp} in the exponential case).

In the subcritical case $\Vert h \Vert_1 \langle f,  g \rangle <1$, Theorem \ref{thm:lambda_temps_long_sous_critique} gives that for any $x\in I$, $\lambda(t,x) \xrightarrow[t\to\infty]{} \ell(x)$ where $\ell$ is the solution of \eqref{eq:def_l_lim}, that is $\ell(x)=u(x) +\Vert h \Vert_1 \dfrac{f(x)  \langle u,  g \rangle }{1-\Vert h \Vert_1 \langle f,  g \rangle}$. We give an example of simulation in the case $f=g$ in Figure \ref{fig:simE}.

\begin{figure}[h!]
\centering
\subfloat[Matrix of $\mathcal{G}^{(N)}$]{\includegraphics[width=0.25\textwidth]{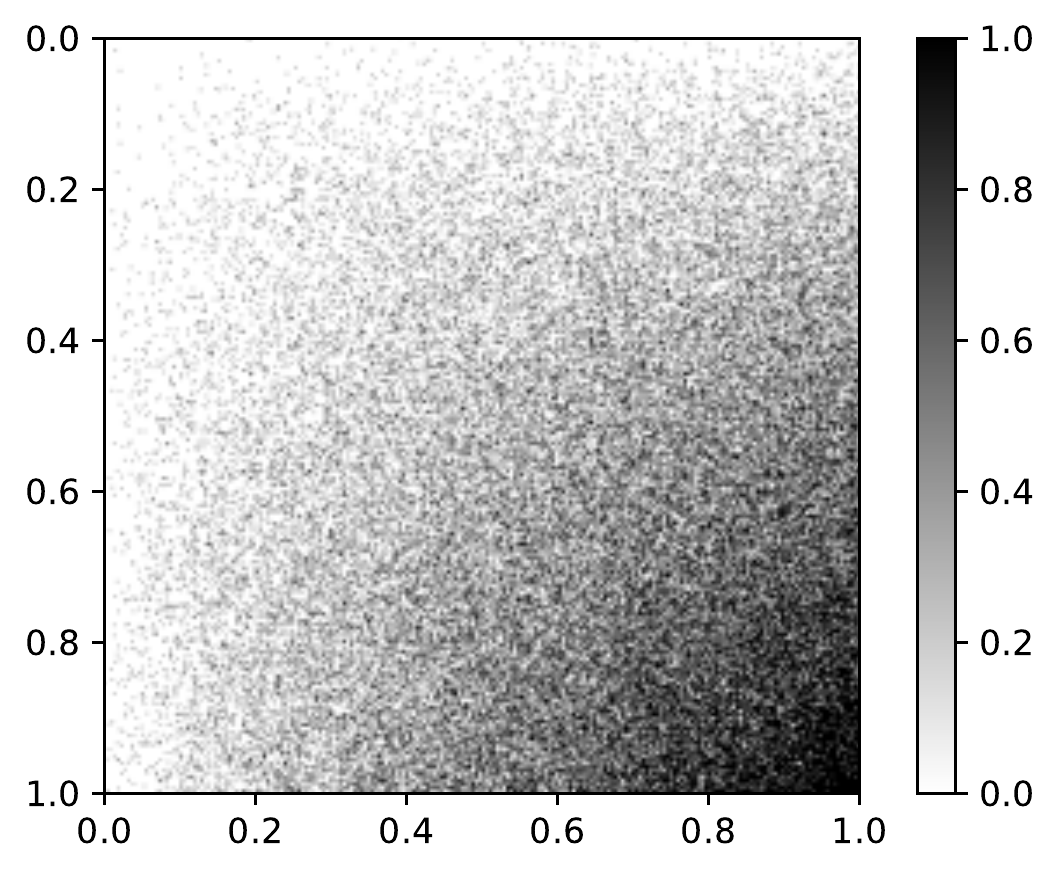}\label{subfig:simE_graphe_alea}}
\quad
\subfloat[Graphon $W(x,y)=xy$]{\includegraphics[width=0.25\textwidth]{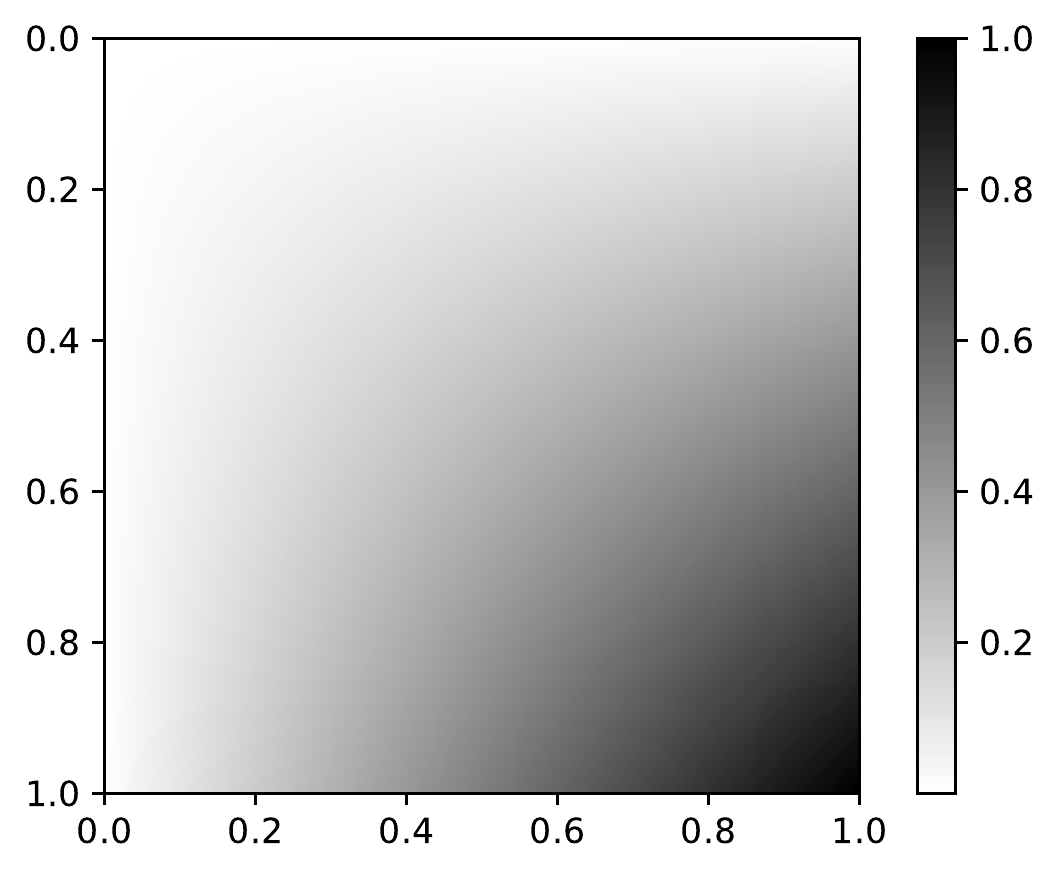}\label{subfig:simE_graphe_fix}}
\quad
\subfloat[Each dot represents $\lambda_N(T,x)$ for $x\in \underline{x}^{(N)}$, and the plain line corresponds to the macroscopic limit $\ell(x)$. ]{\includegraphics[width=0.30\textwidth]{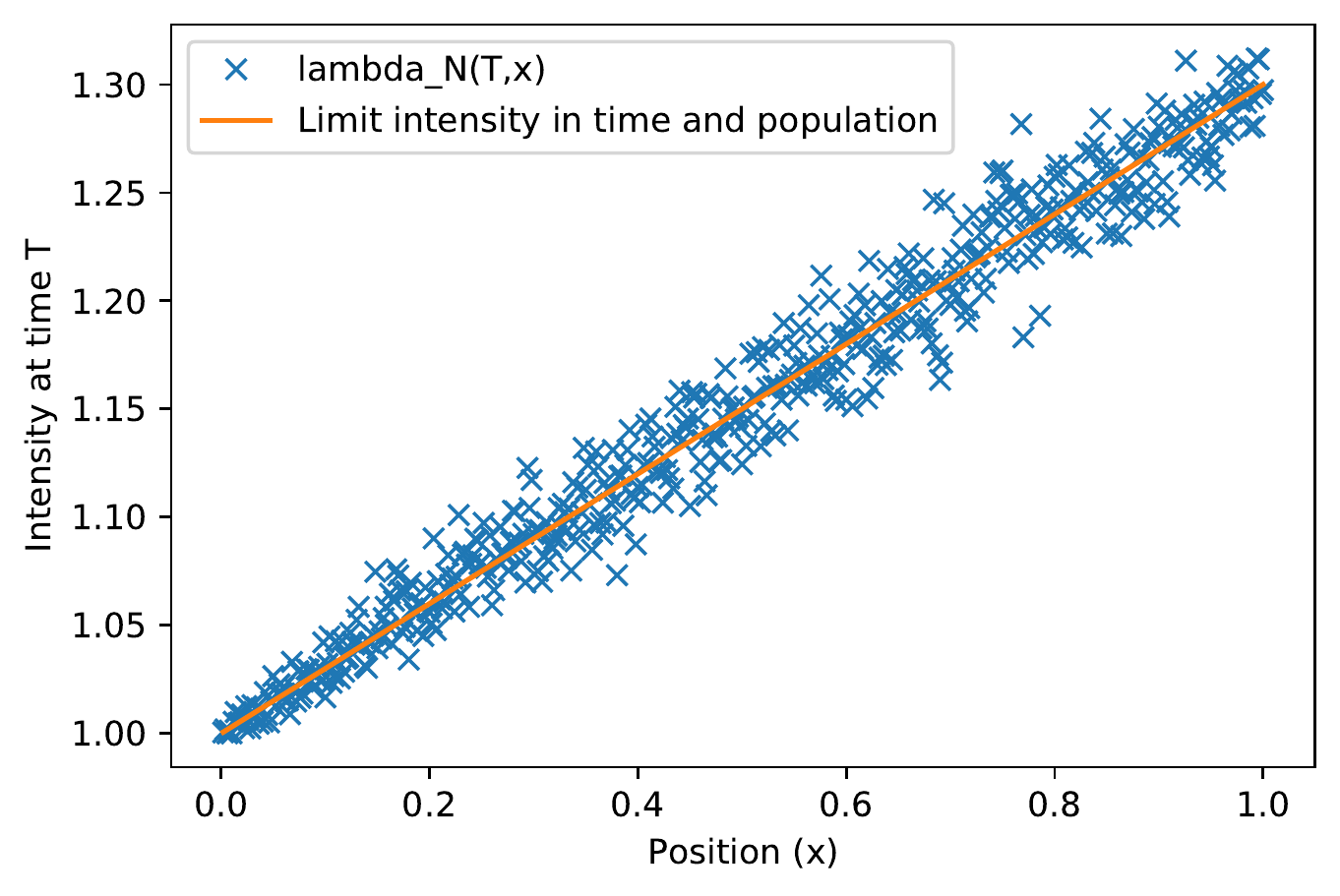}\label{subfig:simE_lambda(T)}}
\\
\subfloat[Evolution of microscopic and macroscopic intensities of two particles at positions $x=$0.5 (blue - the highest) and 0.3 (red - the lowest). In each case, the colored line represents $\lambda_N(t,x)$, the dashed line represents $\lambda(t,x)$ and the dotted line represents the limit $\ell (x)$.]{\includegraphics[width=0.60\textwidth]{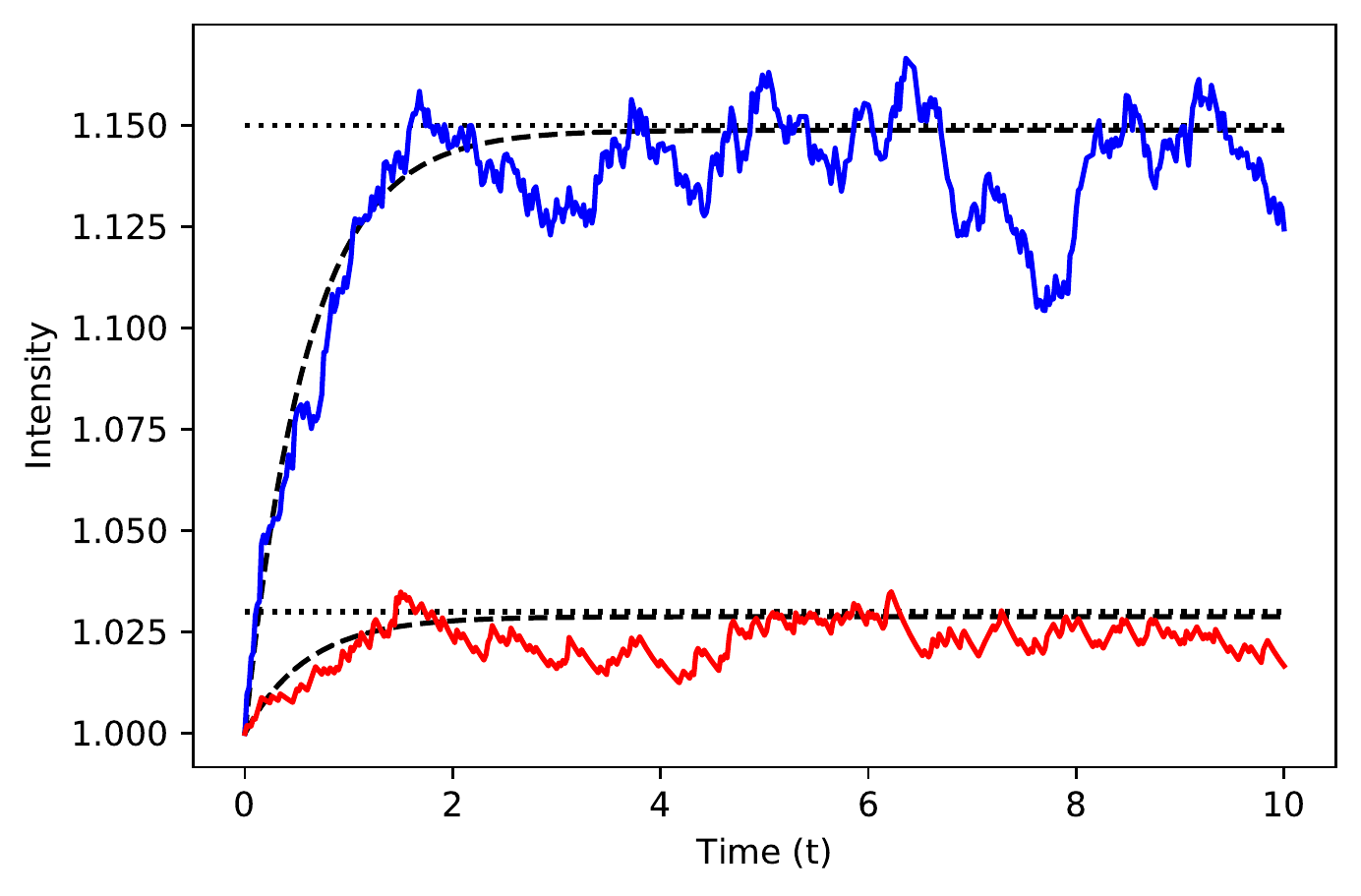}\label{subfig:simE_intensite(t)}}
\caption{Simulation of Example \ref{S:ex_W=fg}}%
\label{fig:simE}%
\fnote{We chose $h(t)=e^{-\alpha t}$ with $\alpha=2$, $u_0(t,x)=1$ and $f(x)=g(x)=x$, that is $W(x,y)=xy$: we are in the subcritical case ($\Vert h \Vert_1\langle f,  f \rangle<1$) and the limiting intensity is  $\ell(x) = 1 + \frac{3}{10}x$. We run a simulation for $N=500$ and $T=10$: in \ref{subfig:simE_graphe_fix}, we represent the graphon $W$, and in \ref{subfig:simE_graphe_alea} we show the matrix random graph $\mathcal{G}^{(N)}$ obtained.  In \ref{subfig:simE_lambda(T)}, we represent the spatial distribution of intensities at fixed time $T$. In \ref{subfig:simE_intensite(t)}, we show the time evolution of the intensities for different positions. Note here that the inhomogeneity of $\ell(x)$ is only due to the inhomogeneity of the kernel $W$. }
\end{figure}

\subsection{Example: Multi-class interaction populations}

Another interesting case concerns deterministic and inhomogeneous graphs modeling the macroscopic organization of neurons into vertical columns. A generic construction is the following: divide $I=(0,1]$ into $P$ consecutive subintervals $I_j$ with respective length $\alpha_j>0$, that is, $I_j=(\alpha_1+\cdots+\alpha_{j-1}, \alpha_1+\cdots+\alpha_j]$ and $\alpha_1+\cdots+\alpha_P=1$. Take any connectivity matrix $M$ between the $P$ populations, $M=(m_{ij})_{1\leq i,j \leq P}$ with $m_{ij}\in\{0,1\}$ modeling the deterministic connection between subpopulations $i$ and $j$. Take the self-activity fixed for each population, described by $u_0(t)=\left(u_{0,i}(t)\right)_{1\leq i \leq P}$ and converging towards $u=\left(u_i\right)_{1\leq i\leq P}$ as $t\to\infty$. Define finally $W(x,y)=\sum_{i,j=1}^P m_{ij} \mathbf{1}_{x\in I_i} \mathbf{1}_{y\in I_j}$, as well as $\widetilde{M}=\left( \alpha_j m_{ij}\right)_{1\leq i,j \leq P}$. Then $r_\infty=\rho(\widetilde{M})$ where $\rho(\widetilde{M})$ is the spectral radius of $\widetilde{M}$ so that the phase transition described above is given here in terms of $\rho(\widetilde{M}) \Vert h \Vert_1<1$ or $\rho(\widetilde{M})\Vert h \Vert_1>1$.

The limiting intensity $\lambda(t,x)$ is constant on each population, described by the vector $\widetilde{\lambda}(t)=\left(\lambda_i(t)\right)_{1\leq i \leq P}$ which solves $\widetilde{\lambda}(t)=u_0(t) + \int_0^t h(t-s) \widetilde{M}\widetilde{\lambda}(s)ds$. In the subcritical case, the limit $\ell=\left(\ell_i\right)_{1\leq i \leq P}$ is piecewise constant (on each population) and solves $\ell = u + \Vert h \Vert_1 \widetilde{M} \ell$. In the supercritical case, $\sum_{i=1}^P \alpha_i \lambda_i(t)^2 \xrightarrow{t\to\infty} \infty$ when $\widetilde{M}$ is symmetric and primitive.

\begin{rem}\label{rem:prim_mat}
A closer look at the proof of Theorem 2.3 and (2.5) of \cite{athreya1976feller} shows that $\lambda_i(t)\to \infty$ for all $i\in \llbracket 1, p \rrbracket$ in the simpler case when $M$ is only irreducible but not necessarily symmetric nor primitive (e.g. the case considered in \cite{Ditlevsen2017}).
\end{rem}

\section{Possible extensions}

Inhibition is an important factor in neuronal dynamics. In the present model, we restricted ourselves for simplicity to a non-negative interaction kernel $W$. Nevertheless, we can easily introduce a dependence in inhibition by considering signed spatial interaction $W$: take $w_{ij}$ in \eqref{eq:def_lambdabarre} of the form $w_{ij}=\kappa_i^{(N)}\xi_{ij}^{(N)} s(x_i,x_j)$, where $s(x,y)\in\{\pm 1\}$ expresses the nature of interaction between neurons located in $x$ and $y$: $s(x,y)=1$ if the interaction is excitatory or $-1$ if it is inhibitory. For instance, if the nature of the interaction only depends on the neuron sending the information, take $s(x,y)=s(y)$. The resulting macroscopic limit is now expressed in terms of a signed interaction kernel $W$. The results presented in this paper remain the same with appropriate regularity assumptions on $s$, up to notational changes in the norms where $W$ is replaced by $\vert W\vert$.

Another possible extension concerns the  memory kernel $h$. In our paper, this kernel is identical on the population. One could think that neurons can present an inhomogeneity in the way they remember the past information, that is considering a memory kernel depending also on the positions of the neurons $h(t,x,y)$. With enough regularity on such $h$, the same results hold up to notational changes.

\section{Proofs: existence and uniqueness of the model and its limit}

\subsection{Proof of Proposition \ref{prop:exis_H_N}}\label{S:proof_prop:exis_H_N}
 
We study the pathwise uniqueness by considering the total variation distance between two such processes. We show the existence by constructing a Cauchy sequence adapted and using a Picard iteration argument.
We follow the structure of the proof proposed in \cite{delattre2016} (Theorem 6). We consider a family of independent Poisson measures $\left(\pi_i\left(ds,dz\right)\right)_{1\leq i \leq N}$ with intensity $dsdz$. We denote $\kappa_i^{(N)}\xi_{ij}^{(N)}$ by $w_{ij}$. We start by showing uniqueness and we omit the notation $(N)$ for simplicity.  We set $\left(Z_i\left(t\right)\right)_{i \in  \llbracket 1, N \rrbracket , t\geq 0}$ and $\left(\underline{Z}_i\left(t\right)\right)_{i \in  \llbracket 1, N \rrbracket , t\geq 0}$ two solutions of the system \eqref{eq:def_ZiN} such that $\mathbf{E}\left[Z_i\left(t\right)\right]<+\infty$ and $\mathbf{E}\left[\underline{Z}_i\left(t\right)\right]<+\infty$ for any $i \in  \llbracket 1, N \rrbracket$ and $ t\geq 0$.
For any $i \in  \llbracket 1, N \rrbracket$, we consider the total variation distance between $Z_i$ and $\underline{Z}_i$ on $\left[0,t\right]$:
$$ \Delta_i(t):=\int_0^t \vert d\left( Z_i(s) - \underline{Z}_i(s) \right) \vert.$$
$\Delta_i(t)$ counts the number of unshared jumps between $Z_i$ and $\underline{Z}_i$ on $\left[0,t\right]$. We denote respectively by $\lambda_i$ and $\underline{\lambda}_i$ the stochastic intensities of $Z_i$ and $\underline{Z}_i$. As they are constructed on the same Poisson measure $\pi_i$, the unshared jumps are the points of $\pi_i$ located between the two intensities, thus we have 
$$\Delta_i(t)=\int_0^t \int_0^{+\infty}  \left| \mathbf{1}_{\left\{z \leq \lambda_i(s)\right\}} -  \mathbf{1}_{\left\{z \leq \underline{\lambda}_i(s)\right\}} \right|  \pi_i\left(ds,dz\right).$$
Setting  $\delta_i(t):=\mathbf{E}\left[\Delta_i(t)\right]$, we obtain with Fubini's Theorem 
$$\delta_i(t)= \mathbf{E} \left[\int_0^t \int_0^{+\infty} \left|  \mathbf{1}_{\left\{z \leq \lambda_i(s)\right\}} -  \mathbf{1}_{\left\{z \leq \underline{\lambda}_i(s)\right\}} \right|  dzds\right] = \int_0^t \mathbf{E} \left[ \left|   \lambda_i(s) -  \underline{\lambda}_i(s)\right| \right] ds.$$
Using \eqref{eq:def_lambdaiN} and as $f$ is Lipschitz continuous (Hypothesis \ref{hyp:existence_Zin}), we have
\begin{align*}
\delta_i(t)&= \int_0^t \mathbf{E} \left[ \left|  f \left(u_0\left(s,x_i\right) + \dfrac{1}{N} \sum_{j=1}^N w_{ij} \int_{\left]0,s\right[} h\left(s-u\right) dZ_j^{(N)}\left(u\right) \right)\right.\right.\\
& \left.\left. -  f \left(u_0\left(s,x_i\right) + \dfrac{1}{N} \sum_{j=1}^N w_{ij} \int_{\left]0,s\right[} h\left(s-u\right) d\underline{Z}_j^{(N)}\left(u\right) \right)\right|  \right] ds\\
%&\leq \int_0^t \mathbf{E} \left[ L_f \left| \dfrac{1}{N} \sum_{j=1}^N w_{ij} \int_{\left]0,s\right[} h\left(s-u\right) \left( dZ_j^{(N)}\left(u\right) - d\underline{Z}_j^{(N)}\left(u\right) \right)\right|\right] ds \\
% & \leq L_f \dfrac{1}{N}  \sum_{j=1}^N \left| w_{ij} \right| \int_0^t \mathbf{E} \left[ \int_{\left]0,s\right[}  \left| h\left(s-u\right)\right| \left| dZ_j^{(N)}\left(u\right) - d\underline{Z}_j^{(N)}\left(u\right) \right|\right] ds \\
& \leq L_f \dfrac{1}{N}  \sum_{j=1}^N  w_{ij} \mathbf{E} \left[ \int_0^t  \int_{\left]0,s\right[}  \left| h\left(s-u\right)\right| d\Delta_j(u) ds \right].
\end{align*}
We apply Lemma \ref{lem:interversion_int} ($\Delta_i$ is with finite variations, $\Delta_i(0)=0$ and $h$ is locally integrable) and obtain
$$\delta_i(t)% \leq %L_f \dfrac{1}{N}  \sum_{j=1}^N \left|w_{ij} \right|  \mathbf{E} \left[ \int_0^t   \left| h\left(t-s\right)\right| \Delta_j(s) ds \right]\\
\leq L_f \dfrac{1}{N}  \sum_{j=1}^N  w_{ij}  \int_0^t   \left| h\left(t-s\right)\right|  \delta_j\left(s\right) ds.$$
We set $\delta(t)=\sum_{i=1}^N \delta_i(t)$ and $W_N=\max_{(i,j)\in \llbracket 1, N \rrbracket ^2} w_{ij}$. Then, summing on $i$, we have
$$\delta(t)% \leq  L_f \dfrac{1}{N}  \sum_{i=1}^N \sum_{j=1}^N W_N \int_0^t   \left| h\left(t-s\right)\right| \delta_j\left(s\right) ds\\
%& \leq  L_f ~W_N \dfrac{1}{N} \left(\sum_{i=1}^N 1\right) \sum_{j=1}^N \int_0^t  \left| h\left(t-s\right)\right|   \delta_j\left(s\right)  ds \\
%& \leq  L_f ~W_N  \int_0^t  \left| h\left(t-s\right)\right|   \left( \sum_{j=1}^N \delta_j\left(s\right) \right) ds \\
\leq  L_f ~W_N \int_0^t  \left| h\left(t-s\right)\right|  \delta\left(s\right)  ds.
$$
Since $h$ is locally integrable, $\delta$ is non-negative and locally bounded, we can apply Lemma \ref{lem:picard_gen} (i) and obtain that $\delta(t)=0$ for all $t\geq 0$. As each $\Delta_i$ is non-negative, we obtain that for all $i \in  \llbracket 1, N \rrbracket$ and $ t\geq 0$, $\Delta_i(t)=0$ almost surely. Hence $Z_i(t)=\underline{Z}_i(t)$ almost surely for all $i \in  \llbracket 1, N \rrbracket$ and $ t\geq 0$, which gives the uniqueness.
\medskip

We show now the existence of a process satisfying \eqref{eq:def_ZiN}. To do it, we proceed by iteration: for all $i \in  \llbracket 1, N \rrbracket$ and $ t\geq 0$, let $Z_{i,0}(t)=0$. Then, for all $n\geq 0$ we set:
$$Z_{i,n+1}(t)=\int_0^t \int_0^{+\infty} \mathbf{1}_{\left\{z\leq f\left(u_0\left(t,x_i\right) + \frac{1}{N}\sum_{j=1}^N w_{ij} \int_0^{s-} h\left(s-u\right) dZ_{j,n}\left(u\right) \right) \right\}}\pi_i(ds,dz).$$
With $i$ and $n$ fixed, such a process $\left(Z_{i,n+1}\right)$ exists: it is a counting process with stochastic intensity $\lambda_{i,n+1}(t)=f\left(u_0\left(t,x_i\right) + \frac{1}{N}\sum_{j=1}^N w_{ij} \int_0^{t-} h\left(t-u\right) dZ_{j,n}\left(u\right) \right)$. As for the uniqueness,  we set for all $i \in  \llbracket 1, N \rrbracket, n\geq 0$ and $t\geq 0$, $\delta_{i,n}(t) = \mathbf{E}\left[ \int_0^t \left|dZ_{i,n+1}(s)-dZ_{i,n}(s) \right| \right]$ and $\delta_{n}(t)= \sum_{i=1}^N \delta_{i,n}(t)$. As it was done previously, we find:
\begin{align*}
\delta_{i,n+1}(t) &= \mathbf{E}\left[ \int_0^t \left|dZ_{i,n+2}(s)-dZ_{i,n+1}(s) \right| \right] = \mathbf{E} \left[\int_0^t \int_0^{+\infty} \left|  \mathbf{1}_{\left\{z \leq \lambda_{i,n+2}(s)\right\}} -  \mathbf{1}_{\left\{z \leq \lambda_{i,n+1}(s)\right\}} \right|  dzds\right]\\
&\leq \int_0^t \mathbf{E} \left[ L_f \left| \dfrac{1}{N} \sum_{j=1}^N w_{ij} \int_{\left]0,s\right[} h\left(s-u\right) \left( dZ_{j,n+1}\left(u\right) - dZ_{j,n}\left(u\right) \right)\right|\right] ds.
\end{align*}
Summing on $i$ and using Lemma \ref{lem:interversion_int} we obtain
\begin{equation}\label{eq:ineg_delta1}
\delta_{n+1} (t) \leq L_f ~W_N   \int_0^t   \left| h\left(t-s\right)\right| \delta_{n}(s) ds.
\end{equation}
We want to apply Lemma \ref{lem:picard_gen}\textit{(ii)}, but for this we have to show that $\delta_n$ is locally bounded. We note $m_{i,n}(t) = \mathbf{E}\left[Z_{i,n}(t)\right]$ and $v_{n}(t)=\sum_{i=1}^N m_{i,n}(t)$. By construction,
$$m_{i,n+1}(t) = \mathbf{E} \left[ \int_0^t \int_0^{+\infty} \mathbf{1}_{\left\{z\leq f\left(u_0\left(s,x_i\right) + \frac{1}{N}\sum_{j=1}^N w_{ij} \int_0^{s-} h\left(s-u\right) dZ_{j,n}\left(u\right) \right) \right\}}\pi_i(ds,dz) \right].$$
As $\pi_i$ is a random Poisson measure with intensity $dsdz$, we have
$$m_{i,n+1}(t) = \mathbf{E} \left[ \int_0^t f\left(u_0\left(s,x_i\right) + \frac{1}{N}\sum_{j=1}^N w_{ij} \int_0^{s-} h\left(s-u\right) dZ_{j,n}\left(u\right) \right)ds \right].$$
By Hypothesis \ref{hyp:existence_Zin}, we have that $f(y)\leq f(0)+L_f\vert y \vert$ for all $y$ so that:  
\begin{align*}
m_{i,n+1}(t)%&\leq  \mathbf{E} \left[ \int_0^t \left( f(0) + L_f \left| u_0\left(s,x_i\right) + \frac{1}{N}\sum_{j=1}^Nw_{ij} \int_0^{s-} h\left(s-u\right) dZ_{j,n}\left(u\right) \right| \right) ds \right]\\
%&\leq f(0)t + L_f  \int_0^t \left| u_0\left(s,x_i\right)\right|ds  + \frac{1}{N}\sum_{j=1}^N \left| w_{ij}\right| \mathbf{E}\left[ \int_0^t \int_0^{s-} \left|h\left(s-u\right)\right| dZ_{j,n}\left(u\right)  ds \right]\\
&\leq f(0)t + L_f  \Vert u_0 \Vert_\infty t + \frac{1}{N}\sum_{j=1}^N w_{ij} \int_0^t \int_0^{s-} \left|h\left(s-u\right)\right| dm_{j,n}(u) ds.
\end{align*}
Applying Lemma \ref{lem:interversion_int} and summing on $i$ we obtain
\begin{equation}\label{eq:preuve_exi_un_PH1}
v_{n+1}(t)\leq Nt\left(f(0) + L_f  \Vert u_0 \Vert_\infty  \right) + W_N \int_0^t \left|h\left(t-s\right)\right| v_{n}(s)ds.
\end{equation}
As $v_0=0$ and $h$ is locally integrable, by induction $v_n$ is locally bounded for all $n\geq 0$. Yet $\delta_n(t) =  \sum_{i=1}^N \mathbf{E}\left[ \int_0^t \left|dZ_{i,n+1}(s)-dZ_{i,n}(s) \right| \right] \leq v_{n+1}(t) + v_{n}(t)$ hence $\delta_n$ is indeed locally bounded for all $n$. Lemma \ref{lem:picard_gen}\textit{(ii)} and \eqref{eq:ineg_delta1} give then that for all $T>0$, there exists $C_T$ such that $\sup_{t\in [0,T]} \sum_{n\geq 0} \delta_n(t) \leq C_T < \infty $. Thus we have
$$\sup_{t\in [0,T]} \sum_{n\geq 0} \sum_{i=1}^N \mathbf{E}\left[ \int_0^t \left|dZ_{i,n+1}(s)-dZ_{i,n}(s) \right| \right]  \leq C_T < \infty. $$
Thus for $i$ fixed, the sequence of random variables $\left(Z_{i,n}\right)_n$ is Cauchy in $L^1$ on the space $D([0,t],\mathbb{R})$ with the expectation of the total variation distance. Hence there exists a process $Z_i$ such that $\mathbf{E}\left[ \int_0^T \vert dZ_{i,n}(s) - dZ_i(s) \vert \right] \xrightarrow[n\to\infty]{}0$. From this convergence and a diagonal argument, there exists an extraction $\varphi$ such that for all $i$, $$ \int_0^T \left| dZ_{i,\varphi(n)}(s) - dZ_i(s)\right| \xrightarrow[n\to\infty]{}0.$$ Since $ \int_0^T \vert dZ_{i,\varphi(n)}(s) - dZ_i(s) \vert $  is an integer, $Z_{i,\varphi(n)}$ is a.s. stationary and one obtains from this that the right hand side of
\begin{equation}\label{eq:ex_PH_extraction}Z_{i,\varphi(n)+1}(t)=\int_0^t \int_0^{+\infty} \mathbf{1}_{\left\{z\leq f\left(u_0\left(t,x_i\right) + \frac{1}{N}\sum_{j=1}^N w_{ij} \int_0^{s-} h\left(s-u\right) dZ_{j,\varphi(n)}\left(u\right) \right) \right\}}\pi_i(ds,dz)
\end{equation}
is equal to $\displaystyle \int_0^t \int_0^{+\infty} \mathbf{1}_{\left\{z\leq f\left(u_0\left(t,x_i\right) + \frac{1}{N}\sum_{j=1}^N w\left(x_j,x_i\right) \int_0^{s-} h\left(s-u\right) dZ_{j}\left(u\right) \right) \right\}}\pi_i(ds,dz)$. Hence the left hand side of \eqref{eq:ex_PH_extraction} converges too, towards some $\widetilde{Z}_i(t)$. It remains to show that $\widetilde{Z}=Z$. 
Fatou's Lemma gives $$\mathbf{E}\left[\int_0^T \vert dZ_i(s)- d\widetilde{Z}_i(s)\vert \right]  \leq \liminf_{n\to\infty} \mathbf{E}\left[ \int_0^T \vert dZ_{i,\varphi(n)}(s)- dZ_{i,\varphi(n)+1}(s) \vert \right]  = 0$$ as $\left(Z_{i,n}\right)_n$ is a Cauchy sequence. We have then that the limit process verifies a.s. 
$$Z_{i}(t)=\int_0^t \int_0^{+\infty} \mathbf{1}_{\left\{z\leq f\left(u_0\left(t,x_i\right) + \frac{1}{N}\sum_{j=1}^N w_{ij} \int_0^{s-} h\left(s-u\right) dZ_{j}\left(u\right) \right) \right\}}\pi_i(ds,dz).$$
This gives the existence of multivariate Hawkes process $\left(Z_{1}(t),...,Z_{N}(t)\right)_{t\geq 0}$ satisfying \eqref{eq:def_ZiN}. Now let us verify that $t \mapsto \sup_{1\leq i \leq N} \mathbf{E}[Z_i(t)]$ is locally bounded. Recall \eqref{eq:preuve_exi_un_PH1}: as $v_n$ is locally bounded, by Lemma \ref{lem:picard_gen}\textit{(iii)} for all $T>0$, there exists $C_T$ such that $\sup_{t\in [0,T]} \sup_{n\geq 0} v_n(t) \leq C_T < + \infty	$ hence
$$\sup_{t\in [0,T]} \sup_{n\geq 0} \sum_{i=1}^N \mathbf{E}\left[Z_{i,n}(t)\right] \leq C_T < + \infty	$$
and by dominated convergence, for all $ T>0$, $\displaystyle\sup_{t\in [0,T]} \sum_{i=1}^N \mathbf{E}\left[Z_{i}(t)\right] < + \infty$ and the proof is concluded.\qed

\subsection{Proof of Theorem \ref{thm:existence_lambda}}\label{S:proof_thm:existence_lambda}

We show existence and uniqueness of a continuous and bounded solution $\lambda$ to equation \eqref{eq:def_lambdabarre}. We follow the proof proposed in \cite{CHEVALLIER20191} (Proposition 5), with major changes to accomodate our hypotheses.
We apply Banach fixed-point Theorem. We consider the map $F$ defined on $\mathcal{C}_b\left(\left[0,T\right]\times I,\mathbb{R}\right)$ (the set of bounded continuous functions defined on $\left[0,T\right]\times I$) by, for any $g \in \mathcal{C}_b\left(\left[0,T\right]\times I,\mathbb{R}\right)$:
$$F(g)(t,x)=f\left(u_0(t,x)+\int_{I} W(x,y)\int_0^t h(t-s) g(s,y)ds ~\nu(dy) \right) \text{ for all }  (t,x) \in \left[0,T\right]\times I.$$
First, we check that $F$ takes values in $\mathcal{C}_b\left(\left[0,T\right]\times I,\mathbb{R}\right)$: consider $g\in\mathcal{C}_b\left(\left[0,T\right]\times I,\mathbb{R}\right)$. Let us show that $F(g)$ is \textbf{bounded}. 
Fix $(t,x) \in \left[0,T\right]\times I$. As $f$ is Lipschitz continuous, we have:
\begin{align*}
 F(g)(t,x) %&\leq f(0) + L_f\left| u_0(t,x) +  \int_{\mathbb{R}^d} W(x,y)\int_0^t h(t-s) g(s,y)ds ~\nu(dy) \right|\\
 & \leq f(0) + L_f \Vert u_0 \Vert_\infty + L_f \int_{\mathbb{R}^d} \vert W(x,y) \vert \int_0^t \vert h(t-s) \vert g(s,y)ds ~\nu(dy) .
\end{align*}
As $g$ is bounded and $h$ is locally integrable by Hypothesis  \ref{hyp:existence_Zin}, we have
 $$\sup_{t\in[0,T],~x\in I} F(g)(t,x) \leq f(0) + L_f \Vert u_0 \Vert_\infty + L_f \Vert h \Vert_{[0,T],1} \Vert g \Vert _\infty \sup_{x\in I} \int_{\mathbb{R}^d}  W(x,y)   ~\nu(dy)<\infty,$$
where we used Hypothesis \eqref{eq:hyp_w_int_nu_y}.
\medskip

We check now that $F(g)$ is \textbf{continuous}. We show the sequential continuity: we fix $(t,x) \in [0,T]\times I$ and a sequence $(t_n,x_n)$ converging to $(t,x)$. As $f$ is Lipschitz continuous, we have: 
\begin{multline}
\vert F(g)(t_n,x_n) - F(g)(t,x)\vert  \leq  L_f \left| u_0(t_n,x_n)-u_0(t,x) \right| \\  + L_f \left| \int_{I} W(x_n,y) \int_0^{t_n} h(t_n-s) g(s,y)ds\nu(dy)- \int_{I} W(x,y) \int_0^{t} h(t-s) g(s,y)ds \nu(dy) \right|.
\end{multline}
The first term  $L_f \left| u_0(t_n,x_n)-u_0(t,x) \right| $ tends to $0$ when $n$ tends to infinity as $u_0$ is continuous in time and space by Hypothesis \ref{hyp:existence_Zin}. To show the convergence of the second term, we use the following bound: 
\begin{align*}
&\left| \int_I W(x_n,y) \int_0^{t_n} h(t_n-s) g(s,y)ds\nu(dy) -\int_{I} W(x,y) \int_0^{t} h(t-s) g(s,y)ds \nu(dy) \right|
\\
&\leq \left| \int_{I} \left( W(x_n,y)- W(x,y)\right)  \int_0^{t_n} h(t_n-s) g(s,y)ds\nu(dy)\right| \\
&+\left|\int_{I} W(x,y) \left(\int_0^{t_n}h(t_n-s)g(s,y)ds-\int_0^{t} h(t-s) g(s,y)ds \right)\nu(dy) \right|=:A_n+B_n
\end{align*} 
As $h$ is locally integrable, $g$ bounded, we can upper bound $A_n$ immediately:
$$A_n %&\leq \int_{I} \left| W(x_n,y) - W(x,y) \right| \int_0^{t_n}  \vert h(t_n-s)\vert  g(s,y)ds\nu(dy) 
\leq \Vert h \Vert_{[0,T],1} \Vert g \Vert _\infty  \int_{I} \left| W(x_n,y) - W(x,y) \right| \nu(dy)
\leq \Vert h \Vert_{[0,T],1} \Vert g \Vert _\infty  C_w \Vert x - x_n\Vert^\vartheta\xrightarrow[n\to\infty]{} 0,
$$ using \eqref{eq:hyp_W_pseudolip_theta}.
To study the convergence of $B_n$, we do a substitution and split the integral in two:
\begin{align*}
B_n%&\leq\int \vert W(x,y)  \vert \left| \int_0^{t_n} h(t_n-s)g(s,y)ds - \int_0^t h(t-s)g(s,y)ds \right| \nu(dy)\\
&\leq\int \vert W(x,y) \vert \left| \int_0^{t_n} h(u)g(t_n-u,y)du  - \int_0^t h(u)g(t-u,y)ds \right| \nu(dy)\\
&\leq \int \vert W(x,y)) \vert \left( \int_0^t \vert h(u) \vert \left| g(t_n-u,y)-g(t-u,y)\right|du\right) \nu(dy) \\
&+ \int \vert W(x,y) \vert \left( \int_{\min(t,t_n)}^{\max(t,t_n)} \vert h(u) \vert  g(t_n-u,y)du\right)  \nu(dy) =: a_n+b_n.
 \end{align*}
Since $g$ is continuous, for all $y\in I$ we have $\int_0^t \vert h(u) \vert \left| g(t_n-u,y)-g(t-u,y)\right|du \xrightarrow[n\to\infty]{} 0$, and since $\int_0^t  \vert h(u) \vert \left| g(t_n-u,y)-g(t-u,y)\right|du \leq 2 \Vert h \Vert_{[0,T],1} \Vert g \vert _\infty $, we see from dominated convergence theorem that $a_n \xrightarrow[n\to\infty]{} 0$.
We focus on the term $b_n$
\begin{align*}
\int_{\min(t,t_n)}^{\max(t,t_n)} \vert h(u) \vert  g(t_n-u,y)du %& = \int_0^T \vert h(u) \vert g(t_n-u,y) \mathbf{1}_{\left[\min(t,t_n),\max(t,t_n)\right]}(u) du \\
& \leq \Vert g \Vert_{\infty} \int_0^T \vert h(u) \vert \mathbf{1}_{\left[\min(t,t_n),\max(t,t_n)\right]}(u)du .
\end{align*}
Yet $ \vert h(u) \vert \mathbf{1}_{\left[\min(t,t_n),\max(t,t_n)\right]}(u) \xrightarrow[n\to\infty]{} 0$, and we obtain $b_n\xrightarrow[n\to\infty]{} 0$ by dominated convergence as $h$ is locally integrable. We have shown that for all $\left(t,x\right)\in [0,T]\times I$, $ \lim_{n\to\infty} \vert F(g)(t_n,x_n) - F(g)(t,x) \vert =0$ for any sequence $(t_n,x_n)$ tending to $\left(t,x\right)$: $F(g)$ is continuous.
\medskip

We show now that there exists a constant $C>0$ such that for all $(t,x,z) \in \left[0,T\right]\times I^2$:
\begin{equation}\label{eq:Fg_lips_espace}
\vert F(g)(t,x) - F(g)(t,z) \vert \leq C \left( \Vert x-z \Vert +\Vert x-z \Vert^\vartheta\right).
\end{equation}
Let $(t,x,z) \in \left[0,T\right]\times I \times I$. As done previously ($f$ and $u_0$ are Lipschitz continuous), we have:
\begin{align*}
\vert F(g)(t,x) - F(g)(t,z)\vert %&\leq L_f \left| u_0(t,x)-u_0(t,z) + \int_{I} \int_0^t h(t-s) g(s,y)ds\left(W(x,y)-W(z,y)\right)\nu(dy) \right|\\
&\leq L_f L_{u_0} \Vert x-z \Vert + L_f \int_{I} \int_0^t \vert h(t-s) \vert  g(s,y)   ds~\vert W(x,y)-W(z,y)\vert \nu(dy).
\end{align*}
Since $g$ is bounded, $h$ is locally integrable, using \eqref{eq:hyp_W_pseudolip_theta}
 \begin{align*}
\vert F(g)(t,x) - F(g)(t,z)\vert %&\leq L_f L_{u_0} \Vert x-z \Vert + L_f \Vert g \Vert_\infty  \Vert h \Vert_{[0,T],1} \int_{I} \vert W(x,y)-W(z,y) \vert \nu(dy)\\
& \leq L_f L_{u_0} \Vert x-z \Vert + L_f \Vert g \Vert_\infty  \Vert h \Vert_{[0,T],1} C_w \Vert x - z \Vert^\vartheta,
\end{align*} 
which gives \eqref{eq:Fg_lips_espace} .
\medskip

Hence, $\mathcal{C}_b\left(\left[0,T\right]\times I,\mathbb{R}\right)$ is stable by $F$. We prove that $F$ admits a unique fixed point, which is $\lambda$ satisfying \eqref{eq:def_lambdabarre}. To do it, we show that some iteration of $F$ is contractive, and the Banach fixed-point Theorem gives the result. Let $t\in[0,T]$, $g$ and $\tilde{g}$ be two functions in $\mathcal{C}_b\left(\left[0,t\right]\times I,\mathbb{R}\right)$. We use the distance $D_t(g,\tilde{g}):=\sup_{s\in[0,t]} \sup_{x\in I} \vert g(s,x) - \tilde{g}(s,x)\vert$ which makes the space $\mathcal{C}_b\left(\left[0,t\right]\times I,\mathbb{R}\right)$ complete. Obviously, for all $s\leq t$, $D_s(g,\tilde{g})\leq D_t(g,\tilde{g})$. 
Let $x\in\mathbb{R}^d$. As previously,
\begin{align*}
\vert F(g)(t,x)-F(\tilde{g})(t,x)\vert%&\leq L_f \left| \int_{I} W(x,y) \int_0^t h(t-s) \left( g(s,y)-\tilde{g}(s,y)\right) ds ~\nu(dy) \right|%\\
%&\leq L_f \int_{I} \vert W(x,y) \vert \int_0^t \vert h(t-s) \vert ~\vert  g(s,y)-\tilde{g}(s,y) \vert ds ~\nu(dy)\\
%& \leq L_f \int_{I} \vert W(x,y) \vert \int_0^t \vert h(t-s) \vert \sup_{z\in I} \vert  g(s,z)-\tilde{g}(s,z) \vert ds ~\nu(dy) \\
%&\leq  L_f \int_{I} \vert W(x,y) \vert \nu(dy)  \int_0^t \vert h(t-s) \vert \sup_{z\in I} \vert  g(s,z)-\tilde{g}(s,z) \vert ds \\
&\leq L_f \left( \sup_{z\in I} \int_{I} \vert W(z,y)\vert \nu(dy) \right)\int_0^t \vert h(t-s) \vert D_s(g,\tilde{g}) ds.
\end{align*}
Using Cauchy-Schwarz inequality, as $h$ is in $L^2_{loc}$ under Hypothesis \ref{hyp:existence_Zin}
\begin{align*}
\vert F(g)(t,x)-F(\tilde{g})(t,x)\vert & \leq L_f \left( \sup_{z\in I} \int_{I} \vert W(z,y)\vert \nu(dy) \right) \Vert h \Vert_{[0,T],2} \left( \int_0^t \left(D_s (g,\tilde{g})\right)^2 ds \right)^{\frac{1}{2}}.
\end{align*}
Using \eqref{eq:hyp_w_int_nu_y}, we have then shown the existence of a constant $C(f,w,\nu,h,T,p)$ such that for all mappings $g$ and $\tilde{g}$, for all $t\in[0,T]$:
\begin{equation}\label{eq:pt_fixe_distance_contracte}
D_t(F(g),F(\tilde{g}))\leq C \left( \int_0^t \left(D_s(g,\tilde{g})\right)^2ds\right)^{\frac{1}{2}}.
\end{equation}
By induction on $k \in \mathbb{N}$, with \eqref{eq:pt_fixe_distance_contracte}, we show that for all $t\in [0,T]$ and for any mappings $g$ and $\tilde{g}$: $D_t(F^k(g),F^k(\tilde{g}))\leq C^k \left( \dfrac{t^k}{k!} \right) ^{\frac{1}{2}}  D_t(g,\tilde{g})$. The initialisation is immediate, and then for $k\geq 0$, using (\ref{eq:pt_fixe_distance_contracte}) and the induction hypothesis
\begin{align*}
D_t(F^{k+1}(g),F^{k+1}(\tilde{g})) & \leq C \left( \int_0^t \left(D_s(F^k(g),F^k(\tilde{g}))\right)^2ds\right)^{\frac{1}{2}} \\
& \leq C \left( \int_0^t C^{2k} \dfrac{s^k}{k!} D_s(g,\tilde{g})^2 ds \right)^{\frac{1}{2}}  \leq C^{k+1} \left(\dfrac{t^{k+1}}{(k+1)!}\right)^{\frac{1}{2}} D_t(g,\tilde{g}),
\end{align*}
which concludes the induction. We have then for all $k$ and any functions $g$ and $\tilde{g}$ of $\mathcal{C}_b\left(\left[0,T\right]\times I,\mathbb{R}_+\right)$, the $k$-th iteration of $F$ verifies $D_T(F^{k}(g),F^{k}(\tilde{g})) \leq C^k \left( \dfrac{T^k}{k!} \right) ^{\frac{1}{2}}  D_T(g,\tilde{g})$. Hence there exists a rank $k$ such that $F^k$ is contractive, thus has a unique fixed point which is also then the unique fixed point of $F$ in $\mathcal{C}_b\left(\left[0,T\right]\times I,\mathbb{R}\right)$ that we call $\lambda$. Furthermore, we have shown that any image by $F$ verifies the property   \eqref{eq:Fg_lips_espace}, so in particular $\lambda$ verifies it too and \eqref{eq:pt_fixe_lips_espace} is then true (with $C_\lambda$ the constant of equation \eqref{eq:Fg_lips_espace} for $g=\lambda$). Note that such a $ \lambda$ is necessarily nonnegative, as the iterative map $F$ preserves positivity in both cases $f\geq0$ and $f(x)=x$ with $u_{ 0},h\geq0$.
\medskip

We focus now on the second part of Theorem \ref{thm:existence_lambda}: we consider $u_0$ continuously differentiable in time and $\dfrac{\partial u_0}{\partial t}$ bounded on $[0,T]\times I$, $h$ continuous and piecewise continuously differentiable, and $f(x)=x$. 
First, we ensure that \eqref{eq:pt_fixe_derivee} admits a unique continuous bounded solution. Then, by studying a sequence of functions that converges towards $\lambda$, we show that $\lambda$ is differentiable in time and $\dfrac{\partial \lambda}{\partial t}$  satisfies \eqref{eq:pt_fixe_derivee}.
Using the same method as above, we show that the map $G$ defined on $\mathcal{C}_b\left(\left[0,T\right]\times I,\mathbb{R}\right)$ by%, for all $(t,x) \in \left[0,T\right]\times I$
$$ G(g)(t,x)=\dfrac{\partial u_0}{\partial t} (t,x) + h(t) \int_I W(x,y) \lambda(0,y) \nu(dy) + \int_I \int_0^t h(t-s) W(x,y) g (s,y) \nu(dy) ds$$
admits a unique fixed point called $\mu$.
Moreover, we can introduce a sequence of function $\left(\mu_n\right)_n$ that converges uniformly towards $\mu$ defined by iteration, $\mu_0=0$ and
$$ \mu_{n+1}(t,x) := \dfrac{\partial u_0}{\partial t} (t,x) + h(t) \int_I W(x,y) \lambda(0,y) \nu(dy) + \int_I \int_0^t h(t-s) W(x,y) \mu_n (s,y) \nu(dy) ds.$$
Similarly, we introduce a sequence of function $\left(\lambda_n\right)_n$ that converges uniformly towards $\lambda$ defined by iteration, $\lambda_0=0$ and $ \lambda_{n+1}(t,x):= u_0(t,x) + \int_I \int_0^t h(t-s) W(x,y) \lambda_n (s,y) \nu(dy) ds$. By induction, for every $n$, $\lambda_n$ is differentiable in time and bounded and then, by integration by parts we obtain  
\begin{equation}\label{eq:lambda_derivee_suite}
\dfrac{\partial \lambda_{n+1}}{\partial t}(t,x) 
=  \dfrac{\partial u_0}{\partial t} (t,x) + h(t)  \int_I  W(x,y) \lambda_n (0,y) \nu(dy) + \int_0^t \int_I W(x,y) h(t-s) \dfrac{\partial \lambda_n}{\partial s}(s,y) \nu(dy) ds.
\end{equation} 
Now, we can compare $\mu_n$ and $\dfrac{\partial \lambda_{n}}{\partial t}$: setting $ \varpi_{ n}(t,x):= \mu_{n}(t,x) - \dfrac{\partial \lambda_{n}}{\partial t}(t,x)$ for all $n$, for any $(t,x)\in [0,T]\times I$,
$$\left| \varpi_{n+1}(t,x)\right|= \left| h(t)  \int_I  W(x,y)\left( \lambda(0,y)- \lambda_n (0,y)\right) \nu(dy)+ \int_0^t \int_I W(x,y) h(t-s) \varpi_n(s,y) \nu(dy) ds\right|,$$
so that $\left\Vert \varpi_{ n+1}(t, \cdot) \right\Vert_\infty\leq h(t) C_W^{(1)} \Vert \lambda(0,\cdot) - \lambda_n(0,\cdot) \Vert_\infty + C_W^{(1)} \int_0^t h(t-s) \left\Vert  \varpi_{ n}(s,\cdot) \right\Vert_\infty ds$. We obtain, as $\left(\lambda_n\right)$ converges uniformly to $\lambda$, $\limsup_{n\to\infty} \left\Vert \varpi_{ n+1}(t,\cdot)\right\Vert_\infty \leq  C_W^{(1)} \int_0^t h(t-s) \limsup_{n\to\infty} \left\Vert \varpi_{ n}(s,\cdot)\right\Vert_\infty ds$. This gives from Lemma \ref{lem:picard_gen} \textit{(i)}, provided that one has verified that $\displaystyle\limsup_{n\to\infty} \left\Vert \varpi_{ n}(s,\cdot)\right\Vert_\infty<+\infty$ is finite, that $\displaystyle\sup_{t\in [0,T]} \lim\sup_{n\to\infty} \left\Vert \varpi_n(t,\cdot) \right\Vert_\infty =0$. It implies that, as $(\mu_n)$ converges uniformly to $\mu$, so does $\left( \dfrac{\partial\lambda_{n}}{\partial t} \right)_n$, and then as $\lambda$ is differentiable, $\lambda_n \xrightarrow[n\to\infty]{uniformly} \lambda$ and $\dfrac{\partial\lambda_{n}}{\partial t}  \xrightarrow[n\to\infty]{uniformly} \mu$, we obtain $\dfrac{\partial \lambda}{\partial t}=\mu$. It remains to check that $\displaystyle\limsup_{n\to\infty} \left\Vert \varpi_{ n}(s,\cdot)\right\Vert_\infty  $ is finite. As $\left(\lambda_n\right)$ converges to $\lambda$, it is uniformly bounded and as $\dfrac{\partial u_0}{\partial t}$ is bounded, we can find $g$ locally bounded such that, from \eqref{eq:lambda_derivee_suite}, $ \left\Vert \partial_{ t}\lambda_{n+1}(t,\cdot)\right\Vert_\infty \leq g(t)+ C_W^{(1)} \int_0^t  h(t-s) \left\Vert\partial_{ s} \lambda_n(s,\cdot)\right\Vert_\infty  ds$. Lemma \ref{lem:picard_gen} \textit{(iii)} gives then that $\displaystyle \sup_{s\in [0,T]} \sup_{n\geq 0} \left\Vert \partial_{ s}\lambda_n(s,\cdot)  \right\Vert_\infty <\infty$. We can do the same for $\left(\mu_n\right)$ and obtain $\displaystyle \sup_{s\in [0,T]} \sup_{n\geq 0} \left|\left|\mu_n(s,\cdot) \right|\right|_\infty <\infty$ which concludes to $\displaystyle\limsup_{n\to\infty} \left\Vert  \varpi_n(s,\cdot) \right\Vert_\infty  <\infty$ for any $s\in [0,T]$. \qed
%----------------------------------------------------------------------------------------
%----------------------------------------------------------------------------------------
%----------------------------------------------------------------------------------------
\section{Proofs: Convergence of the mean-field process}
%----------------------------------------------------------------------------------------
%----------------------------------------------------------------------------------------

\subsection{Toolbox}
We present useful results that come up in the main proofs.
\begin{prop}\label{prop:toolbox_Xij} Recall the definitions of $\kappa_N$ and $w_N$ in Hypothesis \ref{hyp:conv_graph_concentration}. Let $\left(\alpha_{ij}\right)$ and $\left(\alpha_{ijk}\right)$ such that for every $(i,j,l)\in\llbracket 1,N \rrbracket ^3$, $\vert\alpha_{lj}\vert\leq 1$ and $ \vert\alpha_{ijl}\vert \leq 1$. Define 
$$X_j:= \dfrac{\kappa_N}{N} \sum_{l=1}^N \alpha_{lj} \overline{\xi}_{lj}, \quad \widetilde{X}_i:= \dfrac{\kappa_N}{N} \sum_{l=1}^N \alpha_{il} \overline{\xi}_{il}, \quad
X_{ij}:= \dfrac{\kappa_N}{N} \sum_{l=1}^N \alpha_{ijl} \overline{\xi}_{il},$$
with $\overline{\xi}_{lk}:=\xi_{lk}^{(N)}-W_N(x_l,x_k)$. Then, under Hypothesis \ref{hyp:conv_graph_concentration}, $\mathbb{P}$-almost surely if N is large enough:
$$\sup_{1\leq j \leq N} \vert X_j \vert \leq \varepsilon_N, \quad
\sup_{1\leq i \leq N} \vert \widetilde{X}_i \vert \leq \varepsilon_N \text{ and }\sup_{1\leq i,j \leq N} \vert X_{ij} \vert \leq \varepsilon_N,$$
for $\varepsilon_N:= 32 \dfrac{\kappa_N^2 w_N}{N}\log(N)$.
\end{prop}
Note that under Hypothesis \ref{hyp:conv_graph_concentration}, $\varepsilon_N \xrightarrow[N\to\infty]{} 0$.

\begin{proof} We rely on  Lemma \ref{lem:inegalit_concentration}.
We derive a uniform bound on $\left(X_{j}\right)_{j\in\llbracket 1, N\rrbracket}$: fixing $j$, we apply Lemma \ref{lem:inegalit_concentration} for the choice $U_l=\xi_{lj}^{(N)}, p_l=W_N(x_l,x_j)$ (note that \eqref{eq:w_n} yields that $p_l\leq w_N$), $v_l=\alpha_{lj}$  and the constant $\kappa_N>0$. We obtain, taking the supremum on $j$ and a union bound:
$$\mathbb{P}\left( \sup_{j\in\llbracket 1,N\rrbracket} \vert X_{j} \vert > \varepsilon_N\right) \leq 2N \exp \left(-16 \log(N) B\left( 4\sqrt{2}\left( \dfrac{\log(N)}{Nw_N}\right)^\frac{1}{2}\right)\right).$$
%(the probability $\mathbb{P}$ is taken on the realization of $\left(\xi^{(N)}\right)$).

As $B(u)=u^{-2}\left( \left( 1+u \right) \log \left( 1+u \right) - u \right)\to\dfrac{1}{2}$ when $u\to 0$  and $\displaystyle \dfrac{\log(N)}{Nw_N} \leq \dfrac{\log(N)}{N}\kappa_N^2 w_N \to 0$ when $N\to\infty$ using \eqref{eq:kappaw_n1} and \eqref{eq:kappaw_nlog}, we can choose a deterministic $p$ such that for all $N\geq p$, $B\left( 4\sqrt{2}\left( \dfrac{\log(N)}{Nw_N}\right)^\frac{1}{2}\right) \geq \dfrac{3}{16}$. We then have if $N\geq p$: $\mathbb{P}\left( \sup_{j\in\llbracket 1,N\rrbracket} \vert X_{j} \vert > \varepsilon_N\right)  \leq 2N \exp \left(-3 \log(N) \right)= \dfrac{2}{N^2}$.
Hence, by Borel-Cantelli Lemma, there exists $\widetilde{\mathcal{O}}\in\mathcal{F}$ such that $\mathbb{P}(\widetilde{\mathcal{O}})=1$ and on $\widetilde{\mathcal{O}}$, there exists $\widetilde{N}<\infty$ such that if $N\geq \widetilde{N}$, $\sup_{j\in\llbracket 1,N\rrbracket} \vert X_{j} \vert  \leq \varepsilon_N$. We can show similarly that $\sup_{1\leq i \leq N} \vert \widetilde{X}_i \vert \leq \varepsilon_N$. To show the result on $\left(X_{ij}\right)$, we use the same Lemma \ref{lem:inegalit_concentration} but we need to lower-bound $B\left( 4\sqrt{2}\left( \dfrac{\log(N)}{Nw_N}\right)^\frac{1}{2}\right)$ differently: we can choose a deterministic $\tilde{p}$ for all $N\geq \tilde{p}$, $B\left( 4\sqrt{2}\left( \dfrac{\log(N)}{Nw_N}\right)^\frac{1}{2}\right) \geq \dfrac{1}{4}$ and then the same argument as before works to obtain
$\mathbb{P}\left( \sup_{i,j\in\llbracket 1,N\rrbracket} \vert X_{ij} \vert > \varepsilon_N\right) \leq 2N^2 \exp \left(-16 \log(N) \dfrac{1}{4}\right)\leq\dfrac{2}{N^2},$ and we conclude by Borel-Cantelli Lemma.
\end{proof}

\begin{cor}
\label{prop:estimees_IC}
Under Hypothesis \ref{hyp:conv_graph_concentration}, we have $\mathbb{P}$-almost surely if $N$ is large enough: 
\begin{equation}\label{eq:estimees_IC_sup-j_som-i}
\sup_{1 \leq j \leq N}  \left( \sum_{i=1}^N \dfrac{\kappa_i^{(N)}}{N}\xi_{ij}^{(N)}\right) \leq 1 + \sup_{1 \leq j \leq N}  \left( \sum_{i=1}^N \dfrac{\kappa_i^{(N)}}{N}W_N(x_i,x_j)\right)
\end{equation}
\begin{equation}\label{eq:estimees_IC_sup-i_som-j}
\sup_{1 \leq i \leq N}  \left( \sum_{j=1}^N \dfrac{\kappa_i^{(N)}}{N}\xi_{ij}^{(N)}\right) \leq 1 + \sup_{1 \leq i \leq N}  \left( \sum_{j=1}^N \dfrac{\kappa_i^{(N)}}{N}W_N(x_i,x_j)\right)
\end{equation}
\begin{equation}\label{eq:estimees_IC_som-ij_carre}
\dfrac{1}{N^3}\sum_{i,j=1}^N \left(\kappa_i^{(N)}\right)^2 \xi_{ij}^{(N)} \leq \dfrac{\kappa_N}{N}\left(1+\sum_{i,j=1}^N \dfrac{\kappa_i^{(N)}}{N^2}W_N(x_i,x_j)\right).
\end{equation}
\end{cor}
\begin{proof}
It is a direct application of Proposition \ref{prop:toolbox_Xij} (as for any $i$, $\left|\dfrac{\kappa_i^{(N)}}{\kappa_N}\right|\leq 1$ with \eqref{eq:kappa}), as  $\varepsilon_N$ defined in Proposition  \ref{prop:toolbox_Xij} tends to  0 as $N \to \infty$ under Hypothesis \ref{hyp:conv_graph_concentration}, hence for $N$ large enough $\varepsilon_N \leq 1$.
For instance for \eqref{eq:estimees_IC_som-ij_carre}, with $\displaystyle X_j=\dfrac{\kappa_N}{N}\sum_{i=1}^N \left( \dfrac{\kappa_i^{(N)}}{\kappa_N}\right)^2  \left(\xi_{ij}^{(N)}-W_N(x_i,x_j)\right)$ we have 
$$\dfrac{1}{N^3}\sum_{i,j=1}^N \left(\kappa_i^{(N)}\right)^2 \xi_{ij}^{(N)} \leq \dfrac{\kappa_N}{N^2} \sum_{j=1}^N X_j +  \dfrac{\kappa_N}{N^3}\sum_{i,j=1}^N \kappa_i^{(N)} W_N(x_i,x_j)
\leq \dfrac{\kappa_N}{N}\left( \varepsilon_N +  C_W\right),$$
where we used Proposition \ref{prop:toolbox_Xij} and \eqref{eq:control_i}.
\end{proof}
We introduce the following auxiliary graph.
\begin{deff}\label{def:graphs_G2}
We denote by $\mathcal{G}_N^{(2)}$ the directed weighted graph with vertices $\left\{ 1, \cdots,N\right\}$ such that every edge $j\rightarrow i$ is present, and with weight $W (x_i,x_j)$.
\end{deff}
The proof of the following technical Proposition is postponed in Section \ref{S:proof_scenarios}.
\begin{prop}\label{prop:scenario_hyp}
Under the Scenarios of Definition \ref{def:scenarios},
\begin{equation}\label{eq:lem_S_box}
d_{\Box,\nu}\left( W^{\mathcal{G}_N^{(2)}},W \right) \xrightarrow[N\to \infty]{} 0,
\end{equation}
\begin{equation}\label{eq:lem_S_box_max}
\Vert W^{\mathcal{G}_N^{(2)}}-W \Vert_{\infty	\to \infty, \nu} \xrightarrow[N\to\infty]{} 0,
\end{equation} 
\begin{equation}\label{eq:lem_S_A5_max}
\max_{1\leq i \leq N} \int_0^T \left| \int_I W(x_i,x)\gamma(s,x) \left( \nu^{(N)}(dx) - \nu(dx) \right) \right| ds \xrightarrow[N\to \infty]{} 0
\end{equation} and
\begin{equation}\label{eq:lem_S_A5}
\dfrac{1}{N}\sum_{i=1}^N \int_0^T \left| \int_I W(x_i,x)\gamma(s,x) \left( \nu^{(N)}(dx) - \nu(dx) \right) \right| ds \xrightarrow[N\to \infty]{} 0,
\end{equation}
where $\gamma(s,x):=\int_0^{s} h(s-u) \lambda(u,x)du$.
\end{prop}
If \eqref{eq:lem_S_box} and \eqref{eq:lem_S_A5} are satisfied in another configuration of positions than in the Scenarios of Definition \ref{def:scenarios}, Theorem \ref{thm:cvg_0} still applies. Likewise, if \eqref{eq:lem_S_box_max} and \eqref{eq:lem_S_A5_max} are satisfied, Theorem \ref{thm:cvg-sup_0} still applies.
%----------------------------------------------------------------------------------------
%----------------------------------------------------------------------------------------

\subsection{Proof of Theorem \ref{thm:cvg_0}} \label{S:proof_thm:cvg_0}

Recall the definitions of $Z_i^{(N)}$ and $\overline{Z}_i$ in \eqref{eq:def_ZiN} and \eqref{eq:construction_lim_Zbarre_i}. We remind that we consider the sequences $\left(\underline{x}_N\right)_{N\geq 1}$ and $\left( \xi^{(N)}_{ij}\right)_{\substack{N\geq 1 \\  i,j \in \llbracket 1,N \rrbracket}}$ fixed (our result is quenched). Let $t\in [0,T]$. For each $i \in \llbracket 1, N \rrbracket$, let $\Delta_i^{(N)}(t)$ be the total variation distance between $Z_i^{(N)}$ and $\overline{Z}_i$ on $[0,t]$:
\begin{equation}\label{eq:def_delta_i}
\Delta_i^{(N)}(t)=\int_0^t \left| d\left( Z_i^{(N)}(s) - \overline{Z}_i(s) \right) \right|.
\end{equation}
Remark that we always have $\sup_{t\in [0,T]} \left| Z_i^{(N)}(t) - \overline{Z}_i(t) \right| \leq \Delta_i^{(N)}(T)$. Taking the expectation, we have
$$ \mathbf{E}\left[\Delta_i^{(N)}(t)\right] = \mathbf{E}\left[\int_0^t \int_0^\infty \vert \mathbf{1}_{\left\{ z \leq \lambda_i^{(N)}(s)\right\}} -  \mathbf{1}_{\left\{ z \leq \lambda(s,x_i)\right\}} \vert \pi_i(ds,dz)\right]
= \int_0^t \mathbf{E}\left[\left|\lambda_i^{(N)}(s)- \lambda(s,x_i)\right|\right]ds.$$
Using the Lipschitz continuity of $f$ and recalling the definition of $ \lambda_{ i}^{ (N)}$ in \eqref{eq:def_lambdaiN} and of $ \lambda$ in \eqref{eq:def_lambdabarre}, we obtain:
% $\displaystyle \lambda_i^{(N)}(s)-\lambda(s,x_i)=\dfrac{\kappa_i^{(N)}}{N}\sum_{j=1}^N \xi_{ij}^{(N)}\int_0^{s-} h(s-u) dZ_j^{(N)}(u) -\int_{I} W(x_i,y)\int_0^{s-} h(s-u) \lambda(u,y)du\nu(dy)$
\begin{align}\label{eq:delta_i_majoration}
 \mathbf{E}\left[\Delta_i^{(N)}(t)\right] 
&\leq L_f \left( \sum_{k=1}^5 A^{(N)}_{i,t,k} \right),
\end{align}
\begin{equation}\label{eq:def_A1}
\text{where }A^{(N)}_{i,t,1}:= \int_0^t \mathbf{E} \left[ \left | \dfrac{\kappa_i^{(N)}}{N}\sum_{j=1}^N \xi_{ij}^{(N)}\int_0^{s-} h(s-u) \left(dZ_j^{(N)}(u)-d\overline{Z}_j(u)\right)\right|\right]ds,
\end{equation}
\begin{equation}\label{eq:def_A2}
A^{(N)}_{i,t,2}:= \int_0^t \mathbf{E} \left[ \left | \dfrac{\kappa_i^{(N)}}{N}\sum_{j=1}^N \xi_{ij}^{(N)}\int_0^{s-} h(s-u) \left(d\overline{Z}_j(u)-\lambda(u,x_j)du\right)\right|\right]ds,
\end{equation}
\begin{equation}\label{eq:def_A3}
A^{(N)}_{i,t,3}:=  \int_0^t\left | \dfrac{\kappa_i^{(N)}}{N}\sum_{j=1}^N \left(\xi_{ij}^{(N)}-W_N(x_i,x_j)\right)\int_0^{s} h(s-u) \lambda(u,x_j)du\right|ds,
\end{equation}
\begin{equation}\label{eq:def_A4}
A^{(N)}_{i,t,4}:=  \int_0^t  \left | \dfrac{1}{N}\sum_{j=1}^N \left(\kappa_i^{(N)}W_N(x_i,x_j)-W(x_i,x_j)\right)\int_0^{s} h(s-u) \lambda(u,x_j)du\right|ds \text{ and}
\end{equation}
\begin{multline}\label{eq:def_A5}
A^{(N)}_{i,t,5}:=  \int_0^t \left | \dfrac{1}{N}\sum_{j=1}^N W(x_i,x_j)\int_0^{s} h(s-u) \lambda(u,x_j)du\right.\\\left. - \int_{I} W(x_i,y)\int_0^{s} h(s-u) \lambda(u,y)du~\nu(dy) \right|ds.
\end{multline}
We are going to control each term $\frac{1}{N} \sum_{i=1}^N A^{(N)}_{i,t,k}$. The term $A^{(N)}_{i,t,1}$ captures the proximity between the particle system $Z_i^{(N)}$ with its meanfield counterpart $\overline{Z}_i$ at the same position. We have, as the graph $\left(\xi^{(N)}\right)$ is fixed,
\begin{align*}
A^{(N)}_{i,t,1} &= \int_0^t \mathbf{E} \left[ \left | \dfrac{\kappa_i^{(N)}}{N}\sum_{j=1}^N \xi_{ij}^{(N)}\int_0^{s-} h(s-u) \left(dZ_j^{(N)}(u)-d\overline{Z}_j(u)\right)\right|\right]ds\\
%&\leq \dfrac{1}{N} \sum_{j=1}^N \kappa_i^{(N)}\xi_{ij}^{(N)} \mathbf{E} \left[ \int_0^t\int_0^{s-}  \vert h(s-u) \vert \left|dZ_j^{(N)}(u)-d\overline{Z}_j(u)\right| ds \right]\\
&\leq \dfrac{1}{N} \sum_{j=1}^N \kappa_i^{(N)}\xi_{ij}^{(N)} \mathbf{E} \left[ \int_0^t\int_0^{s-}  \vert h(s-u) \vert \left|d\left( \Delta_j^{(N)}(u)\right)\right| ds \right].
\end{align*}
 We use Lemma \ref{lem:interversion_int} so that
\begin{equation}\label{eq:A_i,1}
A^{(N)}_{i,t,1} \leq \dfrac{1}{N} \sum_{j=1}^N \kappa_i^{(N)}\xi_{ij}^{(N)} \mathbf{E} \left[ \int_0^t  \vert h(t-s) \vert  \Delta_j^{(N)}(s) ds \right],
\end{equation}
then we have, by summation:
\begin{align*}
\dfrac{1}{N} \sum_{i=1}^N A^{(N)}_{i,t,1} &\leq \dfrac{1}{N}   \sum_{j=1}^N  \left( \sum_{i=1}^N \dfrac{\kappa_i^{(N)}}{N}\xi_{ij}^{(N)}\right)   \int_0^t  \vert h(t-s) \vert \mathbf{E} \left[ \Delta_j^{(N)}(s)\right] ds  \\
& \leq \sup_{1 \leq j \leq N}  \left( \sum_{i=1}^N \dfrac{\kappa_i^{(N)}}{N}\xi_{ij}^{(N)}\right)  \int_0^t  \vert h(t-s) \vert \mathbf{E} \left[\dfrac{1}{N}   \sum_{j=1}^N \Delta_j^{(N)}(s)\right] ds.
\end{align*}
We use \eqref{eq:estimees_IC_sup-j_som-i} and \eqref{eq:control_j} to obtain $\mathbb{P}$-almost surely if $N$ is large enough the bound
\begin{equation}\label{eq:moyA_i,1}
\dfrac{1}{N} \sum_{i=1}^N A^{(N)}_{i,t,1} \leq \left(1 + C_W\right) \int_0^t  \vert h(t-s) \vert \mathbf{E} \left[\dfrac{1}{N}   \sum_{j=1}^N \Delta_j^{(N)}(s)\right] ds.
\end{equation}
The term $ A^{(N)}_{i,t,2}$ captures the proximity between the limit process and its expectation. We have that
 $$A^{(N)}_{i,t,2} = \int_0^t \mathbf{E}\left[ \left| \dfrac{1}{N} \sum_{j=1}^N \left( V^i_j(s) - \mathbf{E}\left[V^i_j(s)\right] \right) \right| \right] ds,$$
 where  $V^i_j(s)=\kappa_i^{(N)}\xi_{ij}^{(N)} \int_0^{s-}h(s-u) d\overline{Z}_j(u)$ is a family of independent random variables (by independence of the $\pi_i$). Note that $\mathbf{E}\left[V^i_j(s)\right]=\kappa_i^{(N)}\xi_{ij}^{(N)} \int_0^s h(s-u) \lambda(u,x_j)du$. Define $M^i_j(s):= V^i_j(s) - \mathbf{E}\left[ V^i_j(s) \right]$, which can also be written as 
 \begin{multline*} 
 M^i_j(s)= \int_0^s \int_0^\infty	\mathbf{1}_{\{ z\leq \lambda(u,x_j)\}} \kappa_i^{(N)}\xi_{ij}^{(N)}h(s-u) \pi_i(du,dz) -  \int_0^s \int_0^\infty	\mathbf{1}_{\{ z\leq \lambda(u,x_j)\}} \kappa_i^{(N)}\xi_{ij}^{(N)}h(s-u) du dz,
 \end{multline*}
so that
\begin{align*}
{\rm Var}\left( V_j^i(s) \right)= \mathbf{E}\left[ M^i_j(s)^2 \right] &= \mathbf{E} \left[  \int_0^s \int_0^\infty	\left( \mathbf{1}_{\{ z\leq \lambda(u,x_j)\}} \kappa_i^{(N)}\xi_{ij}^{(N)}h(s-u)\right)^2 du dz \right]\\
&= \int_0^s \left(\kappa_i^{(N)}\right)^2 \xi_{ij}^{(N)}h(s-u)^2 \lambda(u,x_j) du.
\end{align*}
Thus summing on $i$ and using Lemma \ref{lem:maj_famille_indep},
%$\mathbf{E}\left[ \left| \dfrac{1}{N} \sum_{j=1}^N \left( V^i_j(s) - \mathbf{E}\left[V^i_j(s)\right] \right) \right| \right]\leq \dfrac{1}{N} \sqrt{\sum_{j=1}^N \text{Var} \left( V^i_j(s)\right)},$
$$\dfrac{1}{N} \sum_{i=1}^N A^{(N)}_{i,t,2} \leq \dfrac{1}{N} \sum_{i=1}^N \int_0^t  \dfrac{1}{N} \sqrt{\sum_{j=1}^N \int_0^s \left(\kappa_i^{(N)}\right)^2 \xi_{ij}^{(N)}h(s-u)^2 \lambda(u,x_j) du}~ ds.$$
We apply Jensen's inequality to both uniform measures on $\left\{ 1, \ldots, N\right\}$ 	and $[0,t]$ to obtain:
\begin{align*}
\dfrac{1}{N} \sum_{i=1}^N A^{(N)}_{i,t,2} &\leq \dfrac{t}{N} \int_0^t \sqrt{ \dfrac{1}{N} \sum_{i=1}^N \sum_{j=1}^N  \int_0^s \left(\kappa_i^{(N)}\right)^2 \xi_{ij}^{(N)}h(s-u)^2 \lambda(u,x_j) du} \dfrac{ds}{t} \\
&\leq   \dfrac{t}{N} \sqrt{ \dfrac{1}{Nt} \sum_{i,j=1}^N \left(\kappa_i^{(N)}\right)^2 \xi_{ij}^{(N)} \int_0^t \int_0^s h(s-u)^2 \lambda(u,x_j)duds }
\end{align*}
By Hypothesis \ref{hyp:existence_Zin} on $h$, we have
$$\dfrac{1}{N} \sum_{i=1}^N A^{(N)}_{i,t,2}  \leq t\Vert h \Vert_{t,2}\sqrt{\Vert \lambda \Vert_{[0,t]\times I,\infty}} \sqrt{\dfrac{1}{N^3}\sum_{i,j=1}^N \left(\kappa_i^{(N)}\right)^2 \xi_{ij}^{(N)}}.$$
We use \eqref{eq:estimees_IC_som-ij_carre} and \eqref{eq:control_i} to obtain $\mathbb{P}$-almost surely if $N$ is large enough the bound: 
\begin{equation}\label{eq:moyA_i,2}
\dfrac{1}{N} \sum_{i=1}^N A^{(N)}_{i,t,2} \leq t\Vert h \Vert_{t,2}\sqrt{\Vert \lambda \Vert_{[0,t]\times I,\infty}} \sqrt{\dfrac{\kappa_N}{N}\left(1+  C_W\right)}.
\end{equation}
The term $ A^{(N)}_{i,t,3}$ captures the proximity between the realization of the graph $\left(\xi^{(N)}\right)$ and its expectation. We define for $(s,x,y)\in [0,T]\times I\times I$:
\begin{align}
\gamma(s,x)&:=\int_0^{s} h(s-u) \lambda(u,x)du,\label{eq:gamma}\\
\Gamma_T(x,y)&:=\int_0^{T} \gamma(s,x)\gamma(s,y)ds.\label{eq:Gamma}
\end{align}
Note that we always have $\vert \gamma(s,x) \vert \leq \Vert h \Vert_{s,1} \Vert \lambda \Vert_{[0,s]\times I, \infty} =: \gamma_{s,\infty}$ and $0\leq \Gamma_T(x,y) \leq T\gamma_{T,\infty}^2$. Recall that $\overline{\xi_{ij}}:=\xi_{ij}^{(N)}-W_N(x_i,x_j)$. Then $$A^{(N)}_{i,t,3}=  \int_0^t 	 \left | \dfrac{\kappa_i^{(N)}}{N}\sum_{j=1}^N \overline{\xi_{ij}}\gamma(s,x_j)\right|ds \leq \int_0^T \left | \dfrac{\kappa_i^{(N)}}{N}\sum_{j=1}^N \overline{\xi_{ij}}\gamma(s,x_j)\right|ds.$$ Note that one cannot apply Proposition \ref{prop:toolbox_Xij}  directly in the integrand since we would not get an a.s. result. Therefore,  we control its square, by Jensen's inequality:
$${A^{(N)}_{i,t,3}}^2 \leq T\int_0^T \left(  \dfrac{\kappa_i^{(N)}}{N}\sum_{j=1}^N \overline{\xi_{ij}}\gamma(s,x_j)\right)^2  ds = T^2 \gamma_{T,\infty}^2   \dfrac{\kappa_i^{(N)}}{N} \sum_{j=1}^N \overline{\xi_{ij}} X_{ij},$$
where we set $\displaystyle  X_{ij} :=  \dfrac{\kappa_i^{(N)}}{N} \sum_{l=1}^N  \overline{\xi_{il}}\dfrac{\Gamma_{T}(x_j,x_l)}{T \gamma_{T,\infty}^2}$.
Now, by Proposition \ref{prop:toolbox_Xij}, $\mathbb{P}$-almost surely for $N$ large enough, $\sup_{1\leq i,j \leq N} \vert X_{ij} \vert \leq \varepsilon_N$, thus
\begin{align}\label{eq:A_i,3}
{A^{(N)}_{i,t,3}}^2 &\leq T^2 \gamma_{T,\infty}^2   \dfrac{\kappa_i^{(N)}}{N} \sum_{j=1}^N \left( \xi_{ij}^{(N)} + W_N(x_i,x_j)\right) \sup_{i,j} \left| X_{ij}\right|\notag \\
&\leq T^2 \gamma_{T,\infty}^2 \varepsilon_N  \dfrac{\kappa_i^{(N)}}{N} \sum_{j=1}^N \left( \xi_{ij}^{(N)} + W_N(x_i,x_j)\right).
\end{align}
Taking the square root then summing on $i$, we use the discrete Jensen's inequality to obtain
\begin{align*}
\dfrac{1}{N} \sum_{i=1}^N  A^{(N)}_{i,t,3}  &\leq \sqrt{\varepsilon_N}T \gamma_{T,\infty} \dfrac{1}{N} \sum_{i=1}^N \sqrt{\dfrac{\kappa_i^{(N)}}{N} \sum_{j=1}^N \left( \xi_{ij}^{(N)} + W_N(x_i,x_j)\right)}\\
&\leq\sqrt{\varepsilon_N}T \gamma_{T,\infty} \sqrt{ \sum_{i=1}^N \dfrac{\kappa_i^{(N)}}{N^2} \sum_{j=1}^N \left( \xi_{ij}^{(N)} + W_N(x_i,x_j)\right)},
\end{align*}
if $N$ is large enough $\mathbb{P}$-almost surely. Using \eqref{eq:estimees_IC_sup-i_som-j} and \eqref{eq:control_i}, we have
\begin{align}\label{eq:moyA_i,3}
\dfrac{1}{N} \sum_{i=1}^N  A^{(N)}_{i,t,3}& \leq \sqrt{\varepsilon_N}T \gamma_{T,\infty} \sqrt{ \dfrac{1}{N}\sum_{i=1}^N \dfrac{\kappa_i^{(N)}}{N} \sum_{j=1}^N  \xi_{ij}^{(N)} +  \dfrac{1}{N}\sum_{i=1}^N \dfrac{\kappa_i^{(N)}}{N} \sum_{j=1}^N W_N(x_i,x_j)} \notag \\
&\leq \sqrt{\varepsilon_N}T \gamma_{T,\infty} \sqrt{ 1 +C_W +\sup_{i \in  \llbracket 1,N\rrbracket} \left( \dfrac{\kappa_i^{(N)}}{N} \sum_{j=1}^N W_N(x_i,x_j)\right)} \leq \sqrt{\varepsilon_N}T \gamma_{T,\infty} \sqrt{ 1+2C_W}. 
\end{align}
The term $ A^{(N)}_{i,t,4}$ captures the proximity between the law of the graph on $N$ particles and the limit graphon $W$.
Recall the definition of $\gamma$ in \eqref{eq:gamma} the graphs introduced in Definitions \ref{def:graphs_G1} and \ref{def:graphs_G2}. Denoting by $c(s)=\left(c_j(s)\right)_{1\leq j \leq N} = \left( \dfrac{\gamma(s,x_j)}{\gamma_{t,\infty}} \right)_{1\leq j \leq N} \in [-1,1]^N$, we obtain using \eqref{eq:graphonG} and introducing for any $c=\left(c_1,\cdots,c_N\right) \in [-1,1]^N$ the step function $g^c(v)=\sum_{l=1}^N c_l  \mathbf{1}_{B_l}(v)$ for $\in I$, after summation:

\begin{align*}
\dfrac{1}{N}\sum_{i=1}^N A^{(N)}_{i,t,4}%&= \dfrac{\gamma_{t,\infty}} {N^2} \int_0^t \sum_{i=1}^N  \left |\sum_{j=1}^N \left(\kappa_i^{(N)}W_N(x_i,x_j)-W(x_i,x_j)\right)c_j(s)\right|ds\\
%&= \gamma_{t,\infty} \int_0^t \int \left| \int \sum_{i,j=1}^N \left(\kappa_i^{(N)} W_N (x_i,x_j) - W(x_i,x_j \right) \mathbf{1}_{B_i^{(N)}\times B_j^{(N)}}(u,v) \right.\\&\left. \times \sum_{l=1}^N c_l(s)  \mathbf{1}_{B_l^{(N)}}(v) \nu(dv)\right|\nu(du) ds\\
&= \gamma_{t,\infty} \int_0^t  \int \left| \int \left( W^{\mathcal{G}_N^{(1)}}(u,v)-W^{\mathcal{G}_N^{(2)}}(u,v)\right) g^{c(s)}(v) \nu(dv)\right|\nu(du)ds\\
%&\leq T\gamma_{T,\infty} \sup_{\Vert g  \Vert _ \infty \leq 1 } \int \left| \left(W^{\mathcal{G}_N^{(1)}}(u,v)- W^{\mathcal{G}_N^{(2)}}(u,v)\right) g(v) \nu(dv)\right|\nu(du)\\
&\leq T\gamma_{T,\infty} \Vert  W^{\mathcal{G}_N^{(1)}}-W^{\mathcal{G}_N^{(2)}} \Vert _ { \infty\to 1,\nu},
\end{align*}
where  $\Vert \cdot \Vert _ { \infty\to 1,\nu}$ is defined in \eqref{eq:normeinf1_graphon}. Hence, with Remark \ref{rem:equ_norm_graphon} we obtain:

\begin{equation}\label{eq:moyA_i,4}
\dfrac{1}{N}\sum_{i=1}^N A^{(N)}_{i,t,4} \leq 4T\gamma_{T,\infty} \left( d_{\Box,\nu}\left( W^{\mathcal{G}_N^{(1)}},W \right) + d_{\Box,\nu}\left( W^{\mathcal{G}_N^{(2)}},W \right) \right).
\end{equation}
We use \eqref{eq:lem_S_box} to deal with $d_{\Box,\nu}\left( W^{\mathcal{G}_N^{(2)}},W \right)$.
\medskip

The term $\frac{1}{N} \sum_{i=1}^N A^{(N)}_{i,t,5}$ captures the proximity between the empirical measure of the positions of $N$ particles $\mu^{(N)}$ and its limit $\nu$. We control $A^{(N)}_{t,5} := \frac{1}{N}\sum_{i=1}^N A^{(N)}_{i,t,5}$ with \eqref{eq:lem_S_A5}. Combining \eqref{eq:moyA_i,1}, \eqref{eq:moyA_i,2}, \eqref{eq:moyA_i,3} and  \eqref{eq:moyA_i,4}, we obtain if $N$ is large enough $\mathbb{P}$-almost surely for every $t\in[0,T]$:
\begin{multline}
\mathbf{E}\left[\dfrac{1}{N}\sum_{i=1}^N\Delta_i^{(N)}(t)\right] \leq C_1 \int_0^t  \vert h(t-s) \vert \mathbf{E} \left[\dfrac{1}{N} \sum_{j=1}^N  \Delta_j^{(N)}(s)\right] ds +   C_2\sqrt{\dfrac{\kappa_N}{N}} +C_3 \sqrt{\varepsilon_N} \\+  C_4 d_{\Box,\nu}\left(W^{\mathcal{G}_N^{(1)}},W\right) + C_4 \gamma_{T,\infty}d_{\Box,\nu}\left(W^{\mathcal{G}_N^{(2)}},W\right)  +  L_f A^{(N)}_{t,5}
\end{multline}
with $C_1, C_2, C_3, C_4$ constants depending on $L_f$, $C_W$, $h$ and $T$. We apply Lemma \ref{lem:gronwall_convole} with $u(t)=\mathbf{E}\left[\dfrac{1}{N}\sum_{i=1}^N\Delta_i^{(N)}(t)\right]$ on $[0,T]$, and remind that $\sup_{t\in [0,T]} \left| Z_i^{(N)}(t) - \overline{Z}_i(t) \right| \leq \Delta_i^{(N)}(T)$ to obtain $\mathbb{P}$-almost surely on the realisation of $\left(\xi^{(N)}\right)$ if $N$ is large enough:
\begin{equation}\label{eq:maj_explicite} 
\frac{1}{N} \sum_{i=1}^N \mathbf{E} \left[ \sup_{t\in [0,T]} \left| Z_i^{(N)}(t) - \overline{Z}_i(t) \right| \right] \leq C\left( \sqrt{\dfrac{\kappa_N}{N}} + \sqrt{\varepsilon_N} + d_{\Box,\nu}\left(W^{\mathcal{G}_N^{(1)}},W\right)+ d_{\Box,\nu}\left(W^{\mathcal{G}_N^{(2)}},W\right)+  A^{(N)}_{T,5} \right),
\end{equation}
with $C=\sqrt{2} \max\left(C_2,C_3,C_{4},L_f\right) \exp\left( C_1^2 \Vert h \Vert_{T,2}^2 T\right)$.
By \eqref{eq:kappaw_nlog}, $\displaystyle\lim_{N\to \infty}\varepsilon_N= 0$, $\displaystyle\lim_{N\to\infty}\dfrac{\kappa_N}{N}=0$ and by \eqref{eq:hyp_cut_cvg}  $d_{\Box,\nu}\left( W^{\mathcal{G}_N^{(1)}}, W \right) \xrightarrow[N\to\infty]{} 0$. Combining it with Proposition \ref{prop:scenario_hyp}, we conclude the proof of \eqref{eq:moy_ecart_limite_nulle}. \qed
%----------------------------------------------------------------------------------------
%----------------------------------------------------------------------------------------

\subsection{Proof of Theorem \ref{thm:cvg-sup_0}}\label{S:proof_thm:cvg-sup_0}

It is almost the same as for Theorem \ref{thm:cvg_0} with changes due to the fact that we take now the maximum on $i$. Let us go back to the inequality \eqref{eq:delta_i_majoration}. We are going to control each term $\max_{1\leq i \leq N} A^{(N)}_{i,t,k}$.
Concerning $A^{(N)}_{i,t,1}$, the same estimate \eqref{eq:A_i,1} as in the proof of Theorem \ref{thm:cvg_0} leads now to
$$\max_{1\leq i \leq N} A^{(N)}_{i,t,1} \leq \max_{1\leq i \leq N} \left( \dfrac{1}{N} \sum_{j=1}^N \kappa_i^{(N)}\xi_{ij}^{(N)} \right) \int_0^t \vert h(t-s) \vert \max_{1\leq i \leq N} \mathbf{E} \left[\Delta_i^{(N)}(s)\right]ds.$$
We use \eqref{eq:estimees_IC_sup-i_som-j} and \eqref{eq:control_i} to obtain $\mathbb{P}$-almost surely if $N$ is large enough:  
\begin{equation}\label{eq:maxA_i,1}
\max_{1\leq i \leq N} A^{(N)}_{i,t,1} \leq \left(1+C_W\right) \int_0^t \vert h(t-s) \vert \max_{1\leq i \leq N} \mathbf{E} \left[\Delta_i^{(N)}(s)\right]ds.
\end{equation}
Note that here, we do not use the same control as is the proof of Theorem \ref{thm:cvg_0}, we only need the uniformly bounded indegree.
Concerning $A^{(N)}_{i,t,2}$, we obtain as  in the proof of Theorem \ref{thm:cvg_0}
$$A^{(N)}_{i,t,2} \leq \int_0^t  \dfrac{1}{N} \sqrt{\sum_{j=1}^N \int_0^s \left(\kappa_i^{(N)}\right)^2 \xi_{ij}^{(N)}h(s-u)^2 \lambda(u,x_j) du}~ ds.$$
We use Jensen's inequality on the probability measure $\frac{1}{t}dt$ on $[0,t]$ and then the boundedness of $h$ and $\lambda$ to obtain
\begin{align*}
A^{(N)}_{i,t,2} %&\leq \dfrac{t}{N} \int_0^t  \sqrt{\sum_{j=1}^N \int_0^s \left(\kappa_i^{(N)}\right)^2 \xi_{ij}^{(N)}h(s-u)^2 \lambda(u,x_j) du}~ \dfrac{ds}{t}\\
%&\leq  \dfrac{t}{N} \sqrt{ \int_0^t  \sum_{j=1}^N \int_0^s \left(\kappa_i^{(N)}\right)^2 \xi_{ij}^{(N)}h(s-u)^2 \lambda(u,x_j) du~ \dfrac{ds}{t}}\\
&\leq   \dfrac{t}{N \sqrt{t}} \kappa_i^{(N)} \sqrt{ \sum_{j=1}^N  \xi_{ij}^{(N)} \int_0^t  \int_0^s  h(s-u)^2 \lambda(u,x_j) du~ ds}\\
%A^{(N)}_{i,t,2}% &\leq   \dfrac{\sqrt{t}}{N} \kappa_i^{(N)} \sqrt{ \sum_{j=1}^N  \xi_{ij}^{(N)} \Vert h \Vert_{t,2}^2 \Vert \lambda \Vert_{[0,t]\times\mathbb{R}^d,\infty} }\\
& \leq   \Vert h \Vert_{t,2 } \sqrt{ \Vert \lambda \Vert_{[0,t]\times\mathbb{R}^d,\infty} } \dfrac{\sqrt{t}}{N} \kappa_i^{(N)} \sqrt{ \sum_{j=1}^N  \xi_{ij}^{(N)} },
\end{align*}
and taking the maximum leads to 
$$\max_{1\leq i \leq N} A^{(N)}_{i,t,2} \leq \sqrt{\dfrac{\kappa_N}{N}} \sqrt{ \max_{1\leq i \leq N} \dfrac{\kappa_i^{(N)}}{N}   \sum_{j=1}^N  \xi_{ij}^{(N)} }  \Vert h \Vert_{t,2 } \sqrt{ t \Vert \lambda \Vert_{[0,t]\times\mathbb{R}^d,\infty} }.$$
Using as before \eqref{eq:estimees_IC_sup-i_som-j} and \eqref{eq:control_i}, we obtain $\mathbb{P}$-almost surely if $N$ is large enough
\begin{equation}\label{eq:maxA_i,2}
\max_{1\leq i \leq N} A^{(N)}_{i,t,2} \leq \sqrt{\dfrac{\kappa_N}{N}} \sqrt{1+C_W}  \Vert h \Vert_{t,2 } \sqrt{ t \Vert \lambda \Vert_{[0,t]\times\mathbb{R}^d,\infty} }.
\end{equation}
Concerning $A^{(N)}_{i,t,3}$, we obtain as in the proof of Theorem \ref{thm:cvg_0} (see \eqref{eq:A_i,3}, with \eqref{eq:estimees_IC_sup-i_som-j} and \eqref{eq:control_i}) that $\mathbb{P}$-almost surely
\begin{align*}
{A^{(N)}_{i,t,3}}  &\leq T \gamma_{T,\infty} \sqrt{\varepsilon_N  \dfrac{\kappa_i^{(N)}}{N} \sum_{j=1}^N \left( \xi_{ij}^{(N)} + W_N(x_i,x_j)\right)},
\end{align*}
hence taking the maximum and using \eqref{eq:estimees_IC_sup-i_som-j} and \eqref{eq:control_i}, we obtain $\mathbb{P}$-almost surely if $N$ is large enough
\begin{equation}\label{eq:maxA_i,3}
\max_{1\leq i \leq N} A^{(N)}_{i,t,3}  \leq T \gamma_{T,\infty} \sqrt{\varepsilon_N} \sqrt{1+2C_W}.
\end{equation}
Concerning $ A^{(N)}_{i,t,4}$, we recognise
\begin{align*}
A^{(N)}_{i,t,4}%&=  \int_0^t  \left | \dfrac{1}{N}\sum_{j=1}^N \left(\kappa_i^{(N)}W_N(x_i,x_j)-W(x_i,x_j)\right)\int_0^{s-} h(s-u) \lambda(u,x_j)du\right|ds\\
&= \gamma_{t,\infty}\int_0^t  \left | \dfrac{1}{N}\sum_{j=1}^N \left(\kappa_i^{(N)}W_N(x_i,x_j)-W(x_i,x_j)\right)c_j(s)\right|ds.
\end{align*}
We obtain, using Definitions \ref{def:graphs_G1} and \ref{def:graphs_G2} with Lemma \ref{lem:partition_I} that as
\begin{align*}
&\sup_{1\leq i \leq N}  \left|  \dfrac{1}{N}\sum_{j=1}^N\left(\kappa_i^{(N)} W_N (x_i,x_j) - W(x_i,x_j \right) c_j(s) \right|
%&= \sup_{u\in I} \left| \sum_{i=1}^N   \mathbf{1}_{B_i^{(N)}}(u) \int_I \sum_{j=1}^N\left(\kappa_i^{(N)} W_N (x_i,x_j) - W(x_i,x_j \right) \mathbf{1}_{B_j^{(N)}}(v) c_j(s) \nu(dv)\right|\\
= \sup_{u\in I} \left| \int \left( W^{\mathcal{G}_N^{(1)}}(u,v)-W^{\mathcal{G}_N^{(2)}}(u,v)\right) g^{c(s)}(v) \nu(dv)\right|,
\end{align*}
we have
\begin{align}\label{eq:maxA_i,4}
\sup_{1\leq i \leq N} A^{(N)}_{i,t,4} %&= \sup_{1\leq i \leq N}\int_0^t  \left | \dfrac{1}{N}\sum_{j=1}^N \left(\kappa_i^{(N)}W_N(x_i,x_j)-W(x_i,x_j)\right)\gamma(s,x_j)\right|ds \notag\\
&= \gamma_{t,\infty} \sup_{1\leq i \leq N}\int_0^t  \left | \dfrac{1}{N}\sum_{j=1}^N \left(\kappa_i^{(N)}W_N(x_i,x_j)-W(x_i,x_j)\right)c_j(s)\right|ds\notag\\
&\leq \gamma_{t,\infty} \int_0^t  \sup_{g, \Vert g \Vert _{\infty}\leq 1} \sup_{u\in I} \left| \int \left( W^{\mathcal{G}_N^{(1)}}(u,v)-W^{\mathcal{G}_N^{(2)}}(u,v)\right) g(v) \nu(dv)\right| ds\notag\\
&\leq T\gamma_{T,\infty} \Vert W^{\mathcal{G}_N^{(1)}}-W^{\mathcal{G}_N^{(2)}} \Vert_{\infty	\to \infty, \nu}\notag\\& \leq T\gamma_{T,\infty} \left(\Vert W^{\mathcal{G}_N^{(1)}}-W \Vert_{\infty	\to \infty, \nu} +  \Vert W-W^{\mathcal{G}_N^{(2)}} \Vert_{\infty	\to \infty, \nu}\right) .
\end{align}
Concerning $A^{(N)}_{i,t,5}$, we denote by $\widetilde{A}^{(N)}_{t,5} =\max_{1\leq i \leq N} A^{(N)}_{i,t,5} $.  It is controlled with \eqref{eq:lem_S_A5_max}. 
Combining \eqref{eq:maxA_i,1}, \eqref{eq:maxA_i,2}, \eqref{eq:maxA_i,3} and \eqref{eq:maxA_i,4}, we obtain if $N$ is large enough $\mathbb{P}$-almost surely for every $t\in[0,T]$:
\begin{multline}
\mathbf{E}\left[\max_{1\leq i \leq N}\Delta_i^{(N)}(t)\right] \leq C_1 \int_0^t  \vert h(t-s) \vert \mathbf{E} \left[\max_{1\leq i \leq N} \Delta_j^{(N)}(s)\right] ds +   C_2\sqrt{\dfrac{\kappa_N}{N}}\\ +C_3 \sqrt{\varepsilon_N} + C_4 \Vert W^{\mathcal{G}_N^{(1)}}-W \Vert_{\infty	\to \infty, \nu} + C_4 \Vert W^{\mathcal{G}_N^{(2)}}-W \Vert_{\infty	\to \infty, \nu} +  L_f\widetilde{A}^{(N)}_{t,5}
\end{multline}
with $C_1$, $C_2$, $C_3$ and $C_4$ constants depending on $h$, $f$, $C_W$ and $T$.
%
%\begin{equation}\label{eq:constantes_sup}
%\left \{
%\begin{array}{c @{=} l}
%C_1&L_f (1+C_W) \\
%C_2&L_f  \Vert h \Vert_{T,2 } \sqrt{ T \left(1+C_W\right) \Vert \lambda \Vert_{[0,t]\times\mathbb{R}^d,\infty} }\\
%C_3&L_f T \gamma_{T,\infty} \sqrt{ 1+2C_W}\\
%C_4& L_f T\gamma_{T,\infty}.
%\end{array}
%\right.
%\end{equation}
We apply the Lemma \ref{lem:gronwall_convole} with $u(t)=\mathbf{E}\left[\max_{1\leq i \leq N}\Delta_i^{(N)}(t)\right]$ on $[0,T]$, and remind $\sup_{t\in [0,T]} \left| Z_i^{(N)}(t) - \overline{Z}_i(t) \right| \leq \Delta_i^{(N)}(T)$ to obtain $\mathbb{P}$-almost surely on the realisation of $\left(\xi^{(N)}\right)$ if $N$ is large enough:
\begin{multline}\label{eq:maj_explicite_sup} 
\max_{1\leq i \leq N} \mathbf{E} \left[ \sup_{t\in [0,T]} \left| Z_i^{(N)}(t) - \overline{Z}_i(t) \right| \right] \leq C\left( \sqrt{\dfrac{\kappa_N}{N}} + \sqrt{\varepsilon_N} + \Vert W^{\mathcal{G}_N^{(1)}}-W \Vert_{\infty	\to \infty, \nu} \right.\\\left.+ \Vert W^{\mathcal{G}_N^{(2)}}-W \Vert_{\infty	\to \infty, \nu}+\widetilde{A}^{(N)}_{t,5} \right)
\end{multline}
with $C=\sqrt{2} \max\left(C_2,C_3,C_4,L_f\right) \exp\left( C_1^2 \Vert h \Vert_{T,2}^2 T\right)$.
By \eqref{eq:kappaw_nlog}, $\displaystyle\lim_{N\to \infty}\varepsilon_N= 0$, $\displaystyle\lim_{N\to\infty}\dfrac{\kappa_N}{N}=0$ and by \eqref{eq:hyp_infinf_cvg} $ \displaystyle\lim_{N\to \infty}\Vert W^{\mathcal{G}_N^{(1)}}-W \Vert= 0 $. Combining with Proposition \ref{prop:scenario_hyp}, it concludes the proof of \eqref{eq:sup_ecart_limite_nulle}.\qed
%----------------------------------------------------------------------------------------
%----------------------------------------------------------------------------------------
%----------------------------------------------------------------------------------------
\subsection{Proofs: Application to the Scenarios of Definition \ref{def:scenarios}}\label{S:proof_scenarios}
%----------------------------------------------------------------------------------------
%----------------------------------------------------------------------------------------
In this section, we prove Proposition \ref{prop:scenario_hyp}. We start with auxiliary results that come up in the main proof.
\subsubsection{Toolbox}
\begin{lem}\label{lem:quantile_empirical_uniforme}
Let $(\widetilde{x_i})_{i\geq 1}$ be a sequence of i.i.d positions on $[0,1]$ with distribution $\mathcal{U}[0,1]$. For all $N\geq 1$ and for $i=1,\cdots,N$, define $x_i=\widetilde{x}_{(i)}$ as the order statistics of $\left(\widetilde{x}_1,\cdots,\widetilde{x}_N\right)$ (i.e. $\{\widetilde{x}_1,\cdots,\widetilde{x}_N\} = \{ x_1,\cdots,x_N\}$ and $x_1<\cdots<x_N$). Then, for any borelian sets $A$ and $B$ of $(0,1]$,
\begin{equation}\label{eq:lem_quantile_emp}
\dfrac{1}{N}\sum_{i=1}^N \mathbf{1}_{x_i\in A, \frac{i}{N}\in B} \xrightarrow[N\to \infty]{} \lambda(A\cap B) \quad a.s.
\end{equation}
where $\lambda$ denotes the Lebesgue measure on $[0,1]$.
\end{lem}
\begin{proof}
It is sufficient to show that for all $(t,t')\in (0,1]^2$,
$$\dfrac{1}{N}\sum_{i=1}^N \mathbf{1}_{x_i\leq t, \frac{i}{N}\leq t'} \xrightarrow[N\to \infty]{} \min(t,t') \quad a.s.$$
We introduce the uniform sample quantile function  as in \cite{csorgHo1983quantile}: define for any $y\in [0,1]$
\begin{equation}\label{eq:def_U_stat}
U_N(y) = \left\{
    \begin{array}{ll}
        0 & \text{ if } y=0\\
        x_k  & \text { if } \dfrac{k-1}{N} < y \leq \dfrac{k}{N}, \quad k\in \llbracket 1,N \rrbracket.
    \end{array}
\right.
\end{equation}
First, we show that $\displaystyle \lim_{N\to\infty} \dfrac{1}{N}\sum_{i=1}^N \mathbf{1}_{x_i\leq t, \frac{i}{N}\leq t'} = \lim_{N\to\infty} \int_0^{t'} \mathbf{1}_{U_N(y)\leq t} dy$. We note $k$ the integer such that $x_k \leq t < x_{k+1}$ (and $k=0$ if $x_1>t$). If $t'\geq \frac{k}{N}$, then
$\frac{1}{N}\sum_{i=1}^N \mathbf{1}_{x_i\leq t, \frac{i}{N}\leq t'} =  \frac{1}{N}\sum_{i=1}^N \mathbf{1}_{x_i\leq t, i\leq Nt'} = \frac{1}{N}\sum_{i=1}^{Nt'} \mathbf{1}_{x_i\leq t} = \frac{k}{N}$,
and $\int_0^{t'} \mathbf{1}_{U_N(y)\leq t} dy = \int_0^{t'} \mathbf{1}_{y\leq \frac{k}{n}} dy = \frac{k}{N}$. If $t'<\frac{k}{N}$, $\int_0^{t'} \mathbf{1}_{U_N(y)\leq t} dy = \int_0^{t'} \mathbf{1}_{y\leq \frac{k}{n}} dy = t'$ and
$\frac{1}{N}\sum_{i=1}^N \mathbf{1}_{x_i\leq t, \frac{i}{N}\leq t'} =\frac{1}{N}\sum_{i=1}^k \mathbf{1}_{i\leq Nt'} = \frac{\lfloor Nt' \rfloor}{N} \xrightarrow[N\to \infty]{} t'$. Then, we know from \cite{csorgHo1983quantile} that $\displaystyle\sup_{0 \leq y \leq 1} \vert U_N(y) - y \vert \xrightarrow[N\to \infty]{a.s.} 0$, and hence almost surely, for any fixed $y\in[0,1]$, $U_N(y) \xrightarrow[N\to \infty]{p.s.}  y$ and by dominated convergence $ \int_0^{t'} \mathbf{1}_{U_N(y)\leq t} dy \xrightarrow[N\to \infty]{} \int_0^{t'} \mathbf{1}_{y\leq t} dy =\min(t,t')$,
which concludes the proof.
\end{proof}

\begin{prop}\label{prop:toolbox_iid}
Under the Scenario (1) of Definition \ref{def:scenarios}, for any function $g$ such that $\Vert g \Vert_{L^\chi(I\times I), \nu\times \nu}<\infty$ with $\chi>5$,
\begin{equation}\label{eq:toolbox_iid}
\sup_{1\leq i \leq N} \int_I g(x_i,y) \left( \nu^{(N)}(dy) - \nu(dy) \right) \xrightarrow[N\to\infty]{} 0
\end{equation}
 $\mathbb{P}$-almost surely on the realisation of the sequence $\left(\underline{x}^{(N)}\right)_N$.
\end{prop}

\begin{proof}
Fix $M>0$, and define the function $p_M(u)= u \mathbf{1}_{\vert u \vert \leq M} + M\mathbf{1}_{u>M} - M\mathbf{1}_{u<-M}$ on $\mathbb{R}$. Set $g_M=p_M\circ g$. The following arguments come from \cite{Luon2020} in the proof of Proposition 3.4. We have
\begin{align*}
&\sup_{1\leq i \leq N} \int_I g(x_i,y) \left( \nu^{(N)}(dy) - \nu(dy) \right) \leq \sup_{1\leq i \leq N} \dfrac{1}{N} \sum_{j=1}^N \left| g(x_i,x_j) - g_M(x_i,x_j)\right| \\&+ \sup_{1\leq i \leq N} \int_I \left| g(x_i,y) - g_M(x_i,y)\right| \nu(dy) + \sup_{1\leq i \leq N} \int_I g_M(x_i,y) \left( \nu^{(N)}(dy) - \nu(dy) \right) =: (I) + (II) + (III). 
\end{align*}
To study $(I)$, note that $\left| g(x,y) - g_M(x,y) \right| = \left| g(x,y) - g_M(x,y) \right| \mathbf{1}_{\left| g(x,y) \right|>M} \leq 2 \vert g(x,y) \vert \mathbf{1}_{\left| g(x,y) \right|>M},$ and that for any independent $X,Y$ with distribution $\nu$ 
\begin{align*}
\mathbb{E}\left[\left|g(X,Y)\right|\mathbf{1}_{\left|g(X,Y)\right|>M}\right] &= \sum_{l=0}^{+\infty} \mathbb{E}\left[\left|g(X,Y)\right|\mathbf{1}_{2^lM<\left|g(X,Y)\right|\leq 2^{l+1}M}\right]\\
&\leq \sum_{l=0}^{+\infty} 2^{l+1} M \left( \mathbb{P} \left(\left|g(X,Y)\right|>2^lM\right) - \mathbb{P} \left(\left|g(X,Y)\right|>2^{l+1}M\right) \right)\\
&= 2M\mathbb{P} \left( \left|g(X,Y)\right|>M\right) + \sum_{l=1}^{+\infty} 2^lM \mathbb{P} \left( \left|g(X,Y)\right|>2^lM\right)\\
&\leq \mathbb{E}\left[\left|g(X,Y)\right|^\chi\right]\left( \dfrac{2}{M^{\chi-1}} + \sum_{l=1}^{+\infty} \dfrac{2^l M}{\left(2^l M \right)^\chi}\right) \leq \dfrac{3\mathbb{E}\left[\left|g(X,Y)\right|^\chi\right]}{M^{\chi-1}},
\end{align*}
with Markov inequality. Since $\displaystyle \mathbf{E}\left[ \dfrac{1}{N} \sum_{l=1}^N \left|g(x_i,x_l)-g_M(x_i,x_l)\right| \right] \leq \dfrac{2}{N} \sum_{l=1}^N \mathbb{E}\left[\left|g(x_i,x_l)\right|\mathbf{1}_{\vert g(x_i,x_l)\vert>M}  \right]$, it implies for the choice of $M=N^{\delta_1}$ with $\delta_1>0$ to be defined later, using Markov inequality and a union bound that $\mathbb{P} \left( (I) > \dfrac{1}{N^{\delta_2}}\right) \leq  \dfrac{6\mathbb{E}\left[\left|g(X,Y)\right|^\chi\right]}{N^{\delta_1(\chi-1)-\delta_2-1}}$. Similarly, we can show that 
$\mathbb{P} \left( (II) > \dfrac{1}{N^{\delta_2}}\right) \leq  \dfrac{6\mathbb{E}\left[\left|g(X,Y)\right|^\chi\right]}{N^{\delta_1(\chi-1)-\delta_2-1}}$.
We will use the two previous bounds with Borel-Cantelli Lemma to deduce that $\mathbb{P}$-almost surely, $(I) + (II)\xrightarrow[N\to\infty]{} 0$ by asking $\delta_1(\chi-1)-\delta_2-1>1$.
To deal with $(III)$ we use the boundedness of $g_M$. Note that $(III)$ can be re-written $\displaystyle \sup_{1\leq i \leq N} \dfrac{1}{N}\sum_{l=1}^N Y_l^{(i),M}$ with $Y_l^{(i),M}:=g_M(x_i,x_l) - \int_I g_M(x_i,y)\nu(dy)=g_M(x_i,x_l)  - \mathbb{E}\left[ g_M(x_i,Y) \vert x_i\right]$. We set $\mathcal{F}^{(i)}_l=\sigma\left(x_i, x_1, \ldots, x_l\right)$. We have for $l\neq i$
$$\mathbb{E}\left[Y^{(i),M}_l \left| \mathcal{F}^{(i)}_{l-1}\right.\right]=\mathbb{E}\left[ U_M(x_i,x_l) - \mathbb{E}_Y\left[U_M(x_i,Y)\vert x_i \right] \left| \mathcal{F}^{(i)}_{l-1}\right. \right]=0.$$
As $\left|Y_{l}^{(i),M}\right|\leq 2M$, we can then apply Lemma \ref{lem:inegalit_concentration_Y}: for any $x> 0$,
$$\mathbb{P}\left( \dfrac{1}{N-1} \sum\limits_{\substack{l=1 \\ l \neq i}}^{N}\dfrac{Y^{(i),M}_l}{2M}\geq x\right) \leq \exp\left( -(N-1) \dfrac{x^2}{2}B(x)\right)$$
with the function B defined in \eqref{eq:def_B(u)}. We consider a sequence $\varepsilon_N$ such that $\varepsilon_N\xrightarrow[N\to \infty]{} 0$ (we precise later on which one), and we apply the previous result with $x=\dfrac{\varepsilon_N N}{2M(N-1)}$. As $B(u)=u^{-2}\left( \left( 1+u \right) \log \left( 1+u \right) - u \right)\to\dfrac{1}{2}$ when $u\to 0$, we can choose a deterministic $p$ such that for all $N\geq p$, $B\left(\dfrac{\varepsilon_N N}{2M(N-1)}\right) \geq \dfrac{1}{4}$. We then have if $N\geq p$: 
$\mathbb{P}\left( \dfrac{1}{N} \sum\limits_{\substack{l=1 \\ l \neq i}}^{N}Y^{(i),M}_l\geq \varepsilon_N \right) \leq \exp\left( - \dfrac{1}{32M^2}\dfrac{\varepsilon_N^2 N^2}{N-1}\right)$,
doing the same for $-Y^{(i)}_l$ and with a union bound we obtain 
$$\mathbb{P}\left( \sup_{1\leq i \leq N} \left| \dfrac{1}{N} \sum\limits_{\substack{l=1 \\ l \neq i}}^{N}Y^{(i),M}_l\right|\geq \varepsilon_N \right) \leq 2N \exp\left( - \dfrac{1}{32M^2}\dfrac{\varepsilon_N^2 N^2}{N-1}\right).$$
It is sufficient to find $\varepsilon_N$ such that $\varepsilon_N\xrightarrow[N\to \infty]{} 0$ and $\sum_N 2N \exp\left( - \dfrac{1}{32M^2}\dfrac{\varepsilon_N^2 N^2}{N-1}\right)<\infty$ to conclude by Borel-Cantelli's Lemma, $\mathbb{P}$-almost surely if $N$ is large enough $\sup_{1\leq i \leq N} \left| \dfrac{1}{N} \sum\limits_{\substack{l=1 \\ l \neq i}}^{N}Y^{(i),M}_l\right|\leq \varepsilon_N$. We set then $\varepsilon_N^2:=32 M^2 (N-1)N^\gamma$, and require $-2<\gamma<-1-2\delta_1$. As $Y_i^{(i),M}$ is bounded (by $2M$), adding the term $\frac{1}{N} Y_i^{(i),M}$ does not change the convergence if $\delta_1<1$ which was already asked for the conditions on $\varepsilon_N$ (recall $M=N^{\delta_1}$).
We are left with finding parameters $\left(\delta_1,\delta_2,\gamma\right)$ such that $\delta_1>0$, $\delta_2>0$, $\delta_1(\chi-1)-\delta_2-1>1$, $-2<\gamma<-1-2\delta_1$ (to ensure that the probabilities obtained with $(I)$, $(II)$ and $(III)$  are summable and the sufficient conditions on $\varepsilon_N$). As $\chi>5$, any choice such that $\delta_1 \in (0,\frac{1}{2})$ and $\delta_2 \in (0,1)$ works (as $\delta_1(\chi-1)-1<1$) with $\gamma \in (-2,-1-2\delta_1)$, and we obtain \eqref{eq:toolbox_iid} $\mathbb{P}$-almost surely.
\end{proof}

\begin{cor}\label{cor:toolbox_iid}
Under Scenario (1) of Definition \ref{def:scenarios}, we define
\begin{align}
\epsilon_{i,1}&:= \int_{I\times I} W(x_i,y)W(x_i,z)\Gamma(y,z) \left(\nu^{(N)}(dy)\nu^{(N)}(dz)-\nu(dy)\nu^{(N)}(dz)\right)\label{eq:def_cor_e1}\\
\epsilon_{i,2}&:= \int_{I\times I} W(x_i,y)W(x_i,z)\Gamma(y,z) \left(\nu(dy)\nu^{(N)}(dz)-\nu(dy)\nu(dz)\right),\label{eq:def_cor_e2}
\end{align}
where $\Gamma$ is defined in \eqref{eq:Gamma}.
Then under Hypothesis \ref{hyp:conv_graph_concentration}, $\mathbb{P}$-almost surely, $$\sup_{1\leq i \leq N}\epsilon_{i,1} \xrightarrow[N\to \infty]{} 0 \text{ and } \sup_{1\leq i \leq N}\epsilon_{i,2} \xrightarrow[N\to \infty]{} 0.$$
\end{cor}

\begin{proof}
Note that $\epsilon_{i,2}=\int_I \phi(x_i,z) \left(\nu^{(N)}(dz)-\nu(dz)\right)$ , with $\phi(x,z):= W(x,z) \int_I W(x,y)\Gamma(y,z) \nu(dy)$. As $\Gamma$ is bounded, $\left| \phi(x,z) \right| \leq \left| W(x,z)\right| \Vert\Gamma\Vert_\infty C_W^{(1)}$ and since $W\in L^\chi(I^2,\nu\times\nu)$, $\Vert \phi \Vert_{L^\chi(I\times I), \nu\times \nu}<\infty$, \eqref{eq:def_cor_e2} is an immediate application of Proposition \ref{prop:toolbox_iid}. Similarly, $\epsilon_{i,1}=\int_I g_N(x_i,y) \left(\nu^{(N)}(dy)-\nu(dy)\right)$, with $g_N(x,y):=W(x,y) \int_I W(x,z)\Gamma(y,z) \nu^{(N)}(dz)$. Define $g(x,y):=W(x,y) \int_I W(x,z)\Gamma(y,z) \nu(dz)$, then
$$\epsilon_{i,1} = \int_I \left( g_N(x_i,y)-g(x_i,y)\right)\left(\nu^{(N)}(dy)-\nu(dy)\right) + \int_I g(x_i,y)\left(\nu^{(N)}(dy)-\nu(dy)\right).$$
We have immediately (as done with \eqref{eq:def_cor_e2}) that $\sup_{1\leq i \leq N} \int_I g(x_i,y)\left(\nu^{(N)}(dy)-\nu(dy)\right)\xrightarrow[N\to \infty]{} 0$. For the other term, that we denote by $\epsilon_{i,3}$, we have
$\epsilon_{i,3} = \int_I W(x_i,y) \alpha_N(x_i,y)\left(\nu^{(N)}(dy)-\nu(dy)\right)$ where  $\alpha_N(x_i,y):= \int_I W(x_i,z) \Gamma(y,z)  \left(\nu^{(N)}(dz)-\nu(dz)\right) $. As $\Gamma$ is bounded, Proposition \ref{prop:toolbox_iid} (and its proof) gives that $\alpha_N(x_i,y)\xrightarrow[N\to \infty]{} 0 $ uniformly in $i$ and $y$. Another application of Proposition \ref{prop:toolbox_iid} gives then that $\sup_{1\leq i \leq N} \epsilon_{i,3} \xrightarrow[N\to \infty]{} 0$ which concludes the proof.
\end{proof}
%---------------------------------------------------------------------------------------
%-----------------------------------------------------------
\subsubsection{Proof of Proposition \ref{prop:scenario_hyp} for Scenario (1)} 
We treat the estimates \eqref{eq:lem_S_box}, \eqref{eq:lem_S_box_max}, \eqref{eq:lem_S_A5_max} and \eqref{eq:lem_S_A5} separately.
%-----------------------------------------------------------
%-----------------------------------------------------------
\paragraph{Proof of \eqref{eq:lem_S_box}}
We remind that we want to prove $d_{\Box,\nu}\left( W^{\mathcal{G}_N^{(2)}},W \right) \xrightarrow[N\to \infty]{} 0$, when the positions are i.i.d. according to $\nu$ on $I$. Recall the definition of $\left(x_1,\cdots, x_N\right)$ as the lexicographic  reordering of the i.i.d. sample $(\widetilde{x_1},\widetilde{x_2},\cdots, \widetilde{x_N})$. The proof is organised as follow: we start by looking at the case $d=1$, $I=[0,1]$ and $\nu$ is the Lebesgue measure on $I$, and then extend to the general case.% We start by approximating $W$ by $\widetilde{W}$ continuous in $ L^1(I^2,\nu)$.
\medskip

\textit{Step 1 - Approximation of $W$ in norm $L^1$.} We first prove that for $\varepsilon>0$, there exists $m\geq 1$ sufficiently large such that $\Vert W - W_{\mathcal{P}_m}\Vert_{L^1(I^2)}\leq \varepsilon$.
We fix $\varepsilon>0$.  As $W\in L^1(I^2,\nu)$, there exists $\widetilde{W}$ continuous such that $\Vert W-\widetilde{W}\Vert_{1,\nu}\leq \dfrac{\varepsilon}{3}$. As $\widetilde{W}$ is also uniformly continuous, there exists $\eta>0$ such that if $\Vert u - u'\Vert + \Vert v-v' \Vert \leq \eta$, $\left| \widetilde{W}(u,v)-\widetilde{W}(u',v') \right| \leq \dfrac{\varepsilon}{3}$. We fix $m$ large enough such that $\frac{1}{m}\leq \eta$, and denote by $\mathcal{P}_m=\sqcup_{i=1}^{m}J_i$ the partition with $J_i=\left(\frac{i-1}{m},\frac{i}{m}\right]$. It verifies then, for each $i\in(1,\cdots,m)$ $Diam(J_i)\leq\eta$. We define the step function (which average the values of $W$ over cells obtained with the partition)
\begin{equation}\label{eq:def_Pm}
W_{\mathcal{P}_m}(u,v):= m^2 \sum_{i,j=1}^m \int_{J_i\times J_j} W(x,y) \nu(dx) \nu(dy)\mathbf{1}_{J_i}(u)\mathbf{1}_{J_j}(v).
\end{equation}
We note $\mathcal{G}_N^{(3)}$ the directed weighted graph with vertices $\left\{ 1, \cdots,N\right\}$ such that every edge $j\rightarrow i$ is present, with weight $W_{\mathcal{P}_m}(x_i,x_j)$. We use it to upper-bound the cut-distance between $W$ and $W^{\mathcal{G}_N^{(2)}}$:
\begin{align}\label{eq:preuve_maj_box}
d_{\Box,\nu}\left( W^{\mathcal{G}_N^{(2)}},W \right) &\leq \Vert   W^{\mathcal{G}_N^{(2)}}-W\Vert_{1,\nu}\notag\\
&\leq \Vert   W^{\mathcal{G}_N^{(2)}}-W^{\mathcal{G}_N^{(3)}}\Vert_{1,\nu} + \Vert   W^{\mathcal{G}_N^{(3)}}-W_{\mathcal{P}_m}\Vert_{1,\nu} + \Vert  W_{\mathcal{P}_m} - W\Vert_{1,\nu}.
\end{align}
We are going to control each term of the right hand side of \eqref{eq:preuve_maj_box} in the following steps.
\medskip

\textit{Step 2 - Control of $ \Vert  W_{\mathcal{P}_m} - W\Vert_{1,\nu}$.} We have
\begin{align*}
\Vert  W_{\mathcal{P}_m} - W\Vert_{1,\nu}&\leq \Vert W_{\mathcal{P}_m} - \widetilde{W}_{\mathcal{P}_m}  \Vert_{1,\nu}+ \Vert \widetilde{W}_{\mathcal{P}_m}   -  \widetilde{W}  \Vert_{1,\nu}+\Vert  \widetilde{W} - W\Vert_{1,\nu}.
\end{align*}
As $\Vert W-\widetilde{W}\Vert_{1,\nu}\leq \dfrac{\varepsilon}{3}$, and as for any partition $\mathcal{P}$, $\Vert W_\mathcal{P}\Vert_{1,\nu} \leq \Vert W \Vert_{1,\nu}$, we have $\Vert W_{\mathcal{P}_m} - \widetilde{W}_{\mathcal{P}_m}\Vert_{1,\nu}\leq  \dfrac{\varepsilon}{3}$ and
\begin{align*}
\Vert \widetilde{W}_{\mathcal{P}_m}   -  \widetilde{W}  \Vert_{1,\nu} &= \sum_{i,j=1}^m \int_{J_i}\int_{J_j} \left| \widetilde{W}(u,v) - m^2 \int_{J_i}\int_{J_j} \widetilde{W}(x,y)\nu(dx)\nu(dy) \right| \nu(du)\nu(dv)\\
&\leq \sum_{i,j=1}^m \int_{J_i}\int_{J_j} m^2  \int_{J_i}\int_{J_j} \left| \widetilde{W}(u,v) - \widetilde{W}(x,y)\right|\nu(dx)\nu(dy)~\nu(du)\nu(dv)
%&\leq \sum_{i,j=1}^m \int_{J_i}\int_{J_j} m^2  \int_{J_i}\int_{J_j}  \dfrac{\varepsilon}{3}\nu(dx)\nu(dy)~\nu(du)\nu(dv) 
\leq  \dfrac{\varepsilon}{3},
\end{align*}
hence $\Vert  W_{\mathcal{P}_m} - W\Vert_{1,\nu}\leq \varepsilon$ (recall here that $\nu$ is the Lebesgue measure on $[0,1]$).
\medskip

\textit{Step 3 - Control of $\Vert   W^{\mathcal{G}_N^{(2)}}-W^{\mathcal{G}_N^{(3)}}\Vert_{1,\nu}$.} For all $N\geq 1$, we recall from Lemma \ref{lem:partition_I}  $\left(B_1^{(N)},\cdots,B_N^{(N)}\right)$ the partition of $I$ with $B_i=\left( \dfrac{i-1}{N}, \dfrac{i}{N}\right]$ (we omit by simplicity the upper index $^{(N)}$). Using the notation introduced in \eqref{eq:graphonG} we have
\begin{align}\label{eq:def_F_preuve}
\Vert   W^{\mathcal{G}_N^{(2)}}-W^{\mathcal{G}_N^{(3)}}\Vert_{1,\nu}% &= \int_{I} \int_I \left| \sum_{i,j=1}^N  W(x_i,x_j) \mathbf{1}_{B_i}(u) \mathbf{1}_{B_j}(v)- \sum_{i,j=1}^N  W_{\mathcal{P}_m}(x_i,x_j) \mathbf{1}_{B_i}(u) \mathbf{1}_{B_j}(v) \right| \nu(du)\nu(dv) \notag \\
&= \sum_{i,j=1}^N \int_{B_i}\int_{B_j} \left| W(x_i,x_j) - W_{\mathcal{P}_m}(x_i,x_j) \right|  \nu(du)\nu(dv) \notag \\
&=\dfrac{1}{N^2} \sum_{i,j=1}^N  \left| W(x_i,x_j) - W_{\mathcal{P}_m}(x_i,x_j) \right| = : \dfrac{1}{N^2} \sum_{i,j=1}^N F(x_i,x_j).
\end{align}
We use the following proposition to show that it converges almost surely to $\Vert  W_{\mathcal{P}_m} - W\Vert_{1,\nu} $.
\begin{prop}[Hoeffding \cite{hoeffding1961strong}]\label{prop:hoeffding}
Let $X_1,X_2,\cdots$ be a sequence of i.i.d. random variables with distribution $\nu$, and $f$ a real-valued measurable function. Then if $\mathbb{E}\left[\left| f(X_1,X_2) \right|\right]<+\infty$,
\begin{equation}\label{eq:prop_hoeffding}
\dfrac{1}{N(N-1)}\sum\limits_{\substack{i,j=1 \\ i \neq j}}^N f(X_i,X_j)\xrightarrow[N\to \infty]{a.s.} \mathbb{E}\left[f(X_1,X_2)\right] = \int\int f(x,y)\nu(dx)\nu(dy).
\end{equation}
\end{prop}
We have indeed $ \displaystyle\dfrac{1}{N^2} \sum_{i,j=1}^N F(x_i,x_j) = \dfrac{1}{N^2} \sum_{i=1}^N F(x_i,x_i) + \dfrac{N(N-1)}{N^2} \dfrac{1}{N(N-1)} \sum\limits_{\substack{i,j=1 \\ i \neq j}}^N F(x_i,x_j),$ where the second term converges as $N\to\infty$ to $\iint F(x,y)\nu(dx)\nu(dy)~ a.s.$ and
$$\dfrac{1}{N^2} \sum_{i=1}^N F(x_i,x_i) \leq \dfrac{1}{N} \left( \dfrac{1}{N} \sum_{i=1}^N \left| W(x_i,x_i)\right|  + \dfrac{1}{N} \sum_{i=1}^N \left| W_{\mathcal{P}_m}(x_i,x_i)\right| \right) \xrightarrow[N\to \infty]{}0$$
as the sums are controlled by Hypothesis \ref{hyp:existence_lambda_barre}.
\medskip 

\textit{Step 4 - Control of $\Vert   W^{\mathcal{G}_N^{(3)}}-W_{\mathcal{P}_m}\Vert_{1,\nu}$.} We have
\begin{align}\label{eq:GN3_Pm_uniforme}
\Vert   W^{\mathcal{G}_N^{(3)}}-W_{\mathcal{P}_m}\Vert_{1,\nu} %&= \int\int \left| \sum_{i,j=1}^N W_{\mathcal{P}_m}(x_i,x_j) \mathbf{1}_{B_i\times B_j}(x,y) - W_{\mathcal{P}_m}(x,y)\right| \nu(dx)\nu(dy)\notag\\
&=  \sum_{i,j=1}^N\int_{B_i}\int_{B_j}  \left|  W_{\mathcal{P}_m}(x_i,x_j) - W_{\mathcal{P}_m}(x,y)\right| \nu(dx)\nu(dy).
\end{align}
Recalling \eqref{eq:def_Pm} and setting $\alpha_{kl}= m^2 \int_{J_k\times J_l} W(u,v) \nu(du)\nu(dv)$ we have
\begin{align*}
\Vert   W^{\mathcal{G}_N^{(3)}}-W_{\mathcal{P}_m}\Vert_{1,\nu} &= \sum_{i,j=1}^N\int_{B_i}\int_{B_j}  \left| \sum_{k,l} \alpha_{kl}\mathbf{1}_{J_k\times J_l}(x_i,x_j) - \sum_{k',l'} \alpha_{k'l'}\mathbf{1}_{J_{k'}\times J_{l'}}(x,y)\right| \nu(dx)\nu(dy)\\
&= \sum_{i,j=1}^N \sum_{k,l} \sum_{k',l'} \left| \alpha_{kl}- \alpha_{k'l'}\right| \mathbf{1}_{J_k\times J_l}(x_i,x_j) \int_{B_i}\int_{B_j} \mathbf{1}_{J_{k'}\times J_{l'}}(x,y)\nu(dx)\nu(dy)\\
&= \sum_{k,l} \sum_{k',l'} \left| \alpha_{kl}- \alpha_{k'l'}\right|\sum_{i,j=1}^N   \mathbf{1}_{J_k\times J_l}(x_i,x_j) \nu(J_{k'}\cap  B_i) \nu(J_{l'}\cap B_j).
\end{align*}
We consider $N$ large enough ($N>m$) such that every box $B_i=\left(\frac{i-1}{N},\frac{i}{N}\right]$ (of size $\frac{1}{N}$) is inside a larger box $J_{k'}=\left(\frac{k'-1}{m},\frac{k'}{m}\right]$ (of size $\frac{1}{m}$) (there might be some $B_i$ that are on two different parts of the partition $\mathcal{P}_m$, but we can neglect this contribution - at most of order $\frac{m}{N} \xrightarrow[N\to \infty]{} 0$). Then $\nu(J_{k'}\cap  B_i)=\mathbf{1}_{\left\{ B_i \subset J_{k'}\right\}}\nu(B_i)=\dfrac{1}{N}\mathbf{1}_{\left\{\frac{i}{N} \in J_{k'}\right\}}$, and
\begin{align*}
\Vert   W^{\mathcal{G}_N^{(3)}}-W_{\mathcal{P}_m}\Vert_{1,\nu} &\leq \sum_{k,l=1}^m \sum_{k',l'=1}^m \left| \alpha_{kl}- \alpha_{k'l'}\right|\sum_{i,j=1}^N  \dfrac{1}{N^2} \mathbf{1}_{\left\{(x_i,x_j) \in J_k\times J_l, B_i\subset J_{k'}, B_j \subset J_{l'}\right\}}\\
%&\leq  \sum_{k,l=1}^m \sum_{k',l'=1}^m \left| \alpha_{kl}- \alpha_{k'l'}\right|\sum_{i,j=1}^N  \dfrac{1}{N^2} \mathbf{1}_{\left\{(x_i,x_j) \in J_k\times J_l, \frac{i}{N} \in J_{k'}, \frac{j}{N} \in J_{l'}\right\}}\\
&\leq \sum_{k,l=1}^m \sum_{k',l'=1}^m \left| \alpha_{kl}- \alpha_{k'l'}\right| \left(\dfrac{1}{N}\sum_{i=1}^N   \mathbf{1}_{\left\{x_i\in J_k, \frac{i}{N} \in J_{k'}\right\}}\right) \left(\dfrac{1}{N}\sum_{j=1}^N   \mathbf{1}_{\left\{x_j\in J_l, \frac{j}{N} \in J_{l'}\right\}}\right).
\end{align*}
Then, from Lemma \ref{lem:quantile_empirical_uniforme}, $\dfrac{1}{N}\sum_{i=1}^N   \mathbf{1}_{\left\{x_i\in J_k, \frac{i}{N} \in J_{k'}\right\}}\xrightarrow[N\to \infty]{a.s.} \lambda(J_k\cap J_{k'})= \dfrac{1}{m} \mathbf{1}_{k=k'}$, hence we obtain that
\[\limsup_{ N\to\infty} \left\Vert W^{ \mathcal{ G}_{ N}^{ (3)}} -W_{  \mathcal{ P}_{ m}}\right\Vert_{ 1, \nu} \leq \frac{ 1}{ m^{ 2}}\sum_{ k,l=1}^{ m} \sum_{ k^{ \prime}, l^{ \prime}=1}^m \left\vert \alpha_{ k,l}- \alpha_{ k^{ \prime}, l^{ \prime}} \right\vert \mathbf{ 1}_{ k=k^{ \prime}}\mathbf{ 1}_{ l=l^{ \prime}}.\]
The claim is that the above bound is uniformly $0$ for all $m$: the sum reduces to $k=k^{ \prime}$ and $l=l^{ \prime}$ hence the prefactor $ \left\vert \alpha_{ k,l}- \alpha_{ k^{ \prime}, l^{ \prime}} \right\vert$ gives that this last contribution is $0$, thus almost-surely (on the realisation of the sequence of positions) we have $\Vert   W^{\mathcal{G}_N^{(3)}}-W_{\mathcal{P}_m}\Vert_{1,\nu}\xrightarrow[N\to \infty]{}0$.
\medskip

\textit{Conclusion when the positions are uniformly drawn - }From \eqref{eq:preuve_maj_box} and Steps 3 and 4, we obtain that $\limsup_{ N\to \infty} d_{ \Box} \left( W^{ \mathcal{ G}_{ N}^{ (2)}}, W\right) \leq 2\left\Vert W_{ \mathcal{ P}_{ m}} -W\right\Vert_{ 1, \nu}$. Choosing now $m$ as in Step 2 gives that $\limsup_{ N\to \infty} d_{ \Box} \left( W^{ \mathcal{ G}_{ N}^{ (2)}}, W\right) \leq \varepsilon$ for all $ \varepsilon>0$,  which concludes the proof for the case $x_i \sim \mathcal{U}(0,1)$.
 \medskip
 
\textit{Generalisation: from $[0,1]$ to $[0,1]^d$ - } Consider the case $x_i=\left(u_i^{(1)}, \cdots, u_i^{(d)}\right)$ where $\left(u_i^{(j)}\right)_{1\leq j \leq d}$ are drawn uniformly on $(0,1]$ (but not necessarily independent), and the partition $\displaystyle I=(0,1]^d = \bigsqcup_{i=1}^{N}B_i = \bigsqcup_{i=1}^{N} \left(  \left(\frac{i-1}{N},\frac{i}{N}\right] \times (0,1]^{d-1} \right)$. Proposition \ref{prop:hoeffding} still apply, and the treatment of the terms  $\Vert   W^{\mathcal{G}_N^{(2)}}-   W^{\mathcal{G}_N^{(3)}}\Vert_{1,\nu}$ and $\Vert W_{\mathcal{P}_m}-W\Vert_{1,\nu}$ in \eqref{eq:preuve_maj_box} remains the same. For the term $\Vert   W^{\mathcal{G}_N^{(3)}}-W_{\mathcal{P}_m}\Vert_{1,\nu}$, it suffices to note that the chosen partition $\bigsqcup B_i$ only affects the first coordinates to conclude by the same arguments.

\textit{General case-} Consider $\nu$ absolutely continuous w.r.t. Lebesgue measure, and $I \subset \mathbb{R}^d$. From Sklar's theorem (see Theorem 2.3.3 of \cite{Nelsen1999}) we have:
$$ f_{\nu}(x^{(1)},\cdots,x^{(d)})=c(F_{1}(x^{(1)}),\cdots,F_{d}(x^{(d)}))f_{1}(x^{(1)})\cdots f_{d}(x^{(d)}),$$
where $c$ is the copula density function of $\nu$, $f_{i}$ the $i$-th marginal probability density functions, $F_{i}$ the $i$-th marginal cumulative distribution functions and $f_{\nu}$ the density of $ \nu$ w.r.t. Lebesgue: $\nu(dx)=f_{\nu}(x^{(1)},\cdots,x^{(d)})dx^{(1)}\cdots dx^{(d)}$. It implies, by the change of variables  $u=\left(F_1(x^{(1)}),\cdots,F_d(x^{(d)})\right)$ that $c(u)du=f_\nu(x)dx$. Define also $u_i=\left(F_1(x_i^{(1)},\cdots,F_d(x_i^{(d)})\right)$ and $$W_F\left( u,v \right):= W\left( \left(F_1^{-1}(u^{(1)}),\cdots,F_d^{-1}( u^{(d)})\right), \left(F_1^{-1}(v^{(1)}),\cdots,F_d^{-1}( v^{(d)})\right)\right),$$ the previous change of variable gives then with $B_i:=\left( F_1^{-1}\left(\frac{i-1}{N},\frac{i}{N}\right]\times F_2^{-1}\left((0,1]\right) \times \cdots \times F_d^{-1}\left( (0,1] \right) \right)$,
(note that this partition corresponds to the one introduced in Lemma \ref{lem:partition_I})
\begin{align*}
\Vert   W^{\mathcal{G}_N^{(2)}}-W\Vert_{1,\nu} &= \sum_{i,j=1}^N \int_{B_i}\int_{B_j} \left| W(x_i,x_j)-W(x,y)\right| \nu(dx)\nu(dy)\\
&= \sum_{i,j=1}^N \int_{\left(\frac{i-1}{N},\frac{i}{N}\right]\times  (0,1]^{d-1}}\int_{\left(\frac{j-1}{N},\frac{j}{N}\right]\times (0,1]^{d-1}}\left| W_F(u_i,u_j)-W_F(u,v)\right| c(u)c(v)du ~dv.
\end{align*}
The previous case gives immediately the result.

\paragraph{Proof of \eqref{eq:lem_S_box_max} }
We remind that we want to prove $\Vert W^{\mathcal{G}_N^{(2)}}-W \Vert_{\infty	\to \infty, \nu} \xrightarrow[N\to\infty]{} 0$. 
As in the proof of \eqref{eq:lem_S_box}, we start with the case $I=[0,1]$, $\widetilde{x_i}\sim \mathcal{U}\left(0,1\right)$ i.i.d. (then $\nu$ is the Lebesgue measure). What changes is that we no longer integrate with respect to the first variable, but we take the supremum. The approximation in $L^1(I^2)$ is not adapted anymore, thus we approximate $W$ differently. Recall that $$\Vert W^{\mathcal{G}_N^{(2)}}-W \Vert_{\infty	\to \infty, \nu}  = \sup_{g,\Vert g \Vert_\infty\leq 1} \sup_{u\in I} \left| \int_I \left( W^{\mathcal{G}_N^{(2)}}(u,v) - W(u,v) \right) g(v) \nu(dv) \right|.$$
\medskip 

\textit{Step 1 - A first bound.} Fixing $g$ such that $\Vert g \Vert_\infty\leq 1$ and $u \in I$, for any $N$ there exists a unique $i$ such that $u\in B_i^{(N)}=\left(\dfrac{i-1}{N},\dfrac{i}{N}\right]$. Then
\begin{align*}
\left| \int_I \left( W^{\mathcal{G}_N^{(2)}}(u,v) - W(u,v) \right) g(v) \nu(dv) \right| %&= \left| \int_I \left( \sum_{j=1}^N W(x_i,x_j)\mathbf{1}_{B_j}(v) - W(u,v) \right) g(v) \nu(dv) \right| \\
& = \left| \sum_{j=1}^N \int_{B_j} \left( W\left(x_i,x_j\right)-W(u,v)\right)g(v)\nu(dv) \right|\\
&\leq   \left| \sum_{j=1}^N \int_{B_j} \left( W\left(x_i,x_j\right)-W(x_i,v)\right)g(v)\nu(dv) \right| \\&+ \left|  \int_{I} \left( W(x_i,v)-W(u,v)\right)g(v)\nu(dv) \right| =: A(g,u) + B(g,u).
\end{align*}
\medskip

\textit{Step 2 - Upper-bound of $A(g,u)$ by approximated functions independent of $g$.} As $\Vert g \Vert_\infty\leq 1$, we have $A(g,u)\leq \sum_{j=1}^N \int_{B_j} \left| W\left(x_i,x_j\right)-W(x_i,v)\right| dv$. Note that is does not depend anymore on $g$ and it depends on $u$ only by  the index $i$. To control this term, we first approximate $W$ by a stepfunction in $L^1(I)$, $\widetilde{W}_{\mathcal{P}_m}$. Introduce $\left(\varphi_\eta\right)_{\eta>0}$ as $\varphi_\eta(x)=\eta^{-1} \phi(\frac{x}{\eta})$ where $\phi$ is a  non-negative continuous function of $I$ with $\int\phi=1$. Define for all $x\in I$ $\widetilde{W}_\eta(x,\cdot):=W(x,\cdot)\ast \varphi_\eta$. Note that $y\mapsto \widetilde{W}_\eta(x,y) \in \mathbb{R}$ is a continuous function for all $x\in [0,1]$. As for any $(x,x')\in I^2$, $ \Vert W(x,\cdot) - W(x',\cdot) \Vert_1 \leq C_w \Vert x-x'\Vert^\vartheta$ using \eqref{eq:hyp_W_pseudolip_theta}, $x\mapsto W(x,\cdot)$ is continuous from $[0,1]$ to $L^1(I)$, so that the set of functions $F:=\left\{ W(x,\cdot), x\in [0,1]\right\}$ is compact. Hence, for $\varepsilon>0$, we can find $p\geq 1$ and $p$ positions $y_1,\cdots,y_p$ such that $F\subset \cup_{k=1}^p B_{L_1}\left( W(y_k,\cdot), \varepsilon\right)$. Then, there exists $\eta>0$ such that for all $k\leq p$, $\Vert \widetilde{W}_\eta(y_k,\cdot)-W(y_k,\cdot)\Vert_{I,1}\leq \varepsilon$. From now, we may omit the notation $\eta$ for $\widetilde{W}$.
Let $m\geq 1$ and $\mathcal{P}_m=\sqcup_{i=1}^{m}J_i$ for $J_r=\left( \dfrac{r-1}{m},\dfrac{r}{m}\right]$ the regular partition of $I$ of order $m$. For any kernel $H$ on $I^2$, define
\begin{equation}\label{eq:def_tilde_PM}
H_{\mathcal{P}_m}\left(x,v \right):= m\sum_{r=1}^m \left( \int_{J_r} H(x,y)dy \right) \mathbf{1}_{J_r}(v).
\end{equation} 
The function $H\mapsto H_{\mathcal{P}_m}$ is continuous: $\Vert H_{\mathcal{P}_m}\Vert_{L^1(I^2)}\leq \Vert H \Vert_{L^1(I^2)}$. Note that this definition is different from the one used in the proof of \eqref{eq:lem_S_box} where we integrated on both variables.
By continuity of $y\mapsto \widetilde{W}(y_k,y)$ for all $k=1\cdots p$, there exists $m\geq 1$ such that $ \sup_{1\leq l \leq p} \Vert \widetilde{W}\left(y_l,\cdot\right)-\widetilde{W}_{\mathcal{P}_m}\left(y_l,\cdot\right) \Vert_\infty \leq \varepsilon$, and thus $\displaystyle \sup_{1\leq l \leq p} \int \left| \widetilde{W}\left(y_l,y\right)-\widetilde{W}_{\mathcal{P}_m}\left(y_l,y\right) \right| dy \leq \varepsilon$. Then, for any $x\in I$,
\begin{align*}
\Vert W(x,\cdot) - \widetilde{W}_{\mathcal{P}_m}(x,\cdot)\Vert_{I,1} &\leq \underbrace{\Vert W(x,\cdot) - W(y_l,\cdot)\Vert_{I,1}}_{\leq \varepsilon \text{ by the cover of }F}  + \underbrace{\Vert  W(y_l,\cdot)- \widetilde{W}(y_l,\cdot)\Vert_{I,1}}_{\leq \varepsilon \text{ by the choice of }\eta} \\& + \underbrace{\Vert \widetilde{W}(y_l,\cdot) - \widetilde{W}_{\mathcal{P}_m}(y_l,\cdot)\Vert_{I,1}}_{\leq \varepsilon \text{ by the choice of }m}  +  \Vert \widetilde{W}_{\mathcal{P}_m}(y_l,\cdot)  - \widetilde{W}_{\mathcal{P}_m}(x,\cdot)\Vert_{I,1} \\
&\leq 3 \varepsilon +  \Vert \widetilde{W}(y_l,\cdot)  - \widetilde{W}(x,\cdot)\Vert_{I,1} \leq 4\varepsilon,
\end{align*}
where we used the fact that for any partition $\mathcal{P}$, $\Vert \widetilde{W}_\mathcal{P} \Vert_{I,1}\leq \Vert \widetilde{W}\Vert_{I,1}$ and %$\Vert \left( W(x,\cdot)-W(y_l,\cdot) \right)\ast \varphi_\eta \Vert_{I,1} \leq \Vert  W(x,\cdot)-W(y_l,\cdot) \Vert_{I,1} . \Vert \varphi_\eta \Vert_{I,1} \leq \varepsilon$.
because $$\left\Vert \widetilde{W}(y_{ l}, \cdot) - \widetilde{W}(x, \cdot)\right\Vert_{ I, 1}= \left\Vert  \left(W(y_{ l}, \cdot) - W(x, \cdot)\right)\ast \varphi_{ \eta}\right\Vert_{ I, 1}\leq \left\Vert W(y_{ l}, \cdot) - W(x, \cdot)\right\Vert_{ I, 1}\left\Vert \varphi_{ \eta}\right\Vert_{ I, 1} $$ by Young's inequality. By compactness of $F$ and since $ B_{ L^{ 1}} (W(y_{ l}, \cdot))_{ l=1,\cdots, p}$ is an $ \varepsilon$-covering of $F$, this last term is smaller than $ \varepsilon$.
Using this approximation, we can now upper bound $A(g,u)$ independently of the choice of $g$ and relying on the choice of $u$ only by the index $i$ such that $u\in B_i^{(N)}$: we have
\begin{align*}
A(g,u) &\leq \dfrac{1}{N} \sum_{j=1}^N \left| W(x_i,x_j) - \widetilde{W}_{\mathcal{P}_m}(x_i,x_j)\right| + \sum_{j=1}^N \int_{B_j} \left| \widetilde{W}_{\mathcal{P}_m}(x_i,x_j) - \widetilde{W}_{\mathcal{P}_m}(x_i,v)\right|dv \\&+ \sum_{j=1}^N \int_{B_j} \left| \widetilde{W}_{\mathcal{P}_m}(x_i,v)-W(x_i,v)\right|dv =: A_1^{(i)} + A_2^{(i)} + A_3^{(i)}.
\end{align*}
\medskip

\textit{Step 3 - Uniform control of the $A_k^{(i)}$.} As $A_3^{(i)}= \Vert \widetilde{W}_{\mathcal{P}_m}(x_i,\cdot) - W(x_i,\cdot)\Vert_{I,1}$, we control it by the work done previously  independently of the index $i$ (see Step 2): $\sup_i A_3^{(i)} \xrightarrow[N\to\infty]{} 0$. Set $g(x,y):=W(x,y)-\widetilde{W}_{\mathcal{P}_m}(x,y)$, and as $W\in L^\chi(I^2)$, so does $g$. We can then apply Proposition \ref{prop:toolbox_iid}  and we obtain  $$\sup_{1\leq i \leq N} \left| A_3^{(i)} - A_1^{(i)}\right| = \sup_{1\leq i \leq N} \int_I g(x_i,y)\left(\nu^{(N)}(dy)-\nu(dy)\right)\xrightarrow[N\to\infty]{} 0.$$
We focus now on $A_2^{(i)}$ and show that $\displaystyle \sup_{x} \sum_{j=1}^N \int_{B_j} \left| \widetilde{W}_{\mathcal{P}_m}(x,x_j)- \widetilde{W}_{\mathcal{P}_m}(x,v)\right| dv$ tends to 0: denoting by $\alpha_k(x)=m\int_{J_k}\widetilde{W}(x,y)dy$, we have
\begin{align*}
\sum_{j=1}^N\int_{B_j}  \left| \widetilde{W}_{\mathcal{P}_m}(x,x_j)- \widetilde{W}_{\mathcal{P}_m}(x,v)\right| dv &= \sum_{j=1}^N  \int_{B_j}\left| \sum_{k=1}^m \alpha_k(x) \mathbf{1}_{J_k}(x_j) - \sum_{k'=1}^m \alpha_{k'}(x)\mathbf{1}_{J_{k'}}(v) \right|dv\\
&\leq \sum_{k,k'=1}^m \left| \alpha_k(x) - \alpha_{k'}(x)\right| \sum_{j=1}^N \mathbf{1}_{J_k}(x_j) \left|J_{k'}\cap B_j\right|.
\end{align*}
Similarly to what has been done in Step 4 for the proof of \eqref{eq:lem_S_box}, we consider $N$ large enough ($N>m$) such that every box $B_i=\left(\frac{i-1}{N},\frac{i}{N}\right]$ is inside a larger box $J_{k'}=\left(\frac{k'-1}{m},\frac{k'}{m}\right]$, then $\nu(J_{k'}\cap  B_j)=\mathbf{1}_{\left\{ B_j \subset J_{k'}\right\}}\nu(B_j)=\dfrac{1}{N} \mathbf{1}_{\left\{\frac{j}{N} \in J_{k'}\right\}}$ and
$$\sum_{j=1}^N\int_{B_j}  \left| \widetilde{W}_{\mathcal{P}_m}(x,x_j)- \widetilde{W}_{\mathcal{P}_m}(x,v)\right| dv \leq \sum_{\substack{k,k'=1\\k\neq k'}}^m \left| \alpha_k(x) - \alpha_{k'}(x)\right| \sum_{j=1}^N  \dfrac{1}{N}\mathbf{1}_{\left\{x_j \in J_k, \frac{j}{N}\in J_{k'}\right\}}.$$
As $\alpha_k(x)\leq m\int_I \widetilde{W}(x,y)dy \leq mC_W^{(1)}$ which is independent of $x$ and $k$, 
$$ \sum_{j=1}^N\int_{B_j}  \left| \widetilde{W}_{\mathcal{P}_m}(x,x_j)- \widetilde{W}_{\mathcal{P}_m}(x,v)\right| dv \leq 2mC_W^{(1)}\sum_{\substack{k,k'=1\\k\neq k'}}^m \sum_{j=1}^N  \dfrac{1}{N}\mathbf{1}_{\left\{x_j \in J_k, \frac{j}{N}\in J_{k'}\right\}}.$$
From Lemma \ref{lem:quantile_empirical_uniforme}, $\dfrac{1}{N}\sum_{j=1}^N   \mathbf{1}_{\left\{x_j\in J_k, \frac{j}{N} \in J_{k'}\right\}}\xrightarrow[N\to \infty]{a.s.} \lambda(J_k\cap J_{k'})= \dfrac{1}{m} \mathbf{1}_{k=k'}$, thus almost-surely (on the realisation of the sequence of positions) $\displaystyle \sum_{j=1}^N\int_{B_j}  \left|  \widetilde{W}_{\mathcal{P}_m}(x,x_j)-  \widetilde{W}_{\mathcal{P}_m}(x,v)\right| dv $ tends to 0 independently on the choice of $x$. We have shown that $ \displaystyle \sup_{g,\Vert g \Vert_\infty\leq 1} \sup_{u\in I} A(g,u) \xrightarrow[N\to\infty]{}0$  $\mathbb{P}$-almost surely.
\medskip

\textit{Step 4 - Control of $B(g,u)$ and conclusion.} Using\eqref{eq:hyp_W_pseudolip_theta}  from Hypothesis \ref{hyp:existence_lambda_barre}, we have
$$B(g,u) \leq \int_I \left| W(x_i,v)-W(u,v) \right| \Vert g \Vert_\infty \nu(dv) \leq  C_w \Vert x_i - u\Vert ^\vartheta.$$
Let us show that $\displaystyle \sup_{x\in I} \sum_{i=1}^N \mathbf{1}_{B_i}(x) \Vert x_i-x \Vert ^\vartheta \xrightarrow[N\to\infty]{}0$. Recall \eqref{eq:def_U_stat}: we have $\displaystyle\sum_{i=1}^N \mathbf{1}_{B_i}(x) \Vert x_i-x \Vert ^\vartheta=\Vert U_N(x)-x\Vert^\vartheta$ by definition of $U_N$, the uniform sample quantile function. As we know from \cite{csorgHo1983quantile} that $\displaystyle\sup_{0 \leq y \leq 1} \vert U_N(y) - y \vert \xrightarrow[N\to \infty]{a.s.} 0$,  almost surely $ \displaystyle \sup_{g,\Vert g \Vert_\infty\leq 1} \sup_{u\in I} B(g,u) \xrightarrow[N\to\infty]{}0$.
It concludes the proof for \eqref{eq:lem_S_box_max}.
%----------------------------------------------------------------------------------------
%----------------------------------------------------------------------------------------
\paragraph{Proof of \eqref{eq:lem_S_A5} and \eqref{eq:lem_S_A5_max}} The term of interest is $A^{(N)}_{i,T,5}$, defined in \eqref{eq:def_A5}, we have by Jensen's inequality

\begin{align*}
%A^{(N)}_{i,T,5} &\leq \int_0^T \left | \int_I  F\left( x_i,y,s\right) \left( \nu^{(N)}(dy)-\nu(dy)\right) \right| ds,\\
{A^{(N)}_{i,T,5}}^2 &\leq \left( \int_0^T \left | \int_I  F\left( x_i,y,s\right) \left( \nu^{(N)}(dy)-\nu(dy)\right) \right| ds\right)^2\\
&\leq T \int_0^T \left( \dfrac{1}{N} \sum_{j=1}^N \left( F(x_i,x_j,s) - \int_I F(x_i,y,s)\nu(dy) \right) \right)^2 ds.\\
&\leq \dfrac{T}{N^2} \sum_{j,l=1}^N \int_0^T \left( F\left( x_i,x_j,s\right)F\left( x_i,x_l,s\right)+\left( \int_I F(x_i,y,s)\nu(dy)\right)^2 \right.\\
& \hspace{3cm} \left. -2 F\left( x_i,x_j,s\right)  \int_I F(x_i,y,s)\nu(dy) \right)ds.
\end{align*}
As $ F\left( x_i,y,s\right)F\left( x_i,z,s\right)=W(x_i,y)W(x_i,z) \gamma(s,y)\gamma(s,z)$ for any $y$ and $z$, denoting by $\Gamma(y,z) := \int_0^T \gamma(s,y)\gamma(s,z) ds$ we obtain
\begin{align}\label{eq:lem_maj_Ai5_S1}
{A^{(N)}_{i,T,5}}^2 &\leq \dfrac{T}{N^2} \sum_{j,l=1}^N  W(x_i,x_j)W(x_i,x_l)\Gamma(x_j,x_l) - \dfrac{2T}{N}\sum_{j=1}^N \int_I W(x_i,x_j)W(x_i,y)\Gamma(x_j,y)\nu(dy)\notag\\
&\hspace{3cm} + T \int_{I^2} W(x_i,y)W(x_i,z)\Gamma(y,z)\nu(dy)\nu(dz)\notag\\
&= T \int_{I\times I} W(x_i,y)W(x_i,z)\Gamma(y,z) \left( \nu^{(N)}(dy)\nu^{(N)}(dz) -2\nu^{(N)}(dy)\nu(dz)+\nu(dy)\nu(dz)\right) \notag\\
 &= T \left(\epsilon_{i,1} + \epsilon_{i,2}\right)
\end{align}
where $\epsilon_{i,1}$ and $\epsilon_{i,2}$ are defined and studied in Corollary \ref{cor:toolbox_iid}. Taking the square root and then summing on $i$ or taking the supremum, \eqref{eq:lem_S_A5}  and \eqref{eq:lem_S_A5_max} follow.

\begin{remark}\label{rem:regW_A5} If we ask for more regularity of $W$, we can have a more direct proof of \eqref{eq:lem_S_A5}. Assume that there exist $L_W>0$ and $M_W>0$ such that $$ \sup_{x\in I} \sup_{y \neq y'} \dfrac{\vert W(x,y) - W(x,y') \vert}{\Vert y - y' \Vert}\leq L_W \text{ and } \sup_{x,y\in I} \vert W(x,y) \vert \leq M_W.$$ Hypothesis \ref{hyp:existence_lambda_barre} is trivially satisfied with $\vartheta=1$ and $C_w=L_W$, which implies that $\lambda$ is uniformly Lipschitz continuous in the second variable (in \eqref{eq:pt_fixe_lips_espace}, $\phi(x)= 2\Vert x \Vert$).
We show first that $F$ defined above in \eqref{eq:def_F_preuve} is also uniformly Lipschitz continuous in the second variable: for any $(x,y,y',s) \in I^3\times [0,T]$, 
$$F\left(x,y,s\right) - F\left(x,y',s\right)= \left(  W(x,y) - W(x,y') \right) \gamma(s,y) + W(x,y') \int_0^s h(s-u)\left( \lambda(u,y)-\lambda(u,y') \right) du,$$
then
\begin{align*}
\left| F\left(x,y,s\right) - F\left(x,y',s\right) \right| &\leq \left|  W(x,y) - W(x,y') \right | \vert\gamma(s,y)\vert \\&\quad+ \vert W(x,y') \vert \int_0^s \vert h(s-u) \vert \left| \lambda(u,y)-\lambda(u,y') \right| du\\
&\leq \Vert h \Vert_{T,1} \Vert y - y' \Vert \left( \Vert \lambda \Vert_{[0,T]\times I, \infty} L_W + 2 M_W C_\lambda \right) =: L_F \Vert y - y' \Vert ,
\end{align*}
with $L_F>0$ independent of the choice of $s$ and $x$.
As $A^{(N)}_{i,t,5}= \int_0^t \left | \int_I  F\left( x_i,y,s\right) \left( \nu^{(N)}(dy)-\nu(dy)\right) \right| ds$ and $F$ is uniformly Lipschitz continuous in the second variable with constant $L_F$, we have 
$$\dfrac{1}{N}\sum_{i=1}^N A^{(N)}_{i,T,5} \leq L_F  \sup_{g \in BL }\int_0^T \left | \int_I g(y)\left( \nu^{(N)}(dy)-\nu(dy)\right) \right| ds \leq TL_F d_{BL} \left(  \nu^{(N)},\nu \right)\xrightarrow[N\to \infty]{} 0$$
by Varadarajan Theorem (see \cite{Dudley2002} Theorem 11.4.1 and \cite{varadarajan1958convergence}).
\end{remark}
%----------------------------------------------------------------------------------------%----------------------------------------------------------------------------------------
\subsubsection{Proof of Proposition \ref{prop:scenario_hyp} for Scenario (2)} \label{S:scenario_2}
Recall that  $I=[0,1]$, $x_i^{(N)}=\frac{i}{N}$, and $\nu(dx)=dx$. We focus on the case $W$ continuous. When $W$ is piecewise continuous, the same results follow as we can work on each rectangle where $W$ can be extended to a continuous function, and these rectangles are in finite number.
%-----------------------------------------------------------
%-----------------------------------------------------------
\paragraph{Proof of \eqref{eq:lem_S_box}}
Using Remark \ref{rem:equ_norm_graphon} and \eqref{eq:normeinf1_graphon}, we have
\begin{align*}
d_{\Box}\left( W^{\mathcal{G}_N^{(2)}},W \right) &\leq \Vert W^{\mathcal{G}_N^{(2)}} - W \Vert_{\infty\to 1}  = \sup_{\Vert g \Vert_{\infty}\leq 1} \int \left| \int \left(W^{\mathcal{G}_N^{(2)}} - W \right)(x,y) g(y)  dy \right| dx\\
&\leq \int \int \left| \left(W^{\mathcal{G}_N^{(2)}} - W \right)(x,y) \right| dxdy = \Vert W^{\mathcal{G}_N^{(2)}} - W \Vert_{L^1,[0,1]^2}\\
&= \sum_{i,j=1}^N \int_{\frac{i-1}{N}}^{\frac{i}{N}} \int_{\frac{j-1}{N}}^{\frac{j}{N}} \left | W\left(\frac{i}{N},\frac{j}{N}\right) - W(x,y) \right| dxdy.
\end{align*}
As $W$ is continuous on the compact $[0,1]^2$ in this scenario (2), it is uniformly continuous due to Heine-Cantor theorem thus for any $\varepsilon >0$, there exists $\eta>0$ such that $\vert x - x' \vert + \vert y - y' \vert \leq \eta \Rightarrow \left| W\left(x,y\right)  - W\left(x',y'\right)  \right| < \varepsilon$. For $N$ large enough, $\frac{1}{N} < \eta$ and then \eqref{eq:lem_S_box} holds as $d_{\Box}\left( W^{\mathcal{G}_N^{(2)}},W \right) \leq \sum_{i,j=1}^N \int_{\frac{i-1}{N}}^{\frac{i}{N}} \int_{\frac{j-1}{N}}^{\frac{j}{N}}  \varepsilon ~ dxdy = \varepsilon$.
%----------------------------------------------------------------------------------------%----------------------------------------------------------------------------------------
\paragraph{Proof of \eqref{eq:lem_S_box_max}}
Recall that $$\Vert W^{\mathcal{G}_N^{(2)}}-W \Vert_{\infty	\to \infty, \nu}  = \sup_{g,\Vert g \Vert_\infty\leq 1} \sup_{u\in [0,1]} \left| \int_0^1 \left( W^{\mathcal{G}_N^{(2)}}(u,v) - W(u,v) \right) g(v) dv \right|.$$
As done for \eqref{eq:lem_S_box}, we use the uniform continuity of $W$: for any $\varepsilon >0$, we take $\eta>0$ such that  $\vert x - x' \vert + \vert y - y' \vert \leq \eta \Rightarrow \left| W\left(x,y\right)  - W\left(x',y'\right)  \right| < \varepsilon$. Fix $g$ such that $\Vert g \Vert_\infty\leq 1$ and $u \in ]0,1]$, for any $N$ there exists a unique $i$ such that $u\in B_i^{(N)}=\left(\dfrac{i-1}{N},\dfrac{i}{N}\right]$. For $N$ large enough, $\frac{2}{N} < \eta$ and we have then
\begin{align*}
\left| \int_0^1 \left( W^{\mathcal{G}_N^{(2)}}(u,v) - W(u,v) \right) g(v) dv \right|% &= \left| \int_0^1 \left( \sum_{j=1^N} W(x_i,x_j)\mathbf{1}_{B_j}(v) - W(u,v) \right) g(v) dv \right|\\
&= \left| \sum_{j=1}^N \int_{B_j} \left( W\left(\dfrac{i}{N},\dfrac{j}{N}\right)-W(u,v)\right)g(v)dv \right|\\
&\leq \sum_{j=1}^N \int_{\frac{j-1}{N}}^\frac{j}{N} \left| W\left(\dfrac{i}{N},\dfrac{j}{N}\right)-W(u,v) \right| \vert g(v) \vert dv \leq \varepsilon,
\end{align*}
independently from the choices of $g$ and $u$: we have shown that $\Vert W^{\mathcal{G}_N^{(2)}}-W \Vert_{\infty	\to \infty, \nu} \xrightarrow[N\to\infty]{} 0$ for this Scenario.
%----------------------------------------------------------------------------------------%----------------------------------------------------------------------------------------
\paragraph{Proof of \eqref{eq:lem_S_A5} and \eqref{eq:lem_S_A5_max}}
As $W$ is continuous on $[0,1]^2$ and $\left(s,y\right) \mapsto \gamma(s,y) = \int_0^s h(s-u) \lambda(u,y)du$ is also continuous on $[0,T]\times [0,1]$ as a convolution between $h$ locally integrable and $\lambda$ continuous, the application  $\left(x,y,s\right) \mapsto F\left(x,y,s\right) = W(x,y)\gamma(s,y)$ is continuous on the compact set $K=[0,1]\times[0,1]\times[0,T]$, it is uniformly continuous due to Heine-Cantor theorem. Then, for $\varepsilon >0$, there exists $\eta>0$ such that for any $\left(x,y,s\right)$ and $\left(x',y',s'\right)$ in $K$, $\vert x - x' \vert + \vert y - y' \vert + \vert s - s' \vert \leq \eta \Rightarrow \left| F\left(x,y,s\right)  - F\left(x',y',s'\right)  \right| < \epsilon$. For $N$ large enough, $\frac{1}{N} < \eta$ and we have then
\begin{align}\label{eq:lem_maj_Ai5_S2}
A^{(N)}_{i,T,5}&= \int_0^T \left | \sum_{j=1}^N \int_{\frac{j-1}{N}}^{\frac{j}{N}} F\left( x_i,x_j,s\right)dy -\sum_{j=1}^N \int_{\frac{j-1}{N}}^{\frac{j}{N}} F(x_i,y,s)dy \right|ds\notag\\
&\leq  \int_0^T \sum_{j=1}^N  \int_{\frac{j-1}{N}}^{\frac{j}{N}}  \left | F\left( x_i,x_j,s\right) - F\left( x_i,y,s\right) \right| dy~ds \leq T\varepsilon.
\end{align}
Summing on $i$ or taking the supremum, \eqref{eq:lem_S_A5}  and \eqref{eq:lem_S_A5_max} follow.
%----------------------------------------------------------------------------------------
%----------------------------------------------------------------------------------------
%----------------------------------------------------------------------------------------
\section{Proofs: the empirical measure and the spatial profile}

\subsection{Proof of Theorem \ref{thm:cvg_dBL}}\label{S:proof_thm:cvg_dBL}

We prove the convergence of $\mathbf{E}\left[d_{BL}(\mu_N,\mu_\infty)\right] \xrightarrow[N\to\infty]{}0$.
Some of the following arguments come from \cite{CHEVALLIER20191}. We consider $\mathbb{D}\left([0,T],\mathbb{N}\right)$ with the distance $d_0$ introduced in \cite{billingsley} (§14) which makes it complete, and we have for any $\eta,\zeta$ in $\mathbb{D}\left([0,T],\mathbb{N}\right)$ , $d_0(\eta,\zeta)\leq \sup_{t\leq T} \vert \eta(t) - \zeta(t) \vert$. Recall that $ d_{BL}\left(\mu_N,\mu_\infty\right) = \sup_{\phi, \Vert \phi \Vert _ {BL}\leq 1} \left| \int \phi d\mu_N - \int \phi d\mu_\infty\right|$. We start by proving that for any  $\phi$  fixed, $\mathbf{E}\left| \int \phi\left( d\mu_N -d\mu_\infty\right)\right| \xrightarrow[N\to\infty]{} 0$. By an argument of compactness, we show how it implies \eqref{eq:thm_cvg_dBL}.
\medskip

\textit{Step 1 - Convergence when $\phi$ is fixed.} 
We fix $\phi$ a real-valued function on $S$ (recall that $S:=\mathbb{D}\left([0,T],\mathbb{N}\right)\times I$) such that $ \Vert \phi \Vert _{BL}\leq 1$. Then with the coupling introduced in Definition \ref{def:couplageZZbarre}:
\begin{align*}
\mathbf{E}\left| \int \phi\left( d\mu_N -d\mu_\infty\right)\right| &= \mathbf{E}\left| \dfrac{1}{N} \sum_{i=1}^N \phi\left(Z_i^{(N)}, x_i\right) - \int \phi(\eta, x) P_{[0,T],\infty}\left( d\eta \vert x\right) \nu(dx) \right| \\
&\leq \mathbf{E}\left| \dfrac{1}{N} \sum_{i=1}^N \left( \phi\left(Z_i^{(N)}, x_i\right) - \phi\left(\overline{Z}_i,x_i\right)\right) \right| \\
&+ \mathbf{E}\left| \dfrac{1}{N} \sum_{i=1}^N \left(\phi\left(\overline{Z}_i,x_i\right) - \int \phi\left(\eta, x_i\right)  P_{[0,T],\infty}\left( d\eta \vert x_i\right)\right) \right|\\
&+ \mathbf{E}\left| \dfrac{1}{N} \sum_{i=1}^N \int \phi\left(\eta, x_i\right)  P_{[0,T],\infty}\left( d\eta \vert x_i\right)  - \int \phi(\eta, x) P_{[0,T],\infty}\left( d\eta \vert x\right) \nu(dx) \right|\\&:=A+B+C.
\end{align*}
The term $A$ is treated easily with Theorems \ref{thm:cvg_0} or \ref{thm:cvg-sup_0}: as $\phi$ is Lipschitz continuous and $\Vert \phi \Vert_L \leq 1$, 
$$A\leq  \dfrac{1}{N} \sum_{i=1}^N \mathbf{E} \left[ d_0\left( Z_i^{(N)}, \overline{Z}_i\right) \right] \leq \dfrac{1}{N} \sum_{i=1}^N \mathbf{E} \left[ \sup_{t\leq T} \left| Z_i^{(N)}(t) -  \overline{Z}_i(t)\right| \right] \xrightarrow[N\to\infty]{} 0.$$
To treat $B$, we set for each $i\in\llbracket 1,N \rrbracket$ $G_i:=\phi\left(\overline{Z}_i,x_i\right)$, it is a random variable with expectation $\int \phi\left(\eta, x_i\right)  P_{[0,T],\infty}\left( d\eta \vert x_i\right)$. We have then applying Lemma \ref{lem:maj_famille_indep}, $B\leq \dfrac{1}{N} \sqrt{\sum_{i=1}^N{\rm Var}(G_i)}$. To calculate ${\rm Var}(G_i)$, let $\left(\widetilde{Z}_i(t)\right)_{0\leq t \leq T}$ be an independent copy of $\left(\overline{Z}_i(t)\right)_{0\leq t \leq T}$ and set $\widetilde{G}_i:=g\left(\widetilde{Z}_i,x_i\right)$, then denoting by $\widetilde{E}$ the expectation taken with respect to $\widetilde{G}_i$, we have
$${\rm Var}(G_i)=\mathbf{E}\left[ \left( G_i - \mathbf{E}\left[G_i\right]\right)^2\right] = \mathbf{E} \left[\widetilde{E}\left[ G_i - \widetilde{G}_i \right]^2 \right] \leq \mathbf{E} \left[ \widetilde{E}\left[ \left(G_i - \widetilde{G}_i \right)^2\right] \right] $$
by Jensen's inequality. We have, as $\Vert g \Vert_L\leq 1$:
$$\widetilde{E}\left[ \left(G_i - \widetilde{G}_i \right)^2\right] \leq \widetilde{E} \left[ d_0\left( \overline{Z}_i,\widetilde{Z}_i\right) ^2\right]  \leq \widetilde{E}\left[ \left( \sup_{0\leq t \leq T} \left| \overline{Z}_i(t)- \widetilde{Z}_i(t) \right| \right)^2\right]\leq 2 \overline{Z}_i(T)^2 + 2\widetilde{E}\left[ \widetilde{Z}_i(T) ^2\right] $$
as the processes are increasing. Thus we obtain ${\rm Var}(G_i) \leq 4 \mathbf{E} \left[ \overline{Z}_i (T) ^2 \right]$.
As $Z_i(T)$ is a Poisson random variable with rate $\int_0^T \lambda(t,x_i)dt$,
$$\mathbf{E}\left[\widetilde{Z}_i(T) ^2\right] = {\rm Var}\left(\widetilde{Z}_i(T)\right) + \left( \mathbf{E}\left[\widetilde{Z}_i(T)\right]\right)^2 = \int_0^T \lambda(t,x_i)dt+ \left( \int_0^T \lambda(t,x_i)dt\right)^2$$
which is finite as $\lambda$ is bounded (Theorem \ref{thm:existence_lambda}). We have then shown that $B\xrightarrow[N\to \infty]{} 0$.
To treat $C$, note that it can be rewritten as
$$C=\left| \int \int \phi(\eta,x) P_{[0,T],\infty}(d\eta\vert x) \left(\nu^{(N)}(dx)-\nu(dx)\right)\right|.$$
We denote by $h$ the bounded function $h(x)=\int \phi(\eta,x)P_{[0,T],\infty}(d\eta\vert x)$. Under Scenario (1), $C=\left|  \dfrac{1}{N}\sum_{i=1}^N h(x_i) - \int_I h(x) \nu(dx)\right|\xrightarrow[N\to\infty]{} 0$ by the Law of Large Numbers. Under Scenario (2), we recognise a Riemann sum with $C=\left| \dfrac{1}{N} \sum_{i=1}^N h\left(\frac{i}{N}\right) - \int_0^1 h(x)dx \right|$: it suffices to show that $h$ is continuous to have $C \xrightarrow[N\to\infty]{} 0$. Fix $x$ in $I$ and consider a sequence $\left(x_n\right)$ such that $x_n \xrightarrow[n\to\infty]{} x$. We have
$$\vert h(x_n) -  h(x) \vert \leq \int \left| \phi(\eta,x_n) - \phi(\eta,x) \right| P_{[0,T],\infty}(d\eta\vert x_n)+  \left| \int \phi(\eta,x) \left(  P_{[0,T],\infty}(d\eta\vert x) -  P_{[0,T],\infty}(d\eta\vert x_n) \right) \right|.$$
We deal with the first term: by the Lipschitz continuity of $\phi$ and the fact that $P_{[0,T],\infty}(\cdot\vert x_n)$ is a probability measure, we have $\int \left| \phi(\eta,x_n) - \phi(\eta,x) \right| P_{[0,T],\infty}(d\eta\vert x_n) \leq  \Vert x-x_n \Vert \xrightarrow[n\to\infty]{} 0$. As $x$ is fixed, to have the second term $ \left| \int \phi(\eta,x) \left(  P_{[0,T],\infty}(d\eta\vert x) -  P_{[0,T],\infty}(d\eta\vert x_n) \right) \right| \xrightarrow[n\to\infty]{} 0,$ we show that for any function $\psi$ with Lipschitz constant $\Vert \psi \Vert_L \leq 1$ defined on $\mathbb{D}\left([0,T],\mathbb{N}\right)$, the function $\rho(y):= \int \psi(\eta) P_{[0,T],\infty}(d\eta\vert y)$ is continuous on $I$: let $\pi$ be a random Poisson measure with intensity $dsdz$ on $\mathbb{R}_+\times \mathbb{R}_+$, and for each $y \in I$ construct a Poisson point process $\overline{Z}^y$ on $[0,T]$ with intensity $\lambda(\cdot,y)$ by taking $\overline{Z}^y(t)= \int_0^t\int_0^\infty \mathbf{1}_{z\leq \lambda(s,y)}\pi(ds,dz)$. Then, as $\psi$ is Lipschitz continuous,
\begin{align*}
\vert \rho (x) - \rho(x_n) \vert &= \left| \mathbf{E} \left[ \psi\left( \overline{Z}^{x_n}\right) - \psi\left( \overline{Z}^{x}\right) \right] \right| \leq \mathbf{E} \left[ d_0 \left( \overline{Z}^{x_n}, \overline{Z}^{x} \right)\right]\\
&\leq  \mathbf{E} \left[ \sup_{0\leq t \leq T}\left| \overline{Z}^{x_n}(t) - \overline{Z}^{x}(t) \right|\right] \leq \mathbf{E} \left[ \int_0^t \left| d\left(  \overline{Z}^{x_n}(s) - \overline{Z}^{x}(s) \right)\right| \right]\\
&\leq \mathbf{E}\left[ \int_0^t \int_0^\infty \left| \mathbf{1}_{z\leq \lambda(s,x_n)} - \mathbf{1}_{z\leq\lambda(s,x)}\right| \pi(ds,dz) \right]\\
&\leq \int_0^t \left| \lambda(s,x_n) - \lambda(s,x) \right| ds \leq T \Vert x-x_n \Vert^\vartheta \xrightarrow[n\to\infty]{}0,
\end{align*}
with \eqref{eq:pt_fixe_lips_espace}. Then $\rho$ is indeed continuous on $I$, and so is $h$ hence  $C=\left| \frac{1}{N} \sum_{i=1}^N h\left(\frac{i}{N}\right) - \int_0^1 h(x)dx \right| \xrightarrow[N\to\infty]{} 0$.
We have shown that for any function $\phi$ on $S$ such that $\Vert \phi \Vert_{BL}\leq 1$, we have 
\begin{equation}\label{eq:thm_cvf_dBL_gfixe}
\mathbf{E}\left| \int \phi\left( d\mu_N -d\mu_\infty\right)\right| \xrightarrow[N\to\infty]{} 0.
\end{equation}
\medskip 

\textit{Step 2 - Approximation of any $\phi$ by a finite set of functions and conclusion.}
To derive \eqref{eq:thm_cvg_dBL}, we use an argument from Lemma 4.5 of \cite{Luon2020} and Theorem 11.3.3 of \cite{Dudley2002}.
For all $\varepsilon>0$, there exists a compact set $K\subset S$ with $\mu_\infty(K)>1-\varepsilon$. The set of functions $B:=\left\{\phi_{|K}, \Vert \phi \Vert_{BL}\leq 1\right\}$, restricted to $K$ is a compact set by Arzela-Ascoli Theorem, hence there exists $k\geq 1$ and $k$ functions in $B$ $\phi_1,\cdots	, \phi_k$ such that for any $\phi$ satisfying $\Vert \phi \Vert_{BL}\leq 1$, there exists $j\leq k$ that verifies $\sup_{y\in K} \left| \phi(y) - \phi_j(y)\right| \leq \varepsilon$. We denote by $K^\varepsilon:=\left\{ z\in S, d_S\left( z,K\right)< \varepsilon \right\}$. Then $\sup_{z\in K^\varepsilon} \left| \phi(z) - \phi_j(z)\right|<3\varepsilon$ as for any $z \in K^\varepsilon$, we can find $y_z \in K$ such that $d_S(z,y_z)<\varepsilon$ and
\begin{align*}
\left| \phi(z) - \phi_j(z)\right| &\leq \left| \phi(z) - \phi(y_z)\right| +\left| \phi(y_z) - \phi_j(y_z)\right| + \left| \phi_j(y_z) - \phi_j(z)\right|\\
&\leq \Vert \phi \Vert_L d_S(z,y_z) + \varepsilon + \Vert \phi_j \Vert_L d_S(z,y_z) \leq 3 \varepsilon.
\end{align*}
We introduce the function on $S$: $g(z)=\max \left(0, 1-\dfrac{d_S(z,K)}{\varepsilon}\right)$. Note that $\mathbf{1}_{K} \leq g \leq \mathbf{1}_{K^\varepsilon}$ and $g$ is bounded and Lipschitz continuous. Then, integrating on $\mu_N$, we obtain $\mu_N(K^\varepsilon) \geq \int g d\mu_N$.
We put together all the previous bounds to have, for any $\phi$ such that $\Vert \phi \Vert_{BL}\leq 1$:
\begin{align*}
\left| \int \phi \left(d\mu_\infty - d\mu_N\right)\right| &\leq \int \left| \phi - \phi_j \right|  \left(d\mu_\infty + d\mu_N\right) + \left| \int \phi_j  \left(d\mu_\infty - d\mu_N\right)\right|\\
&\leq \int_{K^\varepsilon} \left| \phi - \phi_j \right|  \left(d\mu_\infty + d\mu_N\right) + \int_{S-K^\varepsilon} \left| \phi - \phi_j \right|  \left(d\mu_\infty + d\mu_N\right) \\&\quad + \left| \int \phi_j  \left(d\mu_\infty - d\mu_N\right)\right|\\
&\leq 3\varepsilon . 2 + 2\mu_\infty\left(S-K^\varepsilon\right) + 2\mu_N\left(S-K^\varepsilon\right) +  \left| \int \phi_j  \left(d\mu_\infty - d\mu_N\right)\right|.
\end{align*}
Hence, taking the supremum on such function $\phi$  we obtain 
$$ \sup_{\phi, \Vert \phi \Vert _ {BL}\leq 1} \left| \int \phi \left(d\mu_N -  d\mu_\infty\right)\right| \leq 8\varepsilon + 2\left(1-\int g d\mu_N\right) + \max_{1\leq j \leq k} \left| \int \phi_j  \left(d\mu_\infty - d\mu_N\right)\right|.$$
Using \eqref{eq:thm_cvf_dBL_gfixe}, for $N$ large enough $\mathbb{E}\left[\int gd\mu_N\right]>\int gd\mu_\infty - \varepsilon$ and as $\int gd\mu_\infty \geq \mu_\infty(K) \geq 1 -\varepsilon$, we have $\mathbb{E}\left[\int gd\mu_N\right]>1-2\varepsilon$ and then 
$\mathbf{E} \left[ d_{BL}\left(\mu_N,\mu_\infty\right)\right] \leq 12\varepsilon + \mathbf{E}\left[\max_{1\leq j \leq k} \left| \int \phi_j  \left(d\mu_\infty - d\mu_N\right)\right| \right]$
and using \eqref{eq:thm_cvf_dBL_gfixe}, $ \mathbf{E}\left[\max_{1\leq j \leq k} \left| \int \phi_j  \left(d\mu_\infty - d\mu_N\right)\right| \right] \xrightarrow[N\to\infty]{}0$ as there is a finite number of functions considered, which concludes the proof of \eqref{eq:thm_cvg_dBL}.\qed
%----------------------------------------------------------------------------------------%----------------------------------------------------------------------------------------
\subsection{Proof of Proposition \ref{prop:cvg_profil_spatial}} \label{S:proof_prop:cvg_profil_spatial}
We show the convergence of the spatial profile $U_N$, when the positions are regularly distributed on $[0,1]$ and $W$ is continuous.  We have
\begin{align*}
&\mathbf{E}\left[  \int_0^T  \int_0^1 \left| U_N(t,x) - u(t,x) \right| dx ~dt \right]  \leq \mathbf{E} \left[ \int_0^T \int_0^1 \left| \sum_{i=1}^N \mathbf{1}_{\left( \frac{i-1}{N},\frac{i}{N}\right]}(x) \left( U_{i,N}(t) - u(t,x) \right) \right| dx dt \right]\\
&\leq \mathbf{E}\left[ \int_0^T \dfrac{1}{N} \sum_{i=1}^N \left| U_{i,N}(t) - u(t,x_i) \right| dt \right] + \int_0^T \int_0^1 \left| \sum_{i=1}^N \mathbf{1}_{\left( \frac{i-1}{N},\frac{i}{N}\right]}(x) \left( u(t,x_i) - u(t,x) \right) \right| dx dt.
\end{align*}
The first term is dealt with the proof of Theorem \ref{thm:cvg_0}: recall \eqref{eq:delta_i_majoration}, we recognise
$$\int_0^T \mathbf{E} \left[ \left |  U_{i,N}(t) - u(t,x_i)  \right|\right]dt \leq \left( \sum_{k=1}^5 A^{(N)}_{i,T,k} \right),$$
and we have showed that $\displaystyle \dfrac{1}{N}\sum_{i=1}^N  A^{(N)}_{i,T,k} \xrightarrow[N\to \infty]{}0 ~\mathbb{P}$-almost surely for each $k=1,\ldots,5$. We then have $\mathbf{E}\left[ \int_0^T \dfrac{1}{N} \sum_{i=1}^N \left| U_{i,N}(t) - u(t,x_i) \right| dt \right]   \xrightarrow[N\to \infty]{}0 ~\mathbb{P}$-almost surely.
The other term is treated easily: as $u$ is continuous on the compact set $[0,T]\times [0,1]$ it is uniformly continuous. Fix $\varepsilon > 0$, then there exists $\eta>0$ such that if $\Vert t - t'\Vert + \Vert x-x' \Vert \leq \eta$, $\left| u(t,x)-u(t',x')\right| \leq \dfrac{\varepsilon}{T}$. We have then for $N$ large enough (such that $\dfrac{1}{N} \leq \eta$):
\begin{align*}
\int_0^T \int_0^1 \left| \sum_{i=1}^N \mathbf{1}_{\left( \frac{i-1}{N},\frac{i}{N}\right]}(x) \left( u(t,x_i) - u(t,x) \right) \right| dx dt &= \int_0^T \sum_{i=1}^N \int_{\frac{i-1}{N}}^\frac{i}{N} \left|  u(t,x_i) - u(t,x) \right| dx dt\\& \leq \int_0^T \dfrac{\varepsilon}{T} dt  = \varepsilon,
\end{align*}
which concludes the proof.\qed
%----------------------------------------------------------------------------------------
%----------------------------------------------------------------------------------------
\section{Proofs: Behavior in large time limit - Linear case}
%----------------------------------------------------------------------------------------
%----------------------------------------------------------------------------------------
\subsection{Proof of Theorem \ref{thm:lambda_temps_long_sous_critique}} \label{S:proof_thm:lambda_temps_long_sous_critique}

We show that in the subcritical case, $\lambda(\cdot,x)$ has a large time limit given by \eqref{eq:def_l_lim}. Assumption \eqref{eq:CS_sous_critical} implies the existence of some $n_0$ such that $\Vert h \Vert_1^{n_0}  \Vert T_W^{n_0} \Vert <1$.
\medskip

\textit{Step 1 -}
We show existence and uniqueness of $\ell$ by applying Banach fixed-point Theorem. We consider the map defined on $\mathcal{C}_b\left(I,\mathbb{R}\right)$ (the set of bounded continuous functions defined on $I$):
\begin{align*}
&F :  g \longmapsto F(g) \text{ such that for all } x\in I,\\
&F(g)(x)=u(x)+\Vert h \Vert_1 \int_{I} W(x,y)g(y)\nu(dy).
\end{align*}
As $u$ is bounded on $I$, $F(g)$ is bounded for any $g\in\mathcal{C}_b\left(I,\mathbb{R}\right)$ by $\Vert u \Vert_\infty + \Vert h \Vert_1 \Vert g \Vert_\infty C_W^{(1)}<\infty	$.
We check now that for any $g$, $F(g)$ is continuous. Let $(x,z) \in I \times I$. We have as $u$ is Lipschitz continous and using \eqref{eq:hyp_W_pseudolip_theta},  for any $g\in\mathcal{C}_b\left(I,\mathbb{R}\right)$:
\begin{align*}
\vert F(g)(x) - F(g)(z)\vert &\leq\left| u(x)-u(z) \right| + \Vert h \Vert_1 \left| \int_{I} \left(W(x,y)-W(z,y)\right)g(y)\nu(dy)\right|\\
&\leq \Vert u \Vert_L \Vert x-z \Vert + \Vert h \Vert_1 \Vert g \Vert_\infty  \Vert x-z \Vert^\vartheta.
\end{align*} 
We have then shown the existence of a constant $C_g$ independent of the choice of $(x,z)$ such that $\vert F(g)(x) - F(g)(z)\vert  \leq C_g\phi \left(\Vert x-z \Vert\right) . $ Hence, $\mathcal{C}_b\left(I,\mathbb{R}\right)$ is stable by $F$. 

We are going to prove that $F$ admits an \textbf{unique fixed point}, which is $\ell$ satisfying \eqref{eq:def_l_lim}. To do it, we show that some iteration of $F$ is contractive, and then the Banach fixed-point Theorem gives the result. Let $g$ and $\tilde{g}$ be two functions in $\mathcal{C}_b\left(I,\mathbb{R}\right)$. As $Fg = u + \Vert h \Vert_1 T_Wg$, we have immediately that $F^{n_0} g = \sum_{k=0}^{n_0-1} \Vert h \Vert_1^k T_W^k u + \Vert h \Vert_1^{n_0} T_W^{n_0} g$. Then
$ \Vert F^{n_0} g - F^{n_0} \tilde{g}\Vert =   \Vert h \Vert_1^{n_0} T_W^{n_0} (g-\tilde{g}) \leq  \Vert h \Vert_1^{n_0}  \Vert T_W^{n_0} \Vert \Vert g-\tilde{g} \Vert_\infty.$
As $n_0$ is chosen such that $\Vert h \Vert_1^{n_0}  \Vert T_W^{n_0} \Vert <1$, $F^{n_0}$ is contractive, thus has an unique fixed point which is also the unique fixed point of $F$ in $\mathcal{C}_b\left(I,\mathbb{R}\right)$ that we call $\ell$, solution to \eqref{eq:continuite_ell}. Note that such a $\ell$ is necessarily nonnegative, as the iterative map $F$ preserves positivity.
\medskip

\textit{Step 2 -} Let us show that under the present hypotheses, $\sup_{t\geq 0} \sup_{x\in I} \left| \lambda(t,x)\right|<\infty$. As $\lambda(t,x) = u_0(t,x) + h \ast \left( T_W \lambda\right) (t,x)$, $T_W\lambda(t,x) = T_W u_0(t,x) + h \ast T^2_W \lambda(t,x)$ and the iteration gives $\lambda(t,x) = \left( \sum_{k=0}^{n_0-1} h^{\ast k} \ast T_W^k u_0 \right) (t,x) + h^{\ast n_0} \ast T_W^{n_0} \lambda(t,x)$ for any $(t,x)\in \mathbb{R}_+\times I$, hence $\Vert \lambda(t,\cdot) \Vert_\infty \leq C(u_0,h,W) + \Vert h \Vert_1^{n_0} \Vert T_W ^{n_0} \Vert \Vert \lambda(t,\cdot) \Vert_\infty$ with $C(u_0,h,W)$ a positive constant. As we are in the subcritical case, it gives then $\sup_{t\geq 0} \Vert \lambda(t,\cdot) \Vert_{\infty}\leq \dfrac{C(u_0,h,W)}{1-\Vert h \Vert_1^{n_0} \Vert T_W^{n_0} \Vert} <\infty $. As $\lambda$ is then continuous and bounded on $\mathbb{R}_+\times I$, we can define its (temporal) Laplace transform: for any $x\in I$ and $z>0$, let
\begin{equation}\label{eq:def_laplace_lambda}
\Lambda(z,x):=\int_0^\infty e^{-tz} \lambda(t,x)dt.
\end{equation}
Let us study $ z\Lambda(z,x)$. We have, for any $x\in I $ and $z>0$, $$ z\Lambda(z,x) = \int_0^\infty ze^{-tz} \lambda(t,x)dt= \lambda(0,x) + \int_0^\infty e^{-tz} \dfrac{\partial \lambda}{\partial t} (t,x) dt.$$  Suppose that we are able to show that $\displaystyle I(x):= \int_0^\infty \left| \dfrac{\partial \lambda}{\partial t}(t,x) \right| <\infty$ for some $x$. Then, by dominated convergence theorem, $\displaystyle \int_0^\infty e^{-tz} \dfrac{\partial \lambda}{\partial t} (t,x) dt$ converges as $z\to 0$ to the finite limit $\displaystyle \int_0^\infty \dfrac{\partial \lambda}{\partial t} (t,x) dt$. This implies in particular that $\lambda(t,x)$ has a finite limit as $t\to\infty$, and we have in this case $\displaystyle \lim_{z\to 0} z\Lambda(z,x) = \lim_{t\to\infty} \lambda(t,x)$.
We have, by integrating by parts
\begin{align*}
\dfrac{\partial \lambda}{\partial t} (t,x)&= \dfrac{\partial u_0}{\partial t}(t,x)+	h(0)\int_I W(x,y) \lambda(t,y)\nu(dy) + \int_0^t \int_I W(x,y)h'(t-s)\lambda(s,y)\nu(dy)~ds\\
&= \dfrac{\partial u_0}{\partial t}(t,x)+ \int_I W(x,y) h(t) \lambda(0,y)\nu(dy) + \int_0^t \int_I W(x,y)h(t-s)\dfrac{\partial \lambda}{ \partial s}(s,y) \nu(dy)~ds,
\end{align*}
where we used Theorem \ref{thm:existence_lambda} for the regularity of $\dfrac{\partial \lambda}{ \partial s}$. We also know from Theorem \ref{thm:existence_lambda} that $(t,x) \to \dfrac{\partial \lambda}{ \partial t}(t,x)$ is bounded of $[0,T]\times I$ for any $T>0$, which implies that for any $A>0$, $\displaystyle \sup_{x\in I} \int_0^A \left| \dfrac{\partial \lambda}{ \partial t}(t,x) \right| dt <\infty$. Integrating on $[0,A]$, we have using \eqref{eq:borne_partial_u0}
\begin{align*}
\int_0^A\left| \dfrac{\partial \lambda}{\partial t} (t,x) \right| dt &\leq C_{u_0} + \Vert h \Vert_1  \Vert \lambda \Vert_\infty D(x) + \int_0^A \int_0^t \int_I W(x,y)  h(t-s) \left| \dfrac{\partial \lambda}{\partial s} (s,y) \right| \nu(dy)dsdt.
\end{align*}
Yet with a change in the bounds of the integrals ($0\leq s \leq t \leq A$)
\begin{align*}
\int_0^A \int_0^t \int_I W(x,y) h(t-s)  \left| \dfrac{\partial \lambda}{\partial s} (s,y) \right| \nu(dy) ~ ds~ dt =& \int_I W(x,y) \int_0^A \int_0^t  h(t-s)  \left| \dfrac{\partial \lambda}{\partial s} (s,y) \right| ds~ dt~ \nu(dy)\\
=& \int_I W(x,y) \int_0^A \left( \int_s^A h(t-s) dt \right)  \left| \dfrac{\partial \lambda}{\partial s} (s,y) \right| ds ~ \nu(dy)\\
\leq & ~ \Vert h \Vert_1 \int_I W(x,y) \int_0^A   \left| \dfrac{\partial \lambda}{\partial s} (s,y) \right| ds~ \nu(dy).
\end{align*}
Setting $\displaystyle I_A(x):= \int_0^A  \left| \dfrac{\partial \lambda}{\partial s} (s,x) \right|ds$ and $C=C_{u_0} + \Vert h \Vert_1  \Vert \lambda \Vert_\infty C_W^{(1)}$, we have shown that $\displaystyle I_A(x) \leq C + \Vert h \Vert_1 \left( T_W I_A \right)(x)$ and by iteration $$I_A(x) \leq  \sum_{k=0}^{n_0-1} C \Vert h \Vert_1 ^k {C_W^{(1)}}^k + \Vert h \Vert _1 ^{n_0} T_W^{n_0}I_A(x) \leq  \sum_{k=0}^{n_0-1} C \Vert h \Vert_1 ^k {C_W^{(1)}}^k + \Vert h \Vert _1 ^{n_0} \Vert T_W^{n_0}\Vert \Vert I_A \Vert_\infty,  $$ and then  $ \Vert I_A \Vert_\infty \leq \dfrac{C(h,u_0,W)}{1-\Vert h \Vert _1 ^{n_0} \Vert T_W^{n_0}\Vert }=C'$, with $C'$ a positive constant independent of $A$. We can then let $A\to\infty$ to obtain $\sup_{x\in I} I(x) < \infty$. Hence by dominated convergence $\lim_{z\to 0} z\Lambda(z,x)$ exists for any $x\in I$ and is equal to $\lim_{t\to\infty}\lambda(t,x)=:\ell(x)$ that can now be defined. 
Coming back to the definition on $\Lambda$, we do the same for $u_0$ and define for any $x\in I$ and $z> 0$ $U(z,x):= \int_0^\infty e^{-tz} u_0(t,x)dt$. As $u_0(t,x)\xrightarrow[t\to\infty]{}u(x)$, note that $\lim_{z\to 0}zU(z,x)=u(x)$. As $h$ is integrable in this framework, we can also define its Laplace transform for any $z\geq 0$ by $H(z):= \int_0^\infty e^{-tz} h(t)dt$, with $H(0)=\Vert h \Vert_1$. Using the fact that the Laplace transform of a convolution is the product of the Laplace transforms, we have for any $x\in I$  and $z>0$
\begin{align}\label{eq:laplace_pointfixe_z}
z\Lambda(z,x)% &= \int_0^\infty ze^{-tz} \left( u_0(t,x) + \int_I W(x,y) \int_0^th(t-s)\lambda(s,y)ds~\nu(dy) \right) dt \notag\\
&= zU(z,x) + H(z)  \int_I W(x,y) z \Lambda(z,y) \nu(dy).
\end{align}
Letting $z\to 0$ in \eqref{eq:laplace_pointfixe_z}, we obtain that $\ell$ is solution of the equation \eqref{eq:def_l_lim}.\qed
%----------------------------------------------------------------------------------------%----------------------------------------------------------------------------------------
\subsection{Proof of Propositions \ref{prop:spectral_TW_L2} and \ref{prop:lambda_temps_long_sur_critique}} \label{S:proof_prop:lambda_temps_long_sur_critique}\

\textit{Proof of Proposition \ref{prop:spectral_TW_L2}} .
The boundedness of $T_{ W}$ on $L^{ 2}(I)$ follows from \eqref{eq:W_deg2} and Cauchy-Schwarz inequality: for $g\in L^{ 2}(I)$
\begin{align*}
\left\Vert T_{ W}g \right\Vert_{ 2}^{ 2} &= \int \left(\int W(x,y) g(y) \nu({\rm d}y)\right)^{ 2} \nu({\rm d}x)\\
& \leq \int \left( \int W(x,y)^{ 2} \nu({\rm d}y)\right) \left(\int g(y)^{ 2} \nu({\rm d}y)\right) \nu({\rm d}x) \leq \left\Vert g \right\Vert_{ 2}^{ 2} C_W^{(2)}.
\end{align*}
It is standard to see that $T_{ W}$ is compact on $L^{ 2}(I)$ and selfadjoint, by \eqref{eq:W_sym}, so that the same result holds readily for $T_{ W}^{ p}$ for all $p\geq1$. The fact that the spectrum of $T_{ W}^{ p}$ is made of a countable set of eigenvalues with no other accumulation points than $0$ is a mere application of the spectral theorem for compact operators. Let us now prove \eqref{eq:spectral_radii_equal}: first note that it suffices to prove that $r_{ 2}(T_{ W}^{ 2})= r_{ \infty}(T_{ W}^{ 2})$. Indeed, for any continuous operator $T$ with spectral radius $r(T)$, for all $p\geq1$, $ r(T^{ p})^{ \frac{ 1}{ p}}= \left( \lim_{ n\to\infty} \left\Vert T^{ pn} \right\Vert^{ \frac{ 1}{ n}}\right)^{ \frac{ 1}{ p}}= \lim_{ n\to\infty} \left\Vert T^{ pn} \right\Vert^{ \frac{ 1}{ pn}}= r(T)$, so that $r(T^{ p})= r(T)^{ p}$. Hence $r_{ 2}(T_{ W}^{ 2})= r_{ \infty}(T_{ W}^{ 2})$ gives $r_{ 2}(T_{ W})= r_{ \infty}(T_{ W})$ and \eqref{eq:spectral_radii_equal} follows. We prove that $r_{ 2}(T_{ W}^{ 2})= r_{ \infty}(T_{ W}^{ 2})$ by proving that they have the same spectrum. To do so, first note that $T_{ W}^{ 2}: L^{ \infty}(I) \to L^{ \infty}(I)$ is compact: consider $\left(f_n\right)_n$ a bounded sequence of $L^\infty(I)$. It is then also bounded in $L^2(I)$, and as $T_W:L^2(I)\to L^2(I)$ is compact, there exists a subsequence $\left(f_{\phi(n)}\right)$ such that $T_Wf_{\phi(n)}$ converges in $L^2(I)$ to a certain $g$. Then for any $x\in I$,
$$ \vert T_W^2 f_{\phi(n)} - T_Wg \vert (x) \leq \int_I W(x,y) \left| T_W f_{\phi(n)}(y) - g(y) \right| dy \leq \sqrt{C_W^{(2)}} \Vert T_Wf_{\phi(n)}-g\Vert_2 \xrightarrow[n\to\infty]{} 0,$$ thus $T_W^2:L^\infty(I)\to L^\infty(I)$ is compact. Hence, if one denotes by $ \sigma_{ \infty}(T_{ W}^{ 2})$ and $ \sigma_{ 2}(T_{ W}^{ 2})$ the corresponding spectrum of $T_{ W}^{ 2}$ (in $L^{ \infty}(I)$ and $L^{ 2}(I)$ respectively), we have that each nonzero element of $ \sigma_{ \infty}(T_{ W}^{ 2})$ and $ \sigma_{ 2}(T_{ W}^{ 2})$ is an eigenvalue of $T_{ W}^{ 2}$: let $\mu \in \sigma_2(T_W^2)\setminus\{0\}$, there exists $g\in L^2(I)$ such that $\mu g = T_W^2g$. As $$\left| T_W^2g(x) \right| = \left| \int_I W(x,y) \int_I W(y,z) g(z) ~ \nu(dz)\nu(dy)\right| \leq C_W^{(1)}\sqrt{C_W^{(2)}} \Vert g \Vert_2 <\infty,$$ $g = \frac{1}{\mu}T_W^2g \in L^\infty(I)$ and $\mu \in \sigma_\infty(T_W^2)$. Conversely, let  $\mu \in \sigma_\infty(T_W^2)\setminus\{0\}$, there exists $g \in L^\infty(I)$ such that $\mu g = T_W^2g$. As $L^\infty(I)\subset L^2(I)$, $\mu \in \sigma_2(T_W^2)$. Hence $r_{ 2}(T_{ W}^{ 2})= r_{ \infty}(T_{ W}^{ 2})$ and \eqref{eq:spectral_radii_equal} follows.

Let us now prove the second part of Proposition~\ref{prop:spectral_TW_L2}: this is essentially a reformulation of the Jentzsh-Krein-Rutman Theorem (see Theorem \ref{thm:schaefer6.6}): under assumption \eqref{eq:W_sym}, the spectral radius $r_{ 2}(T_{ W}^{ k})$ is an eigenvalue of $T_{ W}^{ k}$ with a unique normalized eigenfunction $h_{ 0}$ such that $h_{ 0}>0$, $ \nu$ a.e. on $I$ and every other eigenvalue $ \mu$ of $T_{ W}^{ k}$ has modulus $ \left\vert \mu \right\vert< r_{ 2}(T_{ W}^{ k})$. It remains to prove that $h_{0}^{(k)}$ is in fact continuous and bounded.  As $\Vert T_W^{ k}h_0^{ (k)}\Vert_\infty \leq  \left(C_{ W}^{ (1)}\right)^{ k-1}\sqrt{C_W^{(2)}} \Vert h_0^{ (k)} \Vert_2$ (using Cauchy-Schwarz inequality) and $h_0=\frac{1}{r(T_W^{ k})}T_{ W}^{ k}h_0^{ (k)}$, $h_0^{ (k)}$ is bounded. Condition \eqref{eq:hyp_W_pseudolip_theta} implies that $T_Wh_0^{(k)}$ is continuous on $I$, hence $h_0^{(k)}$ is a positive continuous function on $I$.\qed
\medskip

\textit{Proof of Proposition \ref{prop:lambda_temps_long_sur_critique}.}
We show that in the supercritical case, $\int_I \lambda(t,x)^2 \nu(dx) \xrightarrow[t\to\infty]{} \infty$. 

Consider for the moment the case $k=1$ (see \eqref{eq:Wkpos}). One benefit of working in $L^{ 2}(I)$ instead of $L^{ \infty}(I)$ is to take advantage of the Hilbert structure associated to $T_{ W}$: we know from the spectral theorem, that we can complete $h_0$ in an Hilbert orthonormal basis $\left( h_0, h_1, \cdots\right)$ of eigenvectors in $L^2(I)$ associated to the eigenvalues $(\mu_0=r_\infty, \mu_1, \mu_2, \cdots)$ with $\sup_{k\geq 1} \vert \mu_k \vert =: \widetilde{r}(T_W) < r_\infty$. We denote by $P_0$ the projection on $\text{Vect}(h_0)$ and $P_1=Id-P_0$: for any $g\in L^2(I)$, $P_0 g = \left\langle g\, ,\, h_{ 0}\right\rangle h_0=:p_0(g) h_0$ (with $p_0(g) \in \mathbb{R}$) and $P_1 g = \sum_{n\geq 1}  \left\langle g\, ,\, h_n\right\rangle h_n$. The strategy of proof of Proposition \ref{prop:lambda_temps_long_sur_critique} is then to analyse separately the dynamics of $P_{ 0} \lambda$ and $P_{ 1} \lambda$ for $ \lambda$ solution to \eqref{eq:def_lambdabarre_linear}. Concerning $P_{ 0} \lambda$, as $P_{ 0}$ projects onto $h_{ 0}$, eigenfunction associated to the dominant eigenvalue $r_{ 2}(T_{ W})$, its analysis reduces to a simple one-dimensional linear convolution equation, whose behavior in large time has been analysed in details (see \cite[Lemma 26]{delattre2016} or \cite[Th 4]{Feller1941}). The second step is to show that the contribution of $P_{ 1} \lambda$ remains of lower order as $t\to\infty$.

We focus first on the dynamics of $P_0\lambda$.
Using \eqref{eq:def_lambdabarre_linear}, as $T$ and $P_0$ commute,
{\begin{align*}
p_0(\lambda(t,\cdot))h_0= P_0 \lambda(t,\cdot) &= P_0 u_0(t,\cdot) + \int_0^t h(t-s) T_W P_0 \lambda(s,\cdot) ds\\
%&= p_0(u_0(t,\cdot))h_0 + \int_0^t h(t-s) T_W p_0(\lambda(s,\cdot))h_0 ds\\
&=  \left(p_0(u_0(t,\cdot)) + r_\infty\int_0^t h(t-s) p_0(\lambda(s,\cdot))  ds\right)h_0.
% p_0(\lambda(t,\cdot))h_0&= p_0(u_0(t,\cdot))h_0 +  r_\infty \int_0^t h(t-s) p_0(\lambda(s,\cdot)) ds h_0.
\end{align*}
As $h_0(x)>0$ everywhere (since $h_0$ is continuous), we obtain that $ p_0(\lambda(t,\cdot))$ solves the convolution equation in $\mathbb{R}$
\begin{equation}\label{eq:convolution_proj0}
 p_0(\lambda(t,\cdot)) =  p_0(u_0(t,\cdot)) +  r_\infty \int_0^t h(t-s) p_0(\lambda(s,\cdot)) ds.
\end{equation}
Theorem \ref{thm:feller4} gives then that $ p_0(\lambda(t,\cdot)) \sim_{t\to\infty} Ce^{\sigma_r t}$  where $C>0$ depends on the parameter functions and $\sigma_r>0$ verifies $ r_\infty \int_0^\infty e^{-\sigma_rt} h(t) dt = 1$. We focus now on the other projection, $P_1\lambda$. We project on the rest of the space and take the norm $L^2(I)$:
\begin{align*}
%P_1 \lambda(t,\cdot) &= P_1 u_0(t,\cdot) + \int_0^t h(t-s) T_W P_1 \lambda(s,\cdot) ds\\
\Vert P_1 \lambda(t,\cdot) \Vert_2 &\leq \Vert P_1 u_0(t,\cdot)\Vert_2 + \int_0^t h(t-s) \Vert T_W P_1 \lambda(s,\cdot) \Vert_2 ds.
\end{align*}
As $\displaystyle T_W P_1 \lambda(s,\cdot) =T_W\left(   \sum_{n\geq 1}  \left\langle P_1 \lambda(s,\cdot)\, ,\, h_n\right\rangle h_n  \right) = \sum_{n\geq 1} \left\langle P_1 \lambda(s,\cdot)\, ,\, h_{ n}\right\rangle \mu_nh_n $, we have  $\displaystyle \Vert T_W P_1 \lambda(s,\cdot) \Vert_2^2 = \sum_{n\geq 1} \left| \left\langle P_1 \lambda(s,\cdot)\, ,\, h_n\right\rangle\right|^2 \vert \mu_n \vert^2 \leq \widetilde{r}(T_W)^2 \Vert  P_1 \lambda(s,\cdot)\Vert_2^2$ so that
\begin{equation}\label{eq:proj_p1_ineg}
\Vert P_1 \lambda(t,\cdot) \Vert_2 \leq \Vert P_1 u_0(t,\cdot)\Vert_2 + \widetilde{r}(T_W) \int_0^t h(t-s) \Vert P_1 \lambda(s,\cdot) \Vert_2 ds.
\end{equation}
If we define $\alpha(t)= \Vert P_1 \lambda(t,\cdot) \Vert_2$, we see that $\alpha$ satisfies the convolution inequality $$\alpha(t)\leq  \Vert P_1 u_0(t,\cdot)\Vert_2 + \widetilde{r}(T_W) \int_0^t h(t-s)\alpha(s) ds,$$
hence we can compare it to $\beta_{\widetilde{r}}(t)$ solution of the convolution equality
$$\beta(t)=  \Vert P_1 u_0(t,\cdot)\Vert_2 +1+ \widetilde{r}(T_W) \int_0^t h(t-s)\beta(s) ds$$with Lemma \ref{lem:ineg_sol_conv}: $\alpha(t)\leq \beta(t)$ for all $t\geq 0$, that is

 $\Vert P_1 \lambda(t,\cdot) \Vert_2 \leq \beta_{\widetilde{r}}(t)$ for all $t\geq 0$.  We want now to show that $\beta_{\widetilde{r}}(t) = o\left( e^{\sigma_r t} \right)$ when $t\to\infty$. 
First suppose that we are in the case $\Vert h \Vert_1 \widetilde{r}(T_W) >1$. We apply (as done for $P_0$) Theorem \ref{thm:feller4} and obtain $ \beta_{\widetilde{r}}(t) \sim_{t\to\infty} \widetilde{C}e^{\sigma_{\widetilde{r}} t}$ where $\widetilde{C}>0$ depends on the parameter functions and $\sigma_{\widetilde{r}}>0$ verifies $ \widetilde{r}(T_W) \int_0^\infty e^{-\sigma_{\widetilde{r}} t} h(t) dt = 1$. In this case, $\sigma_{\widetilde{r}} < \sigma_r$ as $ \widetilde{r} (T_W) < r_\infty$, and $\beta_{\widetilde{r}}(t) = o\left( e^{\sigma_r t} \right)$ follows. Suppose now that we are in the case $\Vert h \Vert_1 \widetilde{r}(T_W) \leq 1$. As  $\Vert h \Vert_1 r_\infty >1$, we can find $\overline{r}$ such that $\widetilde{r}(T_W)<\overline{r} < r_\infty$ and  $\Vert h \Vert_1 \overline{r}>1$. Then, considering $\delta$ satisfying $\delta_{\overline{r}}(t) = \Vert P_1 u_0(t,\cdot)\Vert_2 + 2 + \overline{r} \int_0^t h(t-s) \delta_{\overline{r}}(s) ds$, as done before Lemma \ref{lem:ineg_sol_conv} gives $\beta_{\widetilde{r}}(t)\leq \delta_{\overline{r}}(t)$ and Theorem \ref{thm:feller4} gives $\delta_{\overline{r}}(t) \sim_{t\to\infty} \overline{C}e^{\sigma_{\overline{r}} t}$ where $\overline{C}>0$ depends on the parameter functions and $\sigma_{\overline{r}}>0$ verifies $ \overline{r} \int_0^\infty e^{-\sigma_{\overline{r}} t} h(t) dt = 1$. We have then that $\beta_{\widetilde{r}}(t) \leq \delta_{\overline{r}} (t) \sim_{t\to\infty} \overline{C}e^{\sigma_{\overline{r}} t} = o\left( e^{\sigma_r t} \right)$.
In any case, we obtain $ \Vert P_1 \lambda(t,\cdot) \Vert_2 =o\left( e^{\sigma_rt}\right)$, and as Parseval equality gives $$ \Vert \lambda(t,\cdot) \Vert_2^2 = \Vert P_0 \lambda(t,\cdot) \Vert_2^2 + \Vert P_1 \lambda(t,\cdot) \Vert_2^2,$$
it implies that $\Vert \lambda(t,\cdot) \Vert_2 \sim_{t\to\infty} C  e^{\sigma_rt} \xrightarrow[t\to\infty]{}+\infty$, with $C$ a positive constant, whence the result.
\medskip

\textit{Case $k>1$.} We deal with $k=2$ and leave the generalisation to the reader. Hypothesis \ref{hyp:surcritique+} \eqref{eq:Wkpos} is then that the kernel of $T_W^2$ is positive. As $\lambda(t,x) = u_0(t,x) + \int_0^t h(t-s) T_W\lambda(s,\cdot)(x) ds$, we have $T_W\lambda(t,\cdot)(x) = T_Wu_0(t,\cdot)(x) + \int_0^t h(t-s) T_W^2 \lambda(s,\cdot)(x)ds$ and  
\begin{align}\label{eq:lambda_Wk=2}
\lambda(t,x) &= u_0(t,x) + \int_0^t h(t-s) T_Wu_0(s,\cdot)(x)ds + \int_0^t h(t-s) \int_0^s h(s-u) T_W^2 \lambda(u,\cdot)(x)duds \notag\\
%&= v_0(t,x) + \int_0^t h(t-s) \left(  h\ast T_W^2 \lambda(s,\cdot)(x) \right) ds\notag\\
&= v_0(t,x) + \int_0^t \tilde{h}(t-s) T_W^2 \lambda(s,\cdot)(x)ds.
\end{align} 
with $\tilde{h}=h\ast h$ and $v_0(t,x)=u_0(t,x) + \int_0^t h(t-s) T_Wu_0(s,\cdot)(x)ds$. As $\Vert \tilde{h} \Vert_1 = \int_0^\infty \int_0^t h(t-s)h(s)dsdt = \int_0^\infty h(s)\int_s^\infty h(t-s)dtds = \Vert h \Vert_1^2$ and $\sqrt{r(T_W^2)} = r(T_W)$, the condition \eqref{eq:CS_sur_critical} implies $\Vert \tilde{h}\Vert_1 r(T_W^2) =  \Vert h \Vert _1 ^2 r(T_W) ^2 >1$. Then, we can apply the previous case ($k=1$) on $\lambda$ satisfying \eqref{eq:lambda_Wk=2}. \qed

%----------------------------------------------------------------------------------------%----------------------------------------------------------------------------------------
\begin{appendix}
\newcounter{theorem}
 \setcounter{theorem}{0}
     \renewcommand{\thetheorem}{\Alph{section}\arabic{theorem}}
     \renewcommand{\thesection}{\Alph{section}}
%----------------------------------------------------------------------------------------
%----------------------------------------------------------------------------------------

\section{Useful results}

We remind here different results that we use frequently in this paper. The proof of the Lemmas \ref{lem:interversion_int} and \ref{lem:picard_gen} can be found respectively in Lemmas 22 and 23 of \cite{delattre2016}.
 
\begin{lem}\label{lem:interversion_int} 
Let $\phi:\left[0,\infty\right[ \longrightarrow \mathbb{R}$ be locally integrable and $\alpha:\left[0,\infty\right[ \longrightarrow \mathbb{R}$ with finite variations on compact intervals such that $\alpha(0)=0$. Then for all $t\geq 0$, we have 
$$ \int_0^t \int_0^{s-} \phi\left(s-u\right)d\alpha\left(u\right)ds = \int_0^t \int_0^{s} \phi\left(s-u\right)d\alpha\left(u\right)ds = \int_0^t  \phi\left(t-s\right)\alpha\left(s\right)ds.$$
\end{lem}

\begin{lem}\label{lem:maj_famille_indep}
Let $\left(X_i\right)_{1\leq i \leq N}$ be a family of $N$ independent random variables. Then
$$ \mathbf{E} \left[ \left| \dfrac{1}{N} \sum_{i=1}^N \left( X_i - \mathbf{E}\left[X_i\right]\right) \right| \right] \leq \dfrac{1}{N} \sqrt{\sum_{i=1}^N{\rm Var} \left(X_i\right)}.$$
\end{lem}
\begin{proof}
We set $Y:= \dfrac{1}{N} \sum_{i=1}^N X_i$, then $\mathbf{E}\left[Y\right] =  \dfrac{1}{N} \sum_{i=1}^N \mathbf{E}[ X_i]$ and ${\rm Var}(Y)=\dfrac{1}{N^2}  \sum_{i=1}^N {\rm Var}( X_i)$ by independence. 
We have $\mathbf{E} \left[ \left| \dfrac{1}{N} \sum_{i=1}^N \left( X_i - \mathbf{E}\left[X_i\right]\right) \right| \right]  = \mathbf{E} \left[ \left| Y - \mathbf{E}[Y] \right| \right] \leq \sqrt{\mathbf{E}\left[ \left( Y - \mathbf{E}[Y]  \right)^2 \right]} = \sqrt{{\rm Var}(Y) } $ using Jensen's inequality, and the result follows with the expression of ${\rm Var}(Y)$.
\end{proof}

\begin{lem}\label{lem:picard_gen}
Let $\phi:\left[0,\infty\right[ \longrightarrow \left[0,\infty\right[$ be a locally integrable function and $g:\left[0,\infty\right[ \longrightarrow \left[0,\infty\right[$ a locally bounded function.
\begin{enumerate}[label=(\roman*)]
\item Let $u$ be a locally bounded nonnegative function such that for all $t\geq 0$: $u(t) \leq g(t)+ \int_0^t \phi (t-s) u(s) ds.$
Then for all $T\geq 0$ there exists $C_T$ (depending on $T$ and $\phi$) verifying $\displaystyle\sup_{[0,T]} u(t) \leq C_T \sup_{[0,T]} g(t).$
\item Let $\left(u_n\right)$ be a sequence of locally bounded non-negative functions such that for all $t\geq 0$ and $n\geq 0$: $u_{n+1}(t) \leq \int_0^t \phi (t-s) u_n(s) ds.$
Then for all $T\geq 0$ there exists $C_T$ (depending on $T$, $\phi$ and $u_0$) verifying $\displaystyle\sup_{[0,T]} \sum_{n\geq 0} u_n(t) \leq C_T. $
\item Let $\left(u_n\right)$ be a sequence of locally bounded non-negative functions such that for all $t\geq 0$ and $n\geq 0$: $u_{n+1}(t) \leq g(t) + \int_0^t \phi (t-s) u_n(s) ds.$
Then for all $T\geq 0$ there exists $C_T$ (depending on $T$, $\phi$, $u_0$ and $g$) verifying $\displaystyle\sup_{[0,T]} ~\sup_{n\geq 0} ~ u_n(t) \leq C_T. $
\end{enumerate}
\end{lem}

 %----------------------------------------------------------------------------------------
%----------------------------------------------------------------------------------------

\begin{lem}\label{lem:ineg_sol_conv}
Let $r>0$, $h$ be a nonnegative locally integrable function, $u$, $\alpha$ and $\beta$ be locally bounded nonnegative continuous functions such that for all $t\geq 0$:
\begin{align*}
\alpha(t) &\leq u(t) + r \int_0^t h(t-s) \alpha(s) ds,\\
\beta(t) &= u(t)+1 + r \int_0^t h(t-s) \beta(s) ds.
\end{align*}
Then $\alpha(t)\leq \beta(t)$ for all $t\geq 0$.
\end{lem}

\begin{proof} Let $t^*= \inf \left\{ s>0,~\alpha(s) > \beta(s)\right\}$. 
Note that $t^*>0$ as $\beta(0)-\alpha(0) \geq 1$. Suppose $t^*<\infty$, then $\beta(t^*)-\alpha(t^*) \geq 1+\int_0^{t^*}h(t-s)\left(\beta(s) - \alpha(s) \right) ds \geq 1$ which is impossible, then necessarily $t^*=+\infty$ and $\alpha(t) \leq \beta(t)$ for all $t\geq 0$.
\end{proof}
 %----------------------------------------------------------------------------------------
 %----------------------------------------------------------------------------------------

\begin{lem}\label{lem:gronwall_convole}
Let $u$ and $h$ be locally square integrable functions, $u$ non-negative, $T>0$ and $\alpha, \beta$ two constants. Assume that for any $t \in[0,T]$, 
$u(t)\leq \alpha \int_0^t h(t-s) u(s) ds + \beta.$
Then $u$ satisfies the following Gr\"{o}nwall's inequality:
$u(T)\leq \sqrt{2} \beta \exp\left( \alpha^2 \Vert h \Vert_{T,2}^2 T\right)$.
\end{lem}

\begin{proof} Using Cauchy-Schwarz inequality, $u(t)^2 \leq 2 \alpha^2  \Vert h \Vert_{T,2}^2\int_0^t u(s)^2 ds+2\beta^2$. We conclude by applying standard Gr\"{o}nwall lemma to $u^2$ and taking the square root (since $u\geq 0$).
\end{proof}

\begin{lem}\label{lem:inegalit_concentration_Y}
Fix $N > 1$ and $\left(Y_l\right)_{l=1,\ldots,n}$ real valued random variables defined on a probability space $\left(\Omega, \mathcal{F}, \mathbb{P}\right)$. Suppose that there exists $\nu>0$ such that, almost surely, for all $l = 1,\ldots, n-1$, $Y_l\leq 1$, $\mathbb{E}\left[Y_{l+1} \left| Y_l \right.\right] = 0$  and $\mathbb{E}\left[Y_{l+1}^2 \left|Y_l\right.\right]\leq \nu$. Then 
$\mathbb{ P} \left(n^{ -1} (Y_{ 1}+ \ldots+ Y_{ n}) \geq x\right) \leq \exp \left( -n \frac{ x^{ 2}}{ 2v} B \left( \frac{ x}{ v}\right)\right)$ for all $x \geq 0$, where \begin{equation}\label{eq:def_B(u)}
B(u):= u^{-2}\left( \left( 1+u \right) \log \left( 1+u \right) - u \right).
\end{equation}
\end{lem}
\begin{proof}
A direct application of \cite[Corollary 2.4.7]{zeitouni1998large} gives that $$ \mathbb{ P}\left(n^{ -1} (Y_{ 1}+ \ldots+ Y_{ n}) \geq x\right) \leq \exp \left( -n H \left( \frac{ x+v}{ 1+v} \vert \frac{ v}{ 1+v}\right)\right),$$ where $H(p\vert q):= p \log(p/q) +(1-p) \log((1-p)/(1-q))$ for $p,q\in [0, 1]$. Then, the inequality $ H \left( \frac{ x+v}{ 1+v} \vert \frac{ v}{ 1+v}\right)\geq \frac{ x^{ 2}}{ 2v} B \left( \frac{ x}{ v}\right)$ (see \cite[Exercise 2.4.21]{zeitouni1998large}) gives the result.
\end{proof}

\begin{lem}\label{lem:inegalit_concentration}
Fix $N\geq 1$, $\left(p_1,\ldots,p_N\right)$ in $[0,1]$ and a sequence $\left(v_1,\ldots,v_N\right)$ such that $\vert v_l\vert \leq 1$ for any $l\in \llbracket 1,N \rrbracket$. Suppose that there exists $\kappa_N>0$ and $w_N\in]0,1]$ such that $p_l\leq w_N$ for any $l \in \llbracket 1,N \rrbracket$.
Then, setting $\varepsilon_n:= 32 \dfrac{\kappa_N^2w_N}{N}\log(N)$ for $\left(U_1,\ldots,U_N\right)$ independent random variables with $U_l \sim \mathcal{B}(p_l)$, we have
\begin{equation}\label{eq:inegalit_concentration}
\mathbb{P}\left( \left| \dfrac{\kappa_N}{N} \sum_{l=1}^N \left( U_l - p_l\right) v_l \right| > \varepsilon_N \right) \leq 2\exp\left( -16 \log(N) B\left( 4\sqrt{2}\left(\dfrac{\log(N)}{Nw_N}\right)^{\frac{1}{2}}\right)\right)
\end{equation}
with $B$ defined in \eqref{eq:def_B(u)}.
\end{lem}

\begin{proof}
This is a simple corollary of Lemma A.6 applied to $Y_{ l}:= (U_{ l}-p_{ l})v_{ l}$.
\end{proof}
%----------------------------------------------------------------------------------------
%----------------------------------------------------------------------------------------
\section{Useful results - spectral theory}
We include here advanced results of spectral analysis that are used in the paper.
\subsection{Jentzsch/Krein–Rutman Theorem}
The following theorem can be found in \cite{Schaefer1974} (Theorem 6.6) or in \cite{Zerner1987} (Theorem 1).
\begin{thm} \label{thm:schaefer6.6}
Let $E:=L^p(\mu)$, where $1\leq p \leq +\infty$ and $(X,\Sigma,\mu)$ is a $\sigma$-finite measure space. Suppose $T\in \mathcal{L}(E)$ is an operator given by a $(\Sigma\times \Sigma)$-measurable kernel $K\geq 0$, satisfying these two assumptions:
\begin{enumerate}[label=(\roman*)]
\item Some power of $T$ is compact.
\item $S\in \Sigma$ and $\mu(S)>0, \mu(X\setminus S)>0$ implies
$$\int_{X\setminus S}\int_S K(s,t)d\mu(s)d\mu(t)>0.$$
\end{enumerate}
Then $r(T)>0$ is an eigenvalue of $T$ with a unique normalized eigenfunction $f$ satisfying $f(s)>0$ $\mu$ a.e. Moreover, if $K(s,t)>0$ $\mu\otimes\mu$ a.e. then every other eigenvalue $\lambda$ of $T$ has modulus $\vert\lambda\vert < r(T)$.
\end{thm}
\subsection{Renewal theory}
The following theorem can be found in \cite{Feller1941} (Theorem 4). This article studies the behavior of solutions of the integral equation
\begin{equation}\label{eq:renewal}
u(t)=g(t)+\int_0^t u(t-x)f(x)dx,
\end{equation}
where $f$ and $g$ are measurable, non-negative and bounded in every finite interval $[0,T]$.
\begin{thm}\label{thm:feller4}
Suppose $\int_0^\infty f(t)dt >1$, $\int _0^\infty g(t)dt=b<\infty$. Suppose moreover that there exists an integer $n\geq 2$ such that the moments $m_k=\int_0^\infty t^kf(t)dt$, $k=1,2,\cdots,n$, are finite and that the functions $f(t)$, $tf(t)$, $t^2f(t)$, $\cdots$, $t^{n-2}f(t)$ are of bounded total variation over $(0,\infty)$. Suppose finally that 
$$\lim_{t\to\infty}t^{n-2}g(t)=0 \text{ and } \lim_{t\to\infty} t^{n-2}\int_t^\infty g(x)dx=0.$$
Then there is a unique $\alpha_0>0$ such that $\mathcal{L}_f(\alpha_0)=1$ and for some constant $C$ depending on $(mu, f,g)$:
$$u(t) \sim_{t\to\infty} Ce^{\alpha_0t}, \text{ where $u$ solves \eqref{eq:renewal}.
}$$
\end{thm}

\end{appendix}

\bibliographystyle{abbrv}

\end{document}